\documentclass[a4paper,12pt]{amsart}
\usepackage[colorlinks=true, linkcolor=blue, citecolor=blue]{hyperref}
\usepackage{amscd,amssymb,epsfig,amsxtra,amsfonts}
\usepackage{amsmath}
\usepackage{stmaryrd}
\usepackage{mathrsfs}
\usepackage{upgreek}
\usepackage{marginnote}
\usepackage{verbatim}
\usepackage[utf8]{inputenc}
\usepackage{diagrams}
\usepackage{epsfig}
\newcommand{\ms}[1]{\mbox{\tiny$#1$}}
\newcommand{\epsh}[2]
{\begin{array}{c} \hspace{-1.3mm}
   \raisebox{-4pt}{\epsfig{figure=#1,height=#2}}
   \hspace{-1.9mm}\end{array}}

\newcommand{\lc}{{{f}}}
\newcommand{\Fou}{{\mathcal{F}}}
\newcommand{\ith}{{\textsuperscript{th}} }
\newcommand{\coint}{\Upsilon}
\newcommand{\pr}{\operatorname{p}}
\newcommand{\vp}{\varphi}
\newcommand{\ve}{\varepsilon}

\newcommand{\lk}{\operatorname{lk}}

\newcommand{\rcoev}{\stackrel{\longrightarrow}{\operatorname{coev}}}
\newcommand{\rev}{\stackrel{\longrightarrow}{\operatorname{ev}}}
\newcommand{\lev}{\stackrel{\longleftarrow}{\operatorname{ev}}}
\newcommand{\lcoev}{\stackrel{\longleftarrow}{\operatorname{coev}}}
\newcommand{\e}{{\operatorname{e}}}
\newcommand{\Id}{\operatorname{Id}}

\newcommand{\End}{\operatorname{End}}
\newcommand{\Hom}{\operatorname{Hom}}
\newcommand{\Proj}{\operatorname{\mathsf{Proj}}}

\newcommand{\tr}{\operatorname{tr}}
\newcommand{\ptr}{\operatorname{ptr}}

\newcommand{\Vect}{\operatorname{Vect}}
\newcommand{\ResF}{\operatorname{ResF}}

\newcommand{\rk}{\ensuremath{r}}

\newcommand{\Span}{\operatorname{Span}}
\newcommand{\C}{\ensuremath{\mathbb{C}}}
\newcommand{\Z}{\ensuremath{\mathbb{Z}}}
\newcommand{\R}{\ensuremath{\mathbb{R}}}
\newcommand{\N}{\ensuremath{\mathbb{N}}}
\newcommand{\wt}{\widetilde}
\newcommand{\wh}{\widehat}
\newcommand{\wb}{\overline}
\newcommand{\ba}{{\bar\alpha}}
\newcommand{\va}{{\overrightarrow{\alpha}}}
\newcommand{\Lat}{\Lambda}
\newcommand{\bb}{{\bar\beta}}
\newcommand{\si}{{\mu}}

\newcommand{\bp}[1]{{\left(#1\right)}}

\newcommand{\qn}[1]{{\left\{#1\right\}}}
\newcommand{\qN}[1]{{\left[#1\right]}}
\newcommand{\ang}[1]{{\left\langle{#1}\right\rangle}}
\newcommand{\bN}[1]{\left\|#1\right\|}
\newcommand{\U}{\ensuremath{\mathcal{U}}}
\newcommand{\Uqg}{\ensuremath{\mathcal U}}
\newcommand{\HH}{{\mathcal H}}
\newcommand{\RR}{{\mathcal R}}
\newcommand{\UqgH}{\ensuremath{\Uqg^{H}}}

\newcommand{\UH}{\U^H}
\newcommand{\Uo}{\ensuremath{{\U^0}}}
\newcommand{\tru}{\ensuremath{{\tr_u}}}
\newcommand{\hotimes}{\widehat{\otimes}}
\newcommand{\Ce}{\ensuremath{\mathscr{C}^{\omega}}}
\newcommand{\h}{\ensuremath{\mathfrak{h}}}

\newcommand{\cat}{\operatorname{\mathscr{C}}}
\renewcommand{\H}{\ensuremath{\mathfrak{H}}}
\newcommand{\Q}{\ensuremath{\mathsf{Q}}}
\newcommand{\B}{\ensuremath{\mathsf{B}}}
\newcommand{\BW}{\ensuremath{\mathfrak{B}_W}}
\newcommand{\He}{{\mathrm{H}}}
\newcommand{\tv}{\operatorname{\mathsf t}}
\newcommand{\w}[1]{{\left|#1\right|}} 

\newtheorem{theo}{Theorem}[section]
\newtheorem{Lem}[theo]{Lemma}
\newtheorem{Prop}[theo]{Proposition}
\newtheorem{Cor}[theo]{Corollary}
\newcounter{axi} 

\newtheorem{axiom}[axi]{Axiom}
\theoremstyle{definition}
\newtheorem{Def}[theo]{Definition}
\newtheorem{rmq}[theo]{Remark}
\newtheorem{Exem}[theo]{Example}

\newcounter{exo} \newcounter{numexercice}
\renewcommand{\theexo}{\arabic{exo}}

\begin{document}
\title[Modified graded Hennings invariant]{Modified graded Hennings invariants from unrolled quantum groups and modified integral}

\author{Nathan Geer}
\address{Utah State University, Department of Mathematics and Statistics, Logan UT 84341, USA}
\email{nathan.geer@usu.edu}
\author{Ngoc Phu HA}        
\address{Hung Vuong University, Faculty of Natural Sciences, Viet Tri, Phu Tho, Viet Nam}
\email{ngocphu.ha@hvu.edu.vn}
\author{Bertrand Patureau-Mirand}
\address{Universit\'e Bretagne Sud, Laboratoire de Math\'ematiques de Bretagne Atlantique, UMR CNRS 6205, Campus de Tohannic, BP 573     
F-56017 Vannes, France}
\email{bertrand.patureau@univ-ubs.fr}

\maketitle

\begin{abstract}
The second author constructed a topological ribbon Hopf algebra from the unrolled quantum group associated with the super Lie algebra $\mathfrak{sl}(2|1)$. 
We generalize this fact to the context of unrolled quantum groups and construct the associated topological ribbon Hopf algebras. Then we use such an algebra, the discrete Fourier transforms, a symmetrized graded integral and a modified trace to define a modified graded Hennings invariant.
Finally, we use the notion of a modified integral to extend this
invariant to empty manifolds and show that it recovers the
CGP-invariant.
\end{abstract}

\vspace{25pt}

MSC:	57M27, 17B37.

Key words: Unrolled quantum group, topological ribbon Hopf algebra, Hennings type invariant, discrete Fourier transform, modified integral.

\section{Introduction}

It is known that the category of modules over the semi-restricted
quantum group at a root of unity
produces 3-manifold invariants (see \cite{FcNgBp14}).  In
\cite{Ha18a}, the second author used the semi-restricted quantum group
associated to the super Lie algebra $\mathfrak{sl}(2|1)$ (not its
category of modules) to define a Virelizier-Hennings type 3-manifold
invariant (see \cite{Henning96, Vire01}).
Here we generalize the latter
construction to a large setting of Hopf algebras.

For the simplest example of a semi-restricted quantum group see Example \ref{ex1} below.
By adding Cartan elements this quantum group has been extended to a
Hopf algebra called the unrolled quantum group, see for example
\cite{FcNgBp14, NgBp13, FrNaBe}.  In \cite{AS18}, Andruskiewitsch and
Schweigert consider unrolled Hopf algebras which are a generalization
of the unrolled quantum groups containing the case of Lie
(super)algebras and more general diagonal Nichols algebra (see
\cite{Heck10}). 
In this paper we propose a general construction which should produce
quantum 3-manifold invariants for all of these unrolled Hopf algebras.
To prove this one needs to show these algebras satisfy the five axioms
listed in this paper.

Let us now summarize these axioms.  Starting with a free abelian group
$\Lambda$ of rank $r$ we consider three algebras $\U$, $\UH$ and
$\U_{\wb 0}=\U/I_{\wb 0}$ which are generalizations of the
semi-restricted, unrolled and small quantum groups, respectively (for
the simplest example of these quantum groups see Example \ref{ex1} below).
Loosely speaking, the axioms we require are:
\begin{enumerate}
\item[Axiom] \ref{pivotal axiom}:
$\U$
is a $\Lambda$-graded pivotal Hopf algebra (graded by ``weights'').
\item[Axiom] \ref{quasi axiom}: There exists a quasi R-matrix
for $\UH$.
\item[Axiom] \ref{ax:unicorn}:
$\U_{\wb 0}$
is  unimodular.
\item[Axiom] \ref{Ax:non-degenerate}: The integral of the twist and
  of its inverse are non zero.
\item[Axiom] \ref{A:proj}:
There exists a projective $\UH$-module whose restriction to a $\U$-module remains projective.
\end{enumerate}
We will prove that the semi-restricted, unrolled and small quantum groups associated to a simple Lie algebra of rank $r$ satisfy these five axioms and thus lead to the 3-manifold invariants define in this paper.
   
The first main result of this paper is to embed
the unrolled quantum group $\UH$
into a topological ribbon Hopf algebra
 $\widehat{\UH}$
 which has a topology of a complete nuclear space.  In particular, we
 give a completion of the unrolled quantum group which is a
 topological ribbon Hopf algebra describe as follows.  The subalgebra
 generated by the Cartan Lie algebra $\H$ in this unrolled quantum
 group is embedded into the space $\Ce(\H^*)$ of holomorphic functions
 on the dual space $\H^*$.  Then the unrolled quantum group is
 embedded in its completion which is a complete nuclear space and has
 the topology of uniform convergence on compact sets. Using this
 completion we show the quasitriangular structure of the small quantum
 group can be lifted to a topological quasitriangular structure and a
 topological ribbon Hopf algebra.  This leads to a large class of
 topological ribbon Hopf algebra.  In a different context, Markus
 J. Pflauma and Martin Schottenloher \cite{Markus13} already consider
 other kinds of nuclear Hopf algebras and their holomorphic
 deformations.

 The techniques
discussed above were first used by the second author to defined a
topological ribbon Hopf algebra from the super Lie algebra
$\mathfrak{sl}(2|1)$.  In \cite{Ha18a}, Ha used this algebra, its
associated universal invariant of links, a discrete Fourier transforms
and a right $G$-integral to construct an invariant of Virelizier-Hennings type of $3$-manifolds
decorated by a cohomology class.

Using different techniques a modifed Hennings type invariant is given in \cite{DGP}. This is an invariant defined on pairs $(M,T)$ consisting of a $3$-manifold $M$ and a bichrome graph $T$ inside.  The main ingredients of the construction in \cite{DGP} consists of a finite dimensional Hopf algebra $H$ and a modified trace of the category $H$-mod of finite dimensional left modules over $H$.
The second main result of this paper is to define a
modified graded Hennings type invariant
for the unrolled quantum groups satisfying the five axioms of this paper.

The modified invariants of this paper are thus graded versions of the
modified Hennings invariants of \cite{DGP}.  In \cite{DGP2} the later
invariants are shown to give a (non-graded) Hennings type formula for
the Reshetikhin-Turaev type quantum invariants of \cite{FcNgBp14} associated to zero
cohomology classes.  Here we give a similar modified graded Hennings-Virelizier
type formula for the graded version of the invariants of \cite{FcNgBp14} with non-zero cohomology classes.
To do this
we
introduce a new algebraic tool called the modified integral.

Let us discuss the organization of the paper.
 In Section \ref{Section
  unrolled qt group}, we recall some needed properties of unrolled
quantum groups and their representations. Section \ref{topo quntum and
  uni invariant} contains the construction of a topological ribbon
Hopf algebra from an unrolled quantum group.  It also describes some
special elements called {\em power elements} which will be used to
color the mixed coupons of the bichrome graphs. In Section
\ref{Section of inv of bichrome graphs} we adapt the construction of
the universal invariant associated to a ribbon Hopf algebra.
 In Section \ref{Section of inv of 3-manifolds} we define both the graded Hennings invariant $\He$ and its modified version $\He'$ which are invariants of a
 admissible 
 compatible triple $(M, \Gamma, \omega)$ where $M$ is a closed $3$-manifold, $\Gamma$ is a bichrome graph inside $M$ and $\omega$ is a cohomology class of $H^1(M\setminus\Gamma; G)$.
In Section \ref{modified integral}, we present the notion of a
modified integral which allows us to relax the admissibility condition
for the modified graded Hennings invariant and provides another way to
determine the invariant.  Finally, in Section \ref{s:rel} we explain
how this invariant generalize some previously defined non semi-simple
invariants.
\subsection*{Acknowledgments}
Nathan Geer was partially supported by the NSF grant DMS-1452093.  This grant also supported Ngoc Phu Ha to visit Utah State University where much of the work of this paper took place.   

\section{Quasi-triangular unrolled Hopf Algebras}\label{Section unrolled qt group}
In this section we give our definition of a quasi-triangular unrolled Hopf Algebra and consider its category of weight modules.  Similar algebras have been considered in \cite{AS18}.  We discuss why the previously defined unrolled quantum groups of \cite{FcNgBp14, NgBp13, FrNaBe,GP18} are examples of the algebras define in this section.  
\subsection{Unrolled Hopf Algebras}
Fix an integer $\ell$ and a $\ell$\ith root
of unity $\xi=\e^{\frac{2i\pi}\ell}$.
If $(G,+)$ is a abelian group and $V$ is a vector space we say $V$ is  \emph{$G$-graded} if there is a decomposition 
$V = \bigoplus_{g \in G} V_g$ 
where each $V_g$  is a vector space. We say an algebra $V$ is \emph{$G$-graded} if it its underlying vector space is $G$-graded, $1\in G_0$ and the multiplication preserves the grading:  if $v\in V_g$ and $w\in V_h$ then $vw\in V_{g+h}$.  If $v\in V_g$ we say $v$ is \emph{homogeneous}.

 Let  $\Lambda$ be a free abelian group of rank $\rk$ 
and
$W$ be a $\Lambda$-graded finite dimensional vector space over $\C$ with a special degree $0$ element denoted $1_W$.
If $w\in W$ is homogeneous, then we denote its degree by
$\w w\in \Lambda$.  Let
$\Lambda^*=\Hom_{\Z}(\Lambda,\Z)$ be the abelian group of group morphism between $\Lambda$ and $\Z$.   Fix a basis $\{a_1,\ldots,a_\rk\}$ of
$\Lambda^*$.  Finally, let 
$\H=\Lambda^* \otimes_{\Z} \C$.

The group ring $\C[\Lambda^*]$ is the free vector space on $\Lambda^*$ over $\C$ which is generated by the formal variables $\qn{K^a:a\in\Lambda^*}$,
with the relation $K^{a+b}=K^a K^b$ for any $a,b\in \Lambda^*$.  In particular,  
$\C[\Lambda^*]$ can be identified with  the ring of Laurent polynomials in the $\rk$ variables $K_i:=K^{a_i}$.
It has the structure of a Hopf algebra where each element of the set  $\qn{K^a:a\in\Lambda^*}$ is a group-like element.

Recall that a Hopf algebra  is \emph{pivotal} if the square of the antipode can be expressed via the conjugation by a group-like element, called pivot $g$.
We assume the following axiom.
\begin{axiom}\label{pivotal axiom} There exists a $\Lambda$-graded pivotal Hopf algebra 
  $\U$ with underlying vector space
$$\U=W\otimes_{\C} \C[\Lambda^*]$$
which is an extension of the Hopf algebra
$\C[\Lambda^*]=1_W\otimes\C[\Lambda^*]$ and in which for any
homogenous $x\in W\cong W\otimes1$ the following relation holds:
\begin{equation}
  \label{eq:K}
  K^a x K^{-a}=\xi^{a(\w x)}x
\end{equation}
for all $a\in \Lambda^*$.
\end{axiom}
We denote the comultiplication, counit and antipode maps of the  Hopf algebra $\U$ by $\Delta, \epsilon$ and $S$, respectively.
The axiom implies that $\U$ is a $\Lambda$-graded Hopf algebra in
which $\C[\Lambda^*]$ is a commutative Hopf subalgebra in degree $0$.
In particular, the unit is $1_\U=1_W\otimes1\in W\otimes_{\C} \C[\Lambda^*]$. Also, for any
homogenous $x,y\in\U$, $|xy|=|x|+|y|\in\Lambda$,
$|x_{(1)}|+|x_{(2)}|=|x|$ where $\Delta x=\sum x_{(1)}\otimes
x_{(2)}$ and finally $S^2(x)=gxg^{-1}$.

We can now build the unrolled version $\UH$ of $\U$: Let $S\H$ be the
tensor symmetric algebra of $\H$
which
can be identified
 with polynomial maps on $\H^*=\Lambda\otimes_\Z{\C}$.
It has a Hopf algebra structure with elements of $\H$ being primitive.
Consider the semi-direct product
\begin{equation}
  \label{eq:UH}
  \UH=\U\rtimes S\H
\end{equation}
where for each $a\in \Lambda^*$ the action of the associated
element $H_a\in S\H$ on a homogeneous element $x\in\U$ is given
by $[H_a, x]=a(\w x) x$ or equivalently 
\begin{equation}
  \label{eq:weight}
  H_ax=x(H_a+a(\w x)).
\end{equation}

\begin{Prop}\label{P:UHisHopf}
The Hopf algebra morphisms of $\U$ and $S\H$ naturally extend (via multiplication) to comultiplication $\Delta$, counit $\epsilon$ and  antipode $S$ maps on $\UH$, making $\UH$ into a $\Lambda$-graded pivotal Hopf algebra with the same pivot $g$ as in $\U$.
\end{Prop}
\begin{proof} To see that $\UH$ is a Hopf algebra we need to check that the extended maps satisfy the relations given in Equation \eqref{eq:weight}.  For example, it is easy to check that  
$$\Delta(H_ax)=\Delta(H_a)\Delta(x)
= \Delta(x)\Delta(H_a)+ \Delta(x)a(\w x)= \Delta(x(H_a+a(\w x))).$$

Since $g$ is grouplike $\Delta(g)=g\otimes g$ so its $\Lambda$-degree is zero. So the conjugation by $g$ is trivial on $S\H$, as is the square of the antipode.  
\end{proof}

\begin{rmq}\label{R:sublattice}
  Suppose that $\Lambda'$ is a rank $r$ sub-lattice of $\Lambda^*$ such
  that $\U_{\Lambda'}=W\otimes_{\C} \C[\Lambda']\subset\U$ is a sub-Hopf
  algebra.  Then $\U_{\Lambda'}^H=\U_{\Lambda'}\rtimes S\H$ is a sub-Hopf algebra of
  $\U^H$.  Furthermore, if the pivot $g\in\U$ is in $\U_{\Lambda'}$ then $g$ is
  also a pivot for $\U_{\Lambda'}$ and $\U_{\Lambda'}^H$.  We say that $\U_{\Lambda'}$ and
  $\U_{\Lambda'}^H$ are $\Lambda'$ versions of $\U$ and ${\U}^H$.
\end{rmq}

 We call $\UH$ an {\em unrolled Hopf algebra}.
\begin{Exem}\label{ex1} (Continued in Examples \ref{ex1b}, \ref{ex1c} and \ref{ex12d}).
In this example we consider the case of quantum $\mathfrak{sl}_2$.
  Assume $\ell$ is greater than $3$ and let $\ell'=\ell/gcd(\ell,2)$.
  Let $\U$
  be the $\C$-algebra with generators 
$E$, $F$, $K$ and $K^{-1}$ with relations
\begin{align}\label{E:RelDCUqsl}
  KK^{-1}&=K^{-1}K=1, & KEK^{-1}&=qE, &
   KFK^{-1}&=q^{-1}F, \notag\\
  [E,F]&=\frac{K^2-K^{-2}}{q-q^{-1}}, & E^{\ell'}&=0, & F^{\ell'}&=0.
\end{align}
The algebra $\U$ is a Hopf algebra where the coproduct, counit and
antipode are defined by
\begin{align*}\label{E:HopfAlgDCUqsl}
  \Delta(E)&= 1\otimes E + E\otimes K^2, 
  &\varepsilon(E)&= 0, 
  &S(E)&=-EK^{-2}, 
  \\
  \Delta(F)&=K^{-2} \otimes F + F\otimes 1,  
  &\varepsilon(F)&=0,& S(F)&=-K^2F,
    \\
  \Delta(K)&=K\otimes K
  &\varepsilon(K)&=1,
  & S(K)&=K^{-1}.
\end{align*}
  The Hopf algebra $\U$ is pivotal with pivot
$g=K^{2-2\ell'}$.  The sub-Hopf algebra $\U_\xi(\mathfrak{sl}_2)$
generated by $E,F$ and $K^{\pm2}$ is known as the semi-restricted
quantum group where the word semi-restricted is used because
$E^{\ell'}=F^{\ell'}=0$.

Let $\UH$ be the 
$\C$-algebra given by generators $E, F, K^{\pm1}, H$ and
relations in Equation \eqref{E:RelDCUqsl} plus the relations:
\begin{align*}
  HK&=KH, 
& [H,E]&=E, & [H,F]&=-F. 
\end{align*} 
The algebra $\UH$ is a Hopf algebra where the coproduct, counit and 
antipode are defined by the above equations for $E, F$ and $K$ and 
\begin{align*}
  \Delta(H)&=H\otimes 1 + 1 \otimes H, 
  & \varepsilon(H)&=0, 
  &S(H)&=-H.
\end{align*}
The sub-Hopf algebra generated by $E,F,K^{\pm2}$ and $H$ is known as
$\UH_\xi(\mathfrak{sl}_2)$ and called the unrolled quantum group associated to 
$\mathfrak{sl}_2$\footnote{In most of the litterature, including
  \cite{FcNgBp14,NgBp13,NgBpVt09}, the element $K^2$ and $2H$ are
  called $K$ and $H$ respectively.}.

We now explain that our construction recovers these quantum groups.
Let $\Lambda$ be the rank one free abelian group $\Z$.  Let $W$ be the
$\Lambda$-graded $\C$-vector space with basis $F^iE^j$ for
$0\leq i,j< \ell'$ and grading given by $\w{F^iE^j}=j-i$.  Here
$\C[\Lambda^*]$ is identified with Laurent polynomials in the variable
$K$, which as above is a Hopf algebra where each element is group
like.  As a vector space $\U$ is isomorphic to
$W\otimes_{\C} \C[\Lambda^*] $ and satisfies Axiom \ref{pivotal
  axiom}.  Let $\Lambda'=2\Lambda^*$ then $\U_\xi(\mathfrak{sl}_2)=\U_{\Lambda'}$ where  $\U_{\Lambda'}$  is the $\Lambda'$ version of $\U$, see Remark \ref{R:sublattice}.  Moreover, $S\H$ is $\C[H]$ and the Hopf
algebra of Proposition \ref{P:UHisHopf} is isomorphic to $\UH$ where 
$\UH_\xi(\mathfrak{sl}_2)=\U_{\Lambda'}^H$ is its $\Lambda'$ version.
\end{Exem}

More generally, we have:
\begin{Exem}\label{ex2}
  (Continued in Examples \ref{ex2b}, \ref{ex2c} and \ref{ex12d}). Here let $\ell$ be odd and greater than $3$.  Let $\mathfrak{g}$ be a simple finite
  dimensional complex Lie algebra of rank $r$ with a  set of simple roots $\{\alpha_1,...,\alpha_r\}$.  Let $A =(a_{ij})_{1\leq i,j\leq r}$ be the Cartan matrix
  corresponding to these simple roots. Consider the unrolled
  quantum group $\UH_\xi(\mathfrak{g})$ associated to $\mathfrak{g}$, given in
  \cite{NgBp13}.  This algebra has generated denoted $K_\beta$, $X_i$,
  $X_{-i}$ and $H_{\alpha_i}$ where $\beta$ is in the root lattice $\Lambda$
  and $1\leq i\leq r$ (for the relations see
  \cite{NgBp13}).  The semi-restricted quantum group
  $\U_\xi(\mathfrak{g})$ is the subalgebra of $\UH_\xi(\mathfrak{g})$
  generated by $K_\beta$, $X_{\pm i}$ for all $\beta$ and $i$.

  To describe this example we will need the following notation here
  and later.   There exists a diagonal matrix
  $D=diag(d_1, ..., d_r)$ such that $DA$ is symmetric and
  positive-definite ($D$ is unique if $1\in\qn{d_i}_i\subset\Z$).  Let
  $\Lambda$ be the root lattice which is the $\Z$-lattice generated by
  the simple roots $\{\alpha_{i}\}$ and let $\{\alpha^*_{i}\}$ be the
  dual basis of $\Lambda^*$.  Let $\H= \Lambda^*\otimes_\Z \C$ be the
  Cartan subalgebra of $\mathfrak{g}$ and
  $\B : \H^* \times \H^*\to \C$ be the symmetric bilinear form defined
  by $\B(\alpha_i,\alpha_j)=d_ia_{ij}$.  For $\lambda \in \Lambda$ let
  $B_\lambda = \B(\lambda, \cdot)\in \C[\Lambda^*]$ and
  $\Lambda'=\qn{B_\lambda:\lambda\in\Lambda}$.  Define
  $H_{\alpha_i}=\sum_ja_{ij}H_{\alpha^*_j}$ so that the lattice
  $\Lambda'$ also identify with the free group generated by the
  elements $\{d_iH_{\alpha_i},i=1,\ldots,r\}$ in $\H$.

 To describe $W$, let $\beta_1,...,\beta_N\in \Lambda$ be an ordering of the set
  of positive roots where $N=\frac{\dim(\mathfrak{g})-r}2$.
  For each
  $i=1,...,N$, let $X_{\pm\beta_i}$ be the positive (resp. negative)
  root vector of $\U_\xi(\mathfrak{g})$ (see for example \cite[Section
  8.1 and 9.1]{Chari95}).  Let $W$ be the $\Lambda$-graded $\C$-vector space with
  homogeneous basis
$$X_{\beta_1}^{i_1}X_{\beta_2}^{i_2}...X_{\beta_N}^{i_N}X_{-\beta_1}^{j_1}X_{-\beta_2}^{j_2}...X_{-\beta_N}^{j_N}$$ 
with grading $\sum_{k=1}^r (i_k - j_k)\beta_k$ for $ i_1,...,i_N,j_1,...,j_N \in \{0,...,\ell-1\}$.

 The vector space $\U_\xi(\mathfrak{g})\simeq W\otimes_{\C} \C[\Lambda']$ embeds into $\U=W\otimes_{\C} \C[\Lambda^*]$.  The Hopf algebra structure of $\U_\xi(\mathfrak{g})$  extends uniquely to a Hopf algebra structure on $\U$ such that  relations \eqref{eq:K} hold where $\w{X_{\pm i}}=\pm\alpha_i$.  Then $\U$ satisfies Axiom \ref{pivotal
  axiom} with pivot given by $g=B_{2(1-\ell)\rho}$ where $\rho$ is the
half sum of all positive roots (see \cite{NgBp13}).
 Finally, 
  $\U_{\Lambda'}=\U_\xi(\mathfrak{g})$ and $\U_{\Lambda'}^H=\U^H_\xi(\mathfrak{g}) $ are 
 the $\Lambda'$ version of the quantum group
  $\U$ and  $\UH$.

\end{Exem}
\subsection{Category of weight modules}

In the following, we will use the notation
$\xi^x:=\exp\bp{\frac{2i\pi x}\ell}=\sum_{n=0}^\infty \frac{1}{n!}
\big(\frac{2i\pi x}\ell\big)^n$ for $x$ a complex number or an element
of a topological algebra.

A finite dimensional $\UH$-module $V$ is a \emph{weight module} if it
is a semi-simple module over the subalgebra $S\H$
and
\begin{equation}
  \label{eq:wm}
  K^a=\xi^{H_a}
\end{equation}
as operators on $V$, for any $a\in\Lambda^*$.
Let $\cat^H$ be the tensor category of $\UH$-weight modules.
The eigenspaces for the action of $S\H$ on a weight module $V$ are
called \emph{weight spaces} of $V$ and this action gives a
$\H^*$-grading on $V$.

Recall that a \emph{pivotal category} is a tensor category with left
duality $\{ \lcoev_V, \lev_V\}_V$ and right duality
$\{\rcoev_V , \rev_V \}_V$ which satisfy certain compatibility
conditions, see for example \cite{BW}.
The category $\cat^H$ is a pivotal category with duality maps:
	\begin{align*}
	\lcoev_V :& \C \rightarrow V\otimes V^{*}, \text{ given by } 1 \mapsto \sum
	v_i\otimes v_i^*,  \\
	\lev_V: & V^*\otimes V\rightarrow \C, \text{ given by }
	f\otimes w \mapsto f(w),\\
	\rcoev_V :& \C \rightarrow V^*\otimes V, \text{ given by } 1 \mapsto \sum
	v_i^*\otimes g^{-1}v_i,  \\
	\rev_V: & V\otimes V^*\rightarrow \C, \text{ given by }
	w\otimes f \mapsto f(g w)
	\end{align*}
where $\{v_i\}$ is a basis of $V$ and $\{v_i^*\}$ is its
dual basis of $V^*=\Hom_\C(V,\C)$.  

Let
\begin{equation}
   G=\H^*/\Lambda\simeq (\C/\Z)^\rk       
\end{equation}
then $\cat^H$ is $G$-graded: a module of $\cat^H$ is homogeneous of
degree $\ba\in G$ if all its weights belong to $\ba$.  We call $\cat^H_\ba$
the full subcategory of degree $\ba$ homogeneous modules. One easily
check that any module of $\cat^H$ is a direct sum of homogeneous
module and that the $\Hom$ set of two homogeneous module of different
degrees is zero.  We summarize this by writing
\begin{equation}\label{eq:G-graded}
  \cat^H=\bigoplus_{\ba\in G}\cat^H_\ba
\end{equation}

Given an element
$$\Q={\sum_i  c_i\otimes
  c'_i}\in\ \Lambda^*\otimes_\Z\Lambda^*\ \subset\ \H\otimes\H\
\subset\ \UH\otimes\UH$$
we define the following four maps.
First, let $\B: \H^*\times\H^*\to \C$ be the symmetric bilinear form given
by the element $\Q$, in particular,
$\B(\lambda,\mu)=\Q(\lambda\otimes\mu)\in\Z$ for
$(\lambda,\mu)\in\Lambda^2$. Second, if $V,W$ are $\H^*$-graded vector
space then let $\HH_{V,W}=\xi^\Q: V\otimes W\to V\otimes W$ be the
operator defined by
$\HH_{V,W}(v\otimes w)=\xi^{\B(\w v,\w w)}v\otimes w$ for homogeneous
vectors $v$ and $w$.
Third, consider the map 
$B:\Lambda\to\Lambda^*\subset\C[\Lambda^*]$, $\lambda\mapsto B_\lambda$ where $B_\lambda = \B(\lambda, \cdot)$.  Then  
we have $B_\lambda B_\mu=B_{\lambda+\mu}$ in $ \C[\Lambda^*]$ for $\lambda,\mu\in \Lambda$.
Using this map we can
define an outer automorphism $\wt\HH$ of $\UH\otimes\UH$ given by 
\begin{equation}
  \label{eq:tH}
  \wt\HH (x\otimes y)=xB_{\w y}\otimes B_{\w x}y.
\end{equation}
for $x,y\in \UH$.  
This outer automorphism is compatible with
conjugation by $\HH$ in $\cat^H$:  if $\rho_{V_i}:\UH\to \End_{\cat^H}(V_i)$, for $i=1,2$, are objects in $\cat^H$ then
$$
(\rho_{V_1}\otimes \rho_{V_2})(\wt\HH(x\otimes y))=\HH_{V_1,V_2}(\rho_{V_1}(x)\otimes \rho_{V_2}(y))\HH_{V_1,V_2}^{-1}
$$
for $x,y \in  \UH$.

Throughout this paper, for two spaces $X$ and $Y$ we denote their flip map as $\tau :X\otimes Y\to Y\otimes X$ which given by $x\otimes y\mapsto y\otimes x$.  We call an element
\begin{equation}
  \label{eq:cRR}
  \check{\RR}=\sum_ix_i\otimes y_i\in\U \otimes\U
\end{equation}
a \emph{quasi R-matrix for $\Q$} if it satisfies the Relations (\v R1)--(\v R4) given  below:
\begin{align}
  \tau(\Delta u)&=\wt\HH\bp{\check\RR(\Delta u)\check\RR^{-1}},\,
                  \text{ for all }  u\in\UH\tag{\v R1}, \label{eq:Rdelta}\\
  \Delta_1\check\RR&=\wt\HH_{23}^{-1}(\check\RR_{13})\check\RR_{23}
                     \tag{\v R2}\label{eq:deltaR1},\\
  \Delta_2\check\RR&=\wt\HH_{12}^{-1}(\check\RR_{13})\check\RR_{12}
                     \tag{\v R3}\label{eq:deltaR2},
\end{align}
where $\check R_{23}=\sum_j 1_{\UH}\otimes x_j\otimes y_j$, $\check R_{12}=\sum_j  x_j\otimes y_j\otimes 1_{\UH}$, 
$\wt\HH_{23}^{-1}(\check\RR_{13})=\sum_j x_j\otimes (B_{\w
  {y_j}})^{-1}\otimes y_j$ and
$\wt\HH_{12}^{-1}(\check\RR_{13})=\sum_j x_j\otimes (B_{\w
  {x_j}})^{-1}\otimes y_j$.
Finally, we require $\check R$ is compatible  
with the pivot $g$ by assuming
\begin{equation}
  \label{eq:chR-twist}
  \tag{\v R4}\sum_iy_igx_iB_{\w{x_i}}=\sum_ix_ig^{-1}y_i(B_{\w{x_i}})^{-1}.
\end{equation}
Note, above we assume the expression of
$\check R=\sum_j x_j\otimes y_j$ is given with homogeneous elements
$x_i,y_i\in \U$.

\begin{axiom}\label{quasi axiom}
  Assume Axiom \ref{pivotal axiom} and $\UH$ has an element $\check{\RR}$ which is a
  quasi R-matrix for $\Q$.
\end{axiom}
In the rest of the paper, we assume Axiom \ref{quasi axiom} is true.
Then Hopf algebra $\UH$ is not quasi-triangular as
$\HH \notin \UH\otimes \UH$ but as we will see next the category
$\cat^H$ is still braided even ribbon.

Recall a \emph{braiding} on a tensor category $\cat$ consists of a family of natural
isomorphisms $\{c_{V,W}: V \otimes W \rightarrow W\otimes V \}$
satisfying the Hexagon Axiom:
$$c_{U,V\otimes W}=(\Id_V\otimes c_{U,W})\circ(c_{U,V}\otimes\Id_W)\;\;\;\;
c_{U\otimes V,W}=(c_{U,W}\otimes \Id_V)\circ(\Id_U\otimes c_{V,W})$$
for all $U,V,W\in\cat$.  We say $\cat$ is \emph{braided} if it has a
braiding.  If $\cat$ is pivotal and braided, one can define a family
of natural automorphisms
$$\theta_V= (\Id\otimes \rev_V)(c_{V,V}\otimes\Id)(v\otimes \lcoev_V):V\to V.$$
Following \cite{GP18}, we say that $\cat$ is \emph{ribbon} and the
morphism $\theta$ is a \emph{twist} if
\begin{equation}
  \label{eq:twistc}
  \theta_{V^*}=(\theta_V)^*
\end{equation}
for all $V\in\cat$.

\begin{Prop} \label{CHribbon}
  Category $\cat^H$ is ribbon with pivotal structure
  given above and braiding given by
  $c_{V,W}=\tau\circ\HH\circ\check\RR:V\otimes W\to W\otimes V$.
\end{Prop}
\begin{proof}
  Axioms (\v R1)--(\v R4) imply that for weight modules
  $V,V_1,V_2,W$,
  \begin{enumerate}
  \item[(\v R1)]$\implies$ $c_{V,W}:V\otimes W\to W\otimes V$ is a morphism in $\cat^H$,
  \item[(\v R2)]$\implies$ $c_{V_1\otimes V_2,W}=(c_{V_1,W}\otimes\Id)(\Id\otimes c_{V_2,W})$,
  \item[(\v R3)]$\implies$ $c_{W,V_1\otimes V_2}=(\Id\otimes c_{W,V_2})(c_{W,V_1}\otimes\Id)$,
  \item[(\v R4)]$\implies$ the pivotal structure and the braiding are
    compatible i.e.
    $\theta_V:=(\Id\otimes \rev_V)(c_{V,V}\otimes\Id)(\Id\otimes
    \lcoev_V)$ is a twist in $\cat^H$.
  \end{enumerate}
  We justify with more details this last sentence.  The pivotal
  structure gives natural isomorphisms $V^{**}\cong V$. Using them,
  the dual and equivalent equality to \eqref{eq:twistc} is given by
  $$\theta'_V=\theta_V\text{ where }\theta'_V= (\lev_V\otimes\Id)
  (\Id\otimes c_{V,V}) (\rcoev_V\otimes\Id).$$
  Recall that $\rev_V=\lev_V\circ\tau\circ(g\otimes1)$ and
  $\rcoev_V=\bp{\Id\otimes g^{-1}}\circ\tau\circ\lcoev_V$.  As $|g|=0$, the pivot commute with $\H$.
  Similarly, as $\wt\HH$ and $\tau$ are trivial on $\Delta(\H)$,
 then Relation \eqref{eq:Rdelta} implies that $\check R$ commute with $\Delta(\H)$
  and thus $\w{x_i}+\w{y_i}=0$.  Then we can compute for a vector
  $v\in V$ with weight $\lambda$:
  \begin{align*}
    \theta_V(v)
    &=(\Id\otimes \rev_V)(c_{V,V}\otimes\Id)(v\otimes \lcoev_V)\\
    &=(\Id\otimes \lev_V\tau)(\Id\otimes\rho(g)\otimes\Id)
      (\tau\HH_{V,V}\rho^{\otimes2}(\check R)\otimes\Id)(v\otimes \lcoev_V)\\
    &=\sum_i(\Id\otimes \lev_V\tau)(\tau \HH_{V,V}\otimes\Id)
      \bp{\rho(gx_i)(v)\otimes (\rho(y_i)\otimes\Id)(\lcoev)}\\
    &=\sum_i\xi^{\B\bp{\w{x_i}+\lambda,\lambda}}(\Id\otimes \lev_V\tau)(\tau\otimes\Id)
      \bp{\rho(gx_i)(v)\otimes (\rho(y_i)\otimes\Id)(\lcoev)}\\
    &=\sum_i\xi^{\B\bp{\w{x_i}+\lambda,\lambda}}(\rho(y_i)\otimes \lev_V)
      (\Id\otimes \tau)(\tau\otimes\Id)\bp{\rho(gx_i)(v)\otimes \lcoev}\\
    &=\sum_i\xi^{\B\bp{\w{x_i}+\lambda,\lambda}}(\rho(y_i)\otimes \lev_V)
      \bp{\lcoev\otimes\rho(gx_i)(v)}\\
    &=\xi^{\B\bp{\lambda,\lambda}}\sum_i\xi^{\B\bp{\w{x_i},\lambda}}\rho(y_igx_i)(v).
  \end{align*}
  Here the third equality is because $\rho(g)$ commutes with
  $\HH_{V,V}$, the fourth is because $\theta_V$ is $\H$-equivariant so
  $\theta_V(v)$ has weight $\lambda$. This implies that the operator
  $\HH_{V,V}$ is applied on vectors of weight $\lambda$ on the right
  (and $|gx_iv|$ on the left).  Finally the last equality is the
  zig-zag relation.
  Similarly, we compute
  \begin{equation}
    \label{eq:ltwist}
    \theta'_V(v)=\sum_i\xi^{\B\bp{\lambda,\w{y_i}+\lambda}}\rho(x_ig^{-1}y_i)(v)
    =
    \xi^{\B\bp{\lambda,\lambda}}\sum_i\xi^{-\B\bp{\w{x_i},\lambda}}
    \rho(x_ig^{-1}y_i)(v)
  \end{equation}
  and the equality $\theta_V(v)=\theta'_V(v)$ follows from
  Relation \eqref{eq:chR-twist} with $B_{\w{x_i}}v=\xi^{\B(\w{x_i},\lambda)}v$.
\end{proof}
\begin{rmq}\label{R:L'b} Suppose that $\Lambda'$ is a rank $r$ sub-lattice of
  $\Lambda^*$ as in Remark \ref{R:sublattice}.  Then one easily check
  that the analogous category of $\U_{\Lambda'}^H$-weight module is
  indentified with $\cat^H$.  Indeed the restriction functor gives any
  $\UH$-weight module a structure of $\U_{\Lambda'}^H$-weight module but
  reciprocally, Condition \eqref{eq:wm} gives a unique way to extend
  any $\U_{\Lambda'}^H$-weight module to a $\UH$-module.
\end{rmq}
{\bf \noindent Notation:} For $q\in\C\setminus\qn1$ and $j\in\N$, we use
\begin{equation*}
  [j;q]=\frac{1-q^j}{1-q}\text{ and }\qN{j;q}!=[j;q]\cdots[1;q].
\end{equation*}
\begin{Exem}\label{ex1b}
  This example builds upon Example \ref{ex1} and is continued in
  Examples \ref{ex1c} and \ref{ex12d}.  Recall $\ell\ge3$,
  $\ell'=\ell/gcd(\ell,2)$ and $\U$ is a degree $2$ extension of the
 semi-restricted quantum group $\U_{\Lambda'}=\U_\xi(\mathfrak{sl}_2)$.  
  Then a quasi R-matrix for $\Q=2H\otimes H$
   is given by 
  \begin{equation*}
    \check {\mathcal R}=
    \sum_{j=0}^{\ell' -1}\frac{(\xi-\xi^{-1})^{j}}{\qN{j;\xi^{-2}}!}E^j\otimes F^j\in \U\otimes \U,
  \end{equation*}
  and if $v\otimes w\in V\otimes W\in\cat^H$ where $v,w$ have weights
  $\w v$ and $\w w$, then
  $\mathcal H_{V,W}(v\otimes w)=\xi^{2\w v.\w w}v\otimes w$.  With slight abuse of notation  we
  write the equality of operator on $\cat^H$ as
  $$\mathcal H=\xi^{2H\otimes H}.$$
  In \cite{ChKa95}, it is shown that $\check {\mathcal R}$ satisfies Relations \eqref{eq:Rdelta}-\eqref{eq:deltaR2}.
  For $\ell$ even, \cite{Ohtsuki02} has shown
  an equivalent form of Relation \eqref{eq:chR-twist}.  For $\ell$ odd
  the proof of Example \ref{ex2b} can be applied.  
 Hence for any $\ell\ge3$, the quantum group $\U$ associated to $\mathfrak{sl}_2$ of Example \ref{ex1}
 satisfies Axiom \ref{quasi axiom}.
\end{Exem}
\begin{Exem}\label{ex2b}
Here we use the notation and build upon Example \ref{ex2} (continued in Examples \ref{ex2c} and \ref{ex12d}).
 Recall $\U$ is an extension of the semi-restricted quantum group
  $\U_\xi(\mathfrak{g})$ associated to a simple finite
  dimensional complex Lie algebra $\mathfrak{g}$.
Let $\Q$ be the quadratic element 
  given by
  $\Q={\sum_{i,j} d_i \bar a_{ij} H_{\alpha_i}\otimes H_{\alpha_j}}$
  where $(\bar a_{ij})$ is the inverse of the Cartan matrix. In
  the basis $\{H_{\alpha^*_i}\}_{i=1\cdots n}$
  of $\H$
  dual to
  $\{\alpha_i\}_{i=1\cdots n}$, the form $\Q$ is given by
  ${\sum_{i,j} d_i a_{ij}H_{\alpha^*_i}\otimes H_{\alpha^*_j}}$.   So $\HH$ and
  $\check {\mathcal R}\in \U\otimes \U$
  are given by
\begin{equation}\label{eq:unrolled_R-matrix_g}
  \mathcal H=\xi^{\sum_{i,j} d_ia_{ij} H_{\alpha^*_i}\otimes H_{\alpha^*_j}},\quad
  \check {\mathcal R}=\prod_{i=1}^{N}\left
    ( \sum_{j=0}^{\ell -1}\dfrac{\left( (q_{\beta_i}-q_{\beta_i}^{-1})X_{\beta_i}\otimes
        X_{-\beta_i}\right)^j}{\qN{j;q_{\beta_i}^{-2}}!}\right),
\end{equation}
where
$q_{\beta_i}=\xi^{\B(\beta_i, \beta_i)/2}.$
In \cite{NgBp13}, $\check {\mathcal R}$ is shown to satisfy Relations \eqref{eq:Rdelta}-\eqref{eq:deltaR2}.
The Relation \eqref{eq:chR-twist} follows from the fact that
Proposition \ref{CHribbon} holds for quantum groups as shown in
\cite[Theorem 19]{GP18}:  in particular this implies that $\theta'_V=\theta_V$ for any
module $V\in\cat^H$ (see Equation \eqref{eq:ltwist}).  This equality can be interpreted as follows.  Let  $x$ (resp.\ $x'$) be the element on the left (resp.\  right) side of Equation \eqref{eq:chR-twist}.   Then $\theta'_V=\theta_V$ implies that $\rho_{V}(x-x')=0$ where $\rho_V: \UH\to \End_{\cat^H}(V)$ is the representation $V$.  Finally, Proposition~\ref{P:U inject in CH} implies that $x-x'=0$ which gives the desired result.
Thus, the quantum group $\U$
associated to $\mathfrak{g}$ of Example \ref{ex2} satisfies Axiom
\ref{quasi axiom}.
\end{Exem}
\section{Topological unrolled quantum groups}\label{topo quntum and uni invariant} 
Unrolled quantum groups are not ribbon Hopf algebras but they admit a
ribbon category of representations.  The goal of this section is to
define a topology on an unrolled quantum group $\UH$ so that its
completion $\wh\UH$ is a topological ribbon Hopf algebra.  This
topology is the one of uniform convergence on compact sets induced by
the Cartan part of the algebra. The topology is nice because it
produces explicit topological bases for $\wh\UH$ and allows the
identification of $\cat^H$ with finite dimensional $\wh\UH$ weight
modules.  The topological Hopf algebras are Hopf algebra objects in
the category of nuclear spaces. We first recall some definitions from
\cite{Markus13, Gro52, Ha18a}.
\subsection{Topological Hopf algebras}
Let $E$ and $F$ be locally convex spaces.  A topology is \emph{compatible} with $\otimes$ if both 1)
$\otimes: E \times F \rightarrow E \otimes F$ is continuous
and 2) for all $(e, f) \in E' \times F'$ the linear form
$e \otimes f: E \otimes F \rightarrow \C, x \otimes y
\mapsto e(x)f(y)$ is continuous \cite{Markus13,Gro52}.
The locally convex space $E$ is called \emph{nuclear}, if all the compatible
topologies on $E\otimes F$ agree after completion, for all locally
convex spaces $F$, i.e. the topology on the completion of $E\otimes F$
compatible with $\otimes$ is unique.
  For two nuclear
spaces $E$ and $F$ the completion of the tensor product $E \otimes F$
endowed with its compatible topology is denoted
$E \widehat{\otimes} F$.  We recall some facts about nuclear spaces (see  \cite{Gro52}): 1) a finite dimensional space is nuclear, 2) the
tensor product of two nuclear spaces is a nuclear space and 3) a space is
nuclear if and only if its completion is nuclear.  Complete nuclear spaces form a symmetric monoidal category
$\textbf{Nuc}$ with the product $\widehat{\otimes}$ (see
\cite{Markus13}).  Recall we denote the flip  isomorphism by
$\tau:E \widehat{\otimes} F\stackrel \sim\to F \widehat{\otimes} E$.

\newcommand{\Hopf}{{\mathcal A}}
A {\em topological Hopf algebra} is a Hopf algebra object in the monoidal category $\textbf{Nuc}$. That is a complete nuclear $\C$-space $\Hopf$ endowed with the $\C$-linear maps called the product, unit, coproduct, counit and antipode 
\begin{equation*}
m: \Hopf\widehat{\otimes} \Hopf \rightarrow \Hopf,  \eta: \C\rightarrow \Hopf,  \Delta: \Hopf \rightarrow \Hopf\widehat{\otimes} \Hopf, \ve: \Hopf\rightarrow \C \text{ and } S: \Hopf\rightarrow \Hopf 
\end{equation*}
which satisfy the axioms:
\begin{enumerate}
\item the product $m$ is associative on $\Hopf$ admitting $1_{\Hopf}=\eta(1)$ as unity,
\item $(\ve \widehat{\otimes} \Id_{\Hopf})\circ \Delta=(\Id_{\Hopf}\widehat{\otimes} \ve)\circ \Delta=\Id_{\Hopf}$ and the coproduct $\Delta$ is coassociative: $(\Delta \widehat{\otimes} \Id_{\Hopf})\circ \Delta=(\Id_{\Hopf}\widehat{\otimes} \Delta)\circ \Delta$,
\item $\Delta$ and $\ve$ are algebra morphisms where the associative product in $\Hopf\widehat{\otimes} \Hopf$ is determined by $(m\widehat{\otimes}  m)\circ (\Id_\Hopf \widehat{\otimes}  \tau \widehat{\otimes} \Id_\Hopf)$,
\item $m\circ (S\widehat{\otimes}  \Id_\Hopf)\circ \Delta=m\circ (\Id_\Hopf \widehat{\otimes} S)\circ \Delta=\eta\circ \ve$.  
\end{enumerate}

If $V$ is a finite dimensional $\C$-vector space we denote by
$\Ce(V)$ the space of entire
functions on $V$ endowed with the topology of uniform convergence on
compact sets.  Then $\Ce(V)$ is a complete nuclear
space. If $\{Z_i\}_{i=1,\ldots,n}$ are the coordinate functions of $V$ in some basis then we denote $\Ce(V)$ by $\Ce(Z_1,\ldots,Z_n)$. Remark that we have
$\Ce(V_1) \widehat{\otimes}
\Ce(V_2)\simeq \Ce(V_1 \times V_2)$
 where $V_1, V_2$ are finite dimensional
$\C$-vector spaces (see \cite[Theorem 51.6]{Fran67}).

We will need the following example later.  Recall the vector spaces $W$ and $\H$ given in Section \ref{Section unrolled qt group}.
The nuclear space $W \widehat{\otimes} \Ce(\H^{*})$ can be seen as
$W$-valued entire
functions, whose elements are power series in the variables
$(H_i)_{i=1\cdots n}$ with coefficients in $W$.  Alternatively, if
$\BW$ is a base of $W$, one can think of
$W \widehat{\otimes} \Ce(\H^{*})$ as the span of $\BW$ with
coefficients in $\Ce(\H^{*})$.
\begin{Prop}[\cite{Ha18a}] \label{completion}  
Let $\H_1,\H_2,W_1,W_2$ be finite dimensional $\C$-vector spaces. Then
\begin{equation*}
(W_1 \otimes \Ce(\mathfrak{H}_{1}^{*})) \widehat{\otimes} (W_2 \otimes \Ce(\mathfrak{H}_{2}^{*})) \simeq (W_1 \otimes W_2) \otimes \Ce(\mathfrak{H}_{1}^{*}\times \mathfrak{H}_{2}^{*}).
\end{equation*}
\end{Prop}

\subsection{Topology on the completion of $\UqgH$}
The space of entire functions is a nuclear space obtained as the completion of polynomial functions for the topology of uniform convergence on compact sets. We use a similar completion to define a topological ribbon Hopf algebra from $\UH$.\\
Recall that as a vector space,
\begin{equation}\UH=W\otimes_{\C} \C[\Lambda^*]\otimes_{\C} S\H
\end{equation}
where $W$ is a finite dimensional $\C$-vector space, $\C[\Lambda^*]=\C[K_1^{\pm1},\ldots,K_\rk^{\pm1}]$ is the space of Laurent polynomials in $\rk$ variables and $S\H=\C[H_1,\ldots,H_\rk]$ is the space of polynomials in $\rk$ variables (here we write $H_i$ for $H_{a_i}$ and $K_i$ for $K^{a_i}$, where $(a_i)_i$ is the fixed basis of $\Lambda^*$).

Denote $\Ce(\H^*)$ by $\Ce(H_1, \ldots, H_\rk )$.
We embed the commutative algebra 
$\C[\Lambda^*]\otimes_{\C} S\H$
 into
$\Ce(H_1, \ldots, H_\rk )$  by seeing $H_i$ as a
linear function on $\H^*$ and by sending
$K_i=K^{a_i}$ to $\xi^{H_i}=\exp(\frac{2 i \pi}{\ell} H_i)$.  Then
$\UH$ is embedded in $W\otimes_\C\Ce(\H^*)$.

Now we put on $\Ce(\H^*)$ the topology of uniform convergence on compact sets.  Then it is a complete nuclear space and as $W$ is finite dimensional, one has
$$W\otimes_\C\Ce(\H^*)=W\wh\otimes_\C\Ce(\H^*).$$
Furthermore the embedding of $\UH$ in this complete nuclear space is dense in it so we have:
\begin{equation}\label{UH=}
\widehat{\UH}=W \widehat{\otimes} \Ce(\H^*)\simeq W \otimes_{\C} \Ce(H_1, \ldots, H_\rk).
\end{equation}

Let us now describe a family of semi-norms that generates the topology
of $\widehat{\UH}$.  Any two norms on $W$ are equivalent. Let choose a
norm $\bN\cdot$ on $W$.  Then for any compact set $C$ in $\H^*$, we
have a semi-norm $\bN\cdot_C$ on $W\otimes_\C\Ce(\H^*)$ defined as
follows: If $\varphi \in C$ and $x$ is a multivariable power series
such that $x(H_1,\ldots,H_\rk)\in \Ce(\mathfrak{H}^*)$ then
$\varphi_*x(H_1,\ldots,H_\rk)$ is the evaluation of $x$ at $\varphi$,
that is
$\varphi_*x(H_1,\ldots,H_\rk)=x(\varphi(H_1),\ldots,\varphi(H_\rk))\in
\C$. Then for
$u=\sum_{k}w_{k}x_{k}(H_1,\ldots,H_\rk) \in \widehat{\UH}$ the
semi-norm of $u$ associated to $C$ is defined by
\begin{equation}
  \label{norm of x}
  \bN u_C=\sup_{\varphi \in C}\bN{ \varphi_{*}u}
\end{equation}
where $\varphi_{*}u=\sum_{k}w_{k}\varphi_{*}x_k\in W$.  One easily
sees that $\bN\cdot_C$ does not depend of the norm on $W$ up to
equivalence.  The family of semi-norms
$\{\bN\cdot_C\}_{C\text{ compact }}$ generates the topology of uniform
convergence on compact sets of $W$-valued entire functions.

The rest of this subsection is dedicated to proving that the complete
nuclear space $\widehat{\mathcal{U}^H}$ has a natural structure of a
topological ribbon Hopf algebra.
This means that $\widehat{\mathcal{U}^H}$ is a topological Hopf
algebra with an invertible element ${\RR}\in \UH\hotimes\UH$ called
the \emph{universal R-matrix} and a central element
$\theta\in \widehat{\mathcal{U}^H}$ called the
\emph{twist}\footnote{In the literature the inverse of the element
  $\theta$ is often considered.}
  \begin{align} \label{dkcm}
  \tag{R1}\RR\Delta(x)&=\Delta^{op}(x)\RR
                             \text{ for all }x \in \UH,  \\ 
  \tag{R2}\Delta \otimes \Id(\RR)&=\RR_{13}\RR_{23},  \label{dkcm2}\\
  \tag{R3}\Id \otimes \Delta(\RR)&=\RR_{13}\RR_{12}, \label{dkcm3}\\
  \tag{R4}S(\theta) &= \theta, \label{dkcm4}\\
  \tag{R5} \Delta(\theta)&=\RR_{21}\RR(\theta \otimes \theta), \label{dkcm5}\\
  \tag{R6} \ve(\theta)&=1. \label{dkcm6}
\end{align}
Recall the element $\Q={\sum_i c_i\otimes c'_i}\in\ \H\otimes\H$ and
define
$$\HH=\xi^\Q\in \UH\hotimes\UH.$$

\begin{theo}\label{main theorem 1}
$\wh\UH$ is a topological ribbon Hopf algebra
  with universal R-matrix $${\RR}=\HH\check{\RR}\in \UH\hotimes\UH$$
  and twist
  \begin{equation}\label{eq:whtwist}
    \theta = g^{-1}\left( m\circ (S^2\otimes\Id)(\RR)\right)\in\wh\UH.
  \end{equation}
\end{theo}
 
 Remark that as for ribbon Hopf algebras, following the lines
of
\cite[\S VIII]{ChKa95} we can prove that $(\wh\UH,{\RR},\theta)$ satisfies the following additional
properties:
  \begin{align*}
\RR_{12}\RR_{13}\RR_{23}&=\RR_{23}\RR_{13}\RR_{12},& (\ve \otimes \Id)(\RR)&=1=(\Id \otimes \ve)(\RR),
\\
(S \otimes \Id)(\RR)&=\RR^{-1}=(\Id \otimes S^{-1})(\RR),&\RR&=(S \otimes S)(\RR).
\end{align*}

To prove the theorem, we need the following three lemmas which imply that the product, coproduct and antipode are continuous.

As $W$ is finite dimensional, all norms on $W$ are equivalent.
Nevertheless to prove the theorem, we fix a convenient choice of norm
as follows: let $\BW=(w_1,\ldots,w_n)$ be a basis of $W$ where $w_i$
is homogeneous of degree $\lambda_i\in\Lambda$. Then let $\bN\cdot$ be
the maximum norm in $\BW$,
which means
$\bN{\sum_ix_iw_i}=\sup_i|x_i|$.
Similarly, $\BW^{\otimes k}$ is a basis of $W^{\otimes k}$ that will
be equipped with the associated maximum norm.  Then one easily checks
that $\bN{w\otimes w'}=\bN{w}\bN{w'}$.

Let $E$, $F$ be nuclear spaces and let $\mathcal N_F$ be a set of
semi-norms on $F$ that generate its topology.  Then remark that a
linear map $f:E \rightarrow F$ is continuous if and only if for any
continuous semi-norm $\bN{\cdot}_F\in\mathcal N_F$, there exists a continuous
norm $\bN{\cdot}_E$ on $E$ and a constant $\eta \in \mathbb{R}^{+}$ such
that
\begin{equation}\label{E:checking continuity of f}
\forall x\in E,\,\bN{f(x)}_F\leq \eta \bN{x}_E.
\end{equation}
\begin{Lem}\label{bdt}
  For each compact set $C \subset \mathfrak{H}^{*}$, there exists a
  compact set $C'$ and a constant $\lambda_C \in \mathbb{R}$ such that
  $\forall x,y \in \UH$, we have
\begin{equation*} 
  \bN{ xy }_{C} \leq
  \lambda_{C} \bN{ x \otimes y }_{C'\times C}.
\end{equation*}
\end{Lem}
\begin{proof}
  Let $\Lambda_W=\qn{\lambda_i,i=1\cdots N}$ which is a
  finite set and $C'=C+\Lambda_W$ which is also compact.  Then we write
  $$\varphi_*(xy)=\varphi_*\bp{\sum_iw_ix_i(H_1,\ldots,H_\rk)
    \sum_jw_jy_j(H_1,\ldots,H_\rk)}$$
  $$=\varphi_*\bp{\sum_{i,j}w_iw_jx_i(H_1+a_1(\lambda_j),\ldots,H_\rk+a_\rk(\lambda_j))y_j(H_1,\ldots,H_\rk)}$$
  $$=\sum_{i,j}\varphi_*\bp{w_iw_j}.(\varphi+\lambda_j)_*(x_i).\varphi_*(y_j).$$
  Then
  \begin{align*}
    \bN{xy}_C&=\sup_{\varphi\in C}\bN{\varphi_*(xy)}\le\sup_{\varphi\in C}\sum_{i,j}\bN{\varphi_*\bp{w_iw_j}}.|(\varphi+\lambda_j)_*(x_i)|.|\varphi_*(y_j)|\\
             &\le\sum_{i,j}\sup_{\varphi\in C}\bN{\varphi_*\bp{w_iw_j}}
               \sup_{\varphi\in C'}|\varphi_*(x_i)|\sup_{\varphi\in C}|\varphi_*(y_j)|\\
             &\le\bN{x}_{C'}\bN{y}_C\sum_{i,j}\bN{w_iw_j}_C,
  \end{align*}
  and one can take $\lambda_C=\sum_{i,j}\bN{w_iw_j}_C$ so that
  $ \bN{xy}_C\le\lambda_{C} \bN{ x }_{C'} \bN{ y }_{C}= \lambda_{C}
  \bN{ x \otimes y }_{C'\times C}$.
\end{proof}
By Proposition \ref{completion} we have
$\mathcal{U}^H\widehat{\otimes} \mathcal{U}^H \simeq W^{\otimes
  2}\otimes \Ce(H_{i,j})$ where
$H_{i, 1}= H_i \otimes 1,H_{i, 2}=1\otimes H_i$ for $i=1,\ldots,\rk$ and
the $H_{i,j}$ are seen as coordinates functions on
$\mathfrak{H}^*\times \mathfrak{H}^*$.  If
$C_2 \subset \H^* \times \H^*$ is compact, we have the associated
seminorm
\begin{equation}\label{norm of x, dn 2}
\bN{x}_{C_2}=\sup_{\varphi \in C_2}\bN{\varphi_*x}
\end{equation}
where $\bN{\cdot}$ is the maximum norm in the basis $\BW^{\otimes 2}$ of $W^{\otimes 2}$.
\begin{Lem}\label{L:contCopro}
  For each compact set $C_2\subset \H^* \times \H^*$, there exists a
  compact set $C\subset \H^*$ and $\lambda_{C_2} \in \mathbb{R}$ such
  that $\forall x \in \UH$, we have
  $$\bN{\Delta x}_{C_2}\leq \lambda_{C_2}\bN{ x}_{C}.$$
\end{Lem}
\begin{proof}
  Let $\sigma:\H^*\times \H^*\to\H^*$ be the sum
  $(\vp,\vp')\mapsto \vp+\vp'$ and $C=\sigma(C_2)$ which is a compact
  in $\H^*$.  Let $x=\sum_iw_ix_i(H_1,\ldots,H_\rk)$ be written in the
  basis $\BW$, then
  $$\bN {\Delta x}_{C_2}=
  \bN {\sum_i(\Delta w_i) x_i\bp{H_{1,1}+H_{1,2},\ldots,H_{\rk,1}+H_{\rk,2}}}_{C_2}$$
  $$\leq\sum_i\sup_{\varphi \in C_2}\bN{\varphi_*\Delta w_i}
  \sup_{\varphi \in C_2}|\varphi_*x_i\bp{H_{1,1}+H_{1,2},\ldots,H_{\rk,1}+H_{\rk,2}}|$$
  $$\leq\sum_i\bN{\Delta w_i}_{C_2}
  \sup_{\varphi \in C}|\varphi_*x_i(H_1,\ldots,H_\rk)|
  \leq\bN {x}_{C}\sum_i\bN{\Delta w_i}_{C_2},$$
  and one can take $\lambda_{C_2}=\sum_i\bN{\Delta w_i}_{C_2}.$
\end{proof}
\begin{Lem}\label{L:contAnti}
  For each compact set $C \subset \mathfrak{H}^{*}$ there exists a
  compact set $C' \subset \mathfrak{H}^{*}$ and a constant
  $\lambda_C$ such that
  $$\bN{S(x)}_{C}\leq \lambda_C \bN{x}_{C'}
  \text{ for } x \in \UH.$$
\end{Lem}
\begin{proof}
  Let $C'=-\Lambda_W-C=\bigcup_i-(\lambda_i+C)$ which is compact.  Let
  $x=\sum_iw_ix_i(H_1,\ldots,H_\rk)$ be written in the basis $\BW$, then
  \begin{align*}
    \bN{S(x)}_{C}&=\bN{\sum_ix_i(-H_1,\ldots,-H_\rk)S(w_i)}_{C}\\
                 &\leq\bN{\sum_iS(w_i)
                   x_i(-a_1(\lambda_i)-H_1,\ldots,-a_\rk(\lambda_i)-H_\rk)}_{C}\\
                 &\leq\bN{x}_{C'}\sum_i\bN{S(w_i)}_{C}.
  \end{align*}
\end{proof}

\begin{proof}[Proof of Theorem \ref{main theorem 1}] Recall $\wh\UH$
  is the completion of $\UH$.  Therefore, to show that $\wh\UH$ is a
  topological Hopf algebra it is enough to show the structure
  morphisms of the Hopf algebra $\UH$ are continuous for the topology
  of uniform convergence on compact sets:  This is obvious for the unit and follows for the counit
  as $\ve(x(H_1,\ldots,H_\rk))=x(0,\ldots,0)$.  Lemmas \ref{bdt},
  \ref{L:contCopro} and \ref{L:contAnti} combined with Equation
  \eqref{E:checking continuity of f} imply the continuity for the
  product, the coproduct and the antipode, respectively.

  Finally, we need to show that $\RR$ is a universal R-matrix and
  $\theta$ is a twist.  Properties \eqref{dkcm}--\eqref{dkcm3} follow
  from combining the fact that $\wt\HH$ is the conjugation by $\HH$
  and  $\check\RR$ satisfies Relations \eqref{eq:Rdelta}--\eqref{eq:deltaR2}.
  
  Next we prove Axiom \eqref{dkcm4}. Let $u=\sum_i S(b_i) a_i \in \wh\UH$ be the Drinfeld element where
  $\RR=\sum_i a_i\otimes b_i \in \UH\hotimes\UH$.  Then following
  \cite[\S VIII]{ChKa95} we have $u$ is invertible with inverse
  $u^{-1}=\sum_i b_iS^2(a_i)$.
  Then
  $$S(u^{-1})=\sum_i S^3(a_i)S(b_i)=\sum_i S^2(a_i)b_i=g\theta.$$
  But $\RR=\HH\check{\RR}$ where
  $\check{\RR}=\sum_ix_i\otimes y_i\in\U \otimes\U$. Then (recall $x_i$ and
  $y_i$ have opposite degree)
  \begin{align}
    \RR&=\sum_i\wt\HH(1\otimes y_i)\cdot\HH\cdot (x_i\otimes1)
    =\sum_i(B_{|x_i|}^{-1}\otimes y_i)\cdot\HH\cdot (x_i\otimes1)\nonumber\\
    \label{eq:Ra1}&=\sum_i(1\otimes y_i)\cdot\HH\cdot (B_{|x_i|}^{-1}x_i\otimes1)
  \end{align}
  and similarly
  \begin{equation}
    \RR=\sum_i(x_i\otimes 1)\cdot\HH\cdot(1\otimes B_{|x_i|}y_i)\\
    \label{eq:Ra2}
  \end{equation}
  thus from \eqref{eq:Ra1},
  \begin{equation}
    \label{eq:sui}
    S(u^{-1})=m\bp{\HH}\sum_iB_{|x_i|}^{-1}S^2(x_i)y_i
    =m\bp{\HH}g\sum_ix_ig^{-1}y_iB_{|x_i|}^{-1}
  \end{equation}
  and from \eqref{eq:Ra2}, one gets
  \begin{equation}
    \label{eq:ui}
    u^{-1}=m\bp{\HH}\sum_iB_{|x_i|}y_iS^2(x_i)=m\bp{\HH}g^{-1}\sum_iy_igx_iB_{|x_i|}.
  \end{equation}
  Hence Relation \eqref{eq:chR-twist} implies that
  $\theta=g^{-1}S(u^{-1})=gu^{-1}$ and property \eqref{dkcm4} follows
  from $g^{-1}S(u^{-1})=S(u^{-1})g^{-1}=S(gu^{-1})=S(\theta)$.

  To prove
  the last two equalities we follow the proof given in \cite[\S
  VIII]{ChKa95} for braided Hopf algebras.
  In particular, the element $u$ satisfies
  $$
  \Delta(u)=(\RR_{21}\RR)^{-1}(u \otimes u)=(u \otimes u)(\RR_{21}\RR)^{-1}
  \;\; \text{ and } \;\; \ve(u)=1.
  $$
  Then as $\theta=gu^{-1}$, we have 
  \begin{align*}
    \Delta(\theta)= \Delta(gu^{-1})
    &=\Delta(g)\Delta(u^{-1})\\
    &=(g\otimes g)((\RR_{21}\RR)^{-1}(u \otimes u))^{-1}\\
    &=(g\otimes g)(u^{-1} \otimes u^{-1})(\RR_{21}\RR)\\
    &=(\theta\otimes \theta)(\RR_{21}\RR)=(\RR_{21}\RR)(\theta\otimes \theta).
  \end{align*}
  This proves Equation \eqref{dkcm5}.  The proof of Equation \eqref{dkcm6} is similar.  
\end{proof}

\subsection{Category of module}
Let $\wh\U^H$-mod be the category of topological modules (i.e. object
of the category of complete nuclear space equipped with
$\wh\U^H$-module maps in this catetegory).
\begin{Prop}
  $\cat^H$ is a full subcategory of $\wh\UH$-mod.
\end{Prop}
\begin{proof}
  Since a module $V$ of $\cat^H$ is finite dimensional, its set of
  weights is compact and the action of $\UH$:
  $$\rho_V:\UH\otimes V\to V$$
  is continuous and extend to a continuous map
  $\rho_V:\wh\UH\otimes V\to V$.  
\end{proof}

\subsection{Power elements}\label{SS:PowerElements}
Recall
\begin{align*}
  \Lambda&\simeq\Z^\rk,&\Lambda^*&=\Hom_{\Z}(\Lambda,\Z)=\oplus_i\Z H_i,\\
  \H^*&=\Lambda\otimes_\Z{\C},&\H&=\Lambda^* \otimes_{\Z} \C=\oplus_i\C H_i.
\end{align*}
Let
$$G=\H^*/\Lambda\simeq (\C/\Z)^\rk\qquad \text{ and }\qquad G'=\H/\Lambda^* \simeq (\C/\Z)^\rk.$$

Also, let 
$$\H^{(n)}=(\H\otimes 1\otimes\cdots\otimes1) \oplus (1\otimes\H\otimes 1\otimes\cdots\otimes1) \oplus\cdots  \oplus ( 1\otimes\cdots\otimes1\otimes\H)$$ 
which is a subset of $ (S\H)^{\otimes n}\subset{(\UH)}^{\otimes n}$.
The elements of $\H^{(n)}$ are called \emph{linear}.
For $i=1\cdots\rk$ and $j=1\cdots n$ let  $H_{i,j}$ be the element of $\H^{(n)}$ which is 
$H_i$ in the $j$th direct sum (or tensor product) of $\H^{(n)}$.  The elements
$\{H_{i,j}:i=1\cdots\rk,j=1\cdots n\}$ form a basis of $\H^{(n)}$.  
Recall $\widehat{\UH}\simeq W \otimes_{\C} \Ce(H_1, \ldots, H_\rk)$ and it follows that
$$
(\UH)^{\wh\otimes n}\simeq \Ce(H_{i,j})\otimes_{\C} W^{\otimes n} .
$$

An element of $(\UH)^{\wh\otimes n}$  the form $\xi^H$ for some ${H\in\H^{(n)}}$ is called  {\em power linear}.    The
set $\mathcal L_n$ of power linear elements of $(\UH)^{\wh\otimes n}$
form a commutative Lie group among invertible elements of
$(\UH)^{\wh\otimes n}$.  Furthermore, we have the following two facts: 1) the elements of
$\mathcal L_1$ are group-like and 2) 
$\mathcal L_n=(\mathcal L_1)^{\otimes n}$. The complex vector space $\H^{(n)}$ contains a $\Z$-lattice
$\Lambda^{*(n)}=\bigoplus_{i,j}\Z {H_{i,j}}$.
The
exponential of elements of $\Lambda^{*(n)}$ form 
a multiplicative subgroup of $\mathcal L_n$ 
generated by the elements $K_{i,j}=\xi^{H_{i,j}}$.

We say a finite sum $\sum_\alpha H_\alpha H'_\alpha$ in ${(\UH)}^{\otimes n}$, for some $H_\alpha,H'_\alpha\in\Lambda^{*(n)}$ is {\em quadratic}.
 Similarly, elements of the set 
$$\mathcal Q_n=\qn{\xi^q:q\text{ in }{(\UH)}^{\otimes n}\text{ is quadratic}}$$
are called {\em power quadratic}.
For example, the element $\HH=\xi^\Q$ is a power quadratic element in
$\mathcal Q_2$.

Let us write $h_\bullet=\{H_{i,j}\}$ for the vector of coordinates functions. If
$q(h_\bullet)$ is a quadratic element then let
\begin{equation}
  \label{eq:polarisation}
  \wt q(h_\bullet,h_\bullet')=q(h_\bullet+h_\bullet')-q(h_\bullet)-q(h_\bullet')
\end{equation}
be the associated integral symmetric bilinear form such that
$\wt q(h_\bullet,h_\bullet)=2q(h_\bullet)$.  For example, if
$q(h_\bullet)=H_{1,1}H_{3,2}=H_1\otimes H_3$ and
$q(h'_\bullet)=H'_{1,1}H'_{3,2}$ are in $\Lambda^{*(2)}$ then
\begin{align*}\wt q(h_\bullet,h_\bullet')&=(H_{1,1}+H'_{1,1})(H_{3,2}+H'_{3,2})-H_{1,1}H_{3,2}-H'_{1,1}H'_{3,2}\\&=H_{1,1}H'_{3,2} +H'_{1,1}H_{3,2}.
\end{align*}
Remark that the notation $q(h_\bullet,h_\bullet')$ makes easy the
change of variables.  As such a function can be interpreted as a
function on $(\H^{(n)})^*\times(\H^{(n)})^*$, if
$\omega\in(\H^{(n)})^*$, we denote $q(\omega,h_\bullet')$ the function
$q(h_\bullet,h_\bullet')$ where the variables $h_\bullet$ has been
evaluated with $\omega$ i.e., we substitute $H_{i,j}$ with
$\omega(H_{i,j})\in\C$.
 
Since $q$ is an integral linear combination of the elements $H_{i,j}$,
we have
\begin{Lem}\label{lemma transfers quadratic to linear}
Let $q\in (\UH)^{\otimes n}$ be quadratic.  We have (1) if $\omega\in(\H^{(n)})^*\simeq (\H^*)^n$ then $\wt q(\omega,h_\bullet)$
  defines an element of $\H^{(n)}$, (2) furthermore, if 
  $\omega\in\Lambda^n$, then $\wt q(\omega,h_\bullet)\in(\Lambda^*)^n$ and (3)
 $q$ induces a map $G^n\to{G'}^n$ given by
  \begin{equation}
    \label{eq:G'}
    \bar\omega\mapsto \wt q(\bar\omega,h_\bullet)\in (G')^n.
  \end{equation}
\end{Lem}
\begin{rmq}\label{rk:powerG'}
  If $l$ is a linear element, the space
  $\xi^l\U^{\otimes n}\subset \mathcal L_n\U^{\otimes n}$ only depends
  on the class $\wb l$ of $l$ in $(G')^n$.
  Indeed, if $l'$ and $l$ are in the same class in $(G')^n$ then
  $\xi^{l'-l}$ is a unit in $\U^{\otimes n}$ so
  $\xi^{l'-l}\U^{\otimes n}=\U^{\otimes n}$.
  In particular, if $q$
  is quadratic and $\bar\omega\in G^n$ then
  $\xi^{\wt q(\bar\omega,h_\bullet)}\U^{\otimes n}$ is the
  space $\xi^l\U^{\otimes n}$ for any representant $l$ of
  $\wt q(\bar\omega,h_\bullet)\in (G')^n$.
\end{rmq}

The following lemma generalizes the conjugation by $\HH$ given by
Equation \eqref{eq:tH}:
\begin{Lem}\label{lemma changes the order of elements}
  Let $L=\xi^{l}\in\mathcal L_n$ and $Q=\xi^q\in\mathcal Q_n$ where $l\in \H^{(n)}$ is a linear element and $q\in (\UH)^{\otimes n}$ is quadratic.  Let $x$ be an
  homogeneous element of $(\UH)^{\wh\otimes n}$ of weight
  $\lambda\in\Lambda^n$.  Then
  $$LxL^{-1}=\lambda_*(L)x,$$
  and for 
  $K=\xi^{\wt q(\lambda,h_\bullet)}\in\xi^{\Lambda^{*(n)}}\subset\U^{\otimes n}$ we have
  \begin{equation}
    \label{eq:pow-com}
    QxQ^{-1}=\lambda_*(Q)xK.
  \end{equation}  
 \end{Lem}
\begin{proof}
  Recall we write $h_\bullet=\{H_{i,j}\}$, $L=\xi^{l(h_\bullet)}$ and
  $Q=\xi^{q(h_\bullet)}=\xi^{\frac12\wt q(h_\bullet,h_\bullet)}$ where
  $\wt q(h_\bullet,h'_\bullet)$ is the symmetric bilinear form
  determined by $q$.  Then from Equation \eqref{eq:weight}
  we have
  \begin{align*}
    Lx&=\xi^{l(h_\bullet)}\,x=x\,\xi^{l(h_\bullet+\lambda)}=
        x\,{\xi^{l(h_\bullet)}\,\xi^{l(\lambda)}}
        =x\,L\,\lambda_*(L),\quad\text{and}\\
    Qx&
        =x\,\xi^{q(h_\bullet+\lambda)}=
        x\,{\xi^{q(h_\bullet)}\,\xi^{q(\lambda)}\,\xi^{\wt q(\lambda,h_\bullet)}}
        =x\,{Q}\,\lambda_*(Q)\,\xi^{\wt q(\lambda,h_\bullet)}.
  \end{align*}
  The result follows because $Q$ commutes with $\xi^{\wt q(\lambda,h_\bullet)}$ as they are power elements.  
\end{proof}
The following proposition records some properties of power elements,
the proof follows directly from the previous lemma and the definitions
of power elements.  
\begin{Prop}
  (1) We have inclusions of algebras:
  $$\U^{\otimes n}\subset\ang{\mathcal Q_n\U^{\otimes n}}\subset
  \ang{\mathcal Q_n\mathcal L_n\U^{\otimes n}}\subset(\UH)^{\wh\otimes
    n},$$ where $\ang{X}$ means linear combinations of elements of $X$.
  
  \noindent
  (2) The sets $\mathcal Q_n\U^{\otimes n}$ and
  $\mathcal Q_n\mathcal L_n\U^{\otimes n}$ are stable by permutation
  of the factors, by the maps $S_i=\Id\otimes S\otimes \Id$ (antipode
  at the $i$\ith position).
  
  \noindent
  (3)  For
  $\Delta_i= \Id\otimes \Delta\otimes \Id$ and
  $m_i=\Id\otimes m\otimes \Id$, we have
  $$\Delta_i\bp{\mathcal Q_n\U^{\otimes n}}\subset
  {\mathcal Q_{n+1}\U^{\otimes n+1}}\text{ and } \Delta_i\bp{\mathcal
    Q_n\mathcal L_n\U^{\otimes n}}\subset {\mathcal Q_{n+1}\mathcal
    L_{n+1}\U^{\otimes n+1}},$$ 
  $$m_i\bp {\mathcal Q_{n+1}\U^{\otimes n+1}}\subset
  {\mathcal Q_n\U^{\otimes n}}\text{ and } m_i\bp{\mathcal
    Q_{n+1}\mathcal L_{n+1}\U^{\otimes n+1}}\subset {\mathcal
    Q_n\mathcal L_n\U^{\otimes n}}.$$ 
    
  \noindent
  (4) If $V\in\cat^H$ is an homogeneous
  weight module then
  $$\Id^{\otimes n}\otimes\rho_V\bp{{\mathcal Q_{n+1}\mathcal
      L_{n+1}\U^{\otimes n+1}}}\subset\mathcal
  Q_n\mathcal L_n\U^{\otimes n}\otimes_\C \End_\C(V).$$
\end{Prop}

\section{$G$-coalgebra and Discrete Fourier transforms}\label{S:Galg_DFT}
In this section we continue an algebraic investigation of $\U$.  The results of this section will not be used until Section \ref{Section of inv of 3-manifolds}.  In particular, Section \ref{Section of inv of bichrome graphs} can be read independently.  
\subsection{Pivotal Hopf $G$-coalgebra from $\U$}
We now recall the definition of Hopf group-coalgebra from
\cite{Tu2000,Vire02}.
For coherence we will keep an additive notation for the group
$G$\footnote{The original definition of Hopf group-coalgebra uses
  more general non abelian group.}. A Hopf $G$-coalgebra is a family
$H = \{H_\alpha\}_{\alpha\in G}$ of algebras over $\C$ endowed with a
comultiplication
$\Delta = \{\Delta_{\alpha,\beta} : H_{\alpha+\beta} \to H_{\alpha}
\otimes H_{\beta}\}_{\alpha,\beta\in G}$, a counit
$\epsilon: H_0 \to \C$, and an antipode
$S = \{S_\alpha : H_\alpha \to H_{-\alpha}\}_{\alpha \in G}$ which
verify some compatibility conditions, see \cite{Oh93, Turaev10,
  Vire02}.

Let $\mathcal{Z}^0=\Span\qn{K^{\ell a}:a\in\Lambda^*}=\C[K_1^{\pm \ell},\ldots,K_r^{\pm \ell}]$, recall $K_i=K^{a_i}$ where $\{a_1,\ldots,a_r\}$ is the fixed basis of $\Lambda^*$. 
Then $\mathcal{Z}^0$ is a Hopf subalgebra of $\U$ which is in the center of $\U$.
Let $\Hom_{\text{Alg}}(\mathcal{Z}^{0},\C)$ be the group of characters
on $\mathcal{Z}^{0}$ where the multiplication is given by
$kh=(k\otimes h)\circ \Delta$ and
$k^{-1}=k\circ S|_{\mathcal{Z}^{0}}$.  This group is isomorphic to
$G=\bp{(\C/\Z)^r,+}$: Indeed for each
$(\alpha_1,\ldots,\alpha_r)\in \C^r$ denote its image in $(\C/\Z)^r$
by $\ba=(\ba_1, \cdots, \ba_r)$.  Consider the isomorphism of groups
$G\to \Hom_{\text{Alg}}(\mathcal{Z}^{0}, \C)$ given by
$\ba=(\ba_1, \cdots, \ba_r)\mapsto (K_i^{\ell}\mapsto
\xi^{\ell\alpha_i}) $ for $i=1, \cdots, r$.  We will use this
isomorphism to identify $G$ with the characters on $\mathcal{Z}^{0}$.

We now give a Hopf $G$-coalgebra structure associated to $\U$.  
For each 
$\ba=(\ba_1, \cdots,
\ba_r)\in G$, consider the ideal  $I_{\ba}$ (resp.\ $I_{\ba}^H$) of $\U$ (resp.\ $\UH$) generated by the elements
$K_i^{\ell} -\xi^{\ell\alpha_i}$ for $i=1, \cdots, r$.  Equivalently,  $I_{\ba}$ and  $I_{\ba}^H$ are the ideals generated by $\bp{z-\ba(z)}_{z\in \mathcal{Z}^{0}}$ where $\ba$ is a character under the identification $ G\to \Hom_{\text{Alg}}(\mathcal{Z}^{0}, \C)$.

Let $\U_{\ba}= \U/I_{\ba}\text{ and }\UH_{\ba}= \UH/I_{\ba}^H$ with the projection maps 
$$\pi_\ba :\U\to   \U_{\ba} \text{ and } \pi_\ba^H: \UH\to   \UH_{\ba}.
$$
  As discussed in Example 2.3 of \cite{Ha18} there exists morphisms $\Delta_{\ba,\bb}$ and $S_\ba$ given by the commutative diagrams:
\\
\hspace*{4ex}
\begin{minipage}{0.4\linewidth}
\begin{diagram}
\U &\rTo^{\Delta} &\U \otimes \U\\
\dTo_{\pi_{\ba+ \bb}} & &\dTo_{\pi_{\ba}\otimes \pi_{\bb}}\\
\U_{\ba+\bb} &\rTo^{\Delta_{\ba, \bb}} &\U_{\ba} \otimes \U_{\bb} 
\end{diagram}
\end{minipage}
\hspace*{3ex} and \hspace*{3ex}
\begin{minipage}{0.4\linewidth}
\begin{diagram}
\U &\rTo^{S} &\U \\
\dTo_{\pi_{\ba}} & &\dTo_{\pi_{-\ba}}\\
\U_{\ba} &\rTo^{S_{\ba}} &\U_{-\ba} .
\end{diagram}
\end{minipage}
\\[2ex]
making $\U_\bullet=\{\U_{\ba}\}_{\ba\in G}$ a Hopf $G$-coalgebra.   Similarly, using $\pi_\ba^H$ there exists $\Delta^H_{\ba,\bb}$ and $S^H_\ba$ making  $\{\UH_{\ba}\}_{\ba\in G}$ a Hopf $G$-coalgebra.  Remark that $\U_{\wb 0}$ and its dual are ordinary Hopf algebras.

Let us now recall some definitions from \cite{Vire02,Ha18}
\begin{Def}
  Let $H_\bullet=\{H_{\ba}\}_{\ba\in G}$ be a Hopf $G$-coalgebra.
  \begin{enumerate}
  \item $H_\bullet$ is of \emph{finite type} if for any ${\ba}\in G$,
    $\dim_\C(H_{\ba})<\infty$.
  \item A \emph{$G$-grouplike} element is a family
    $\{x_{\ba}\in H_{\ba}\}_{\ba\in G}$ such that 
    $\Delta_{\ba,\bb}(x_{\ba+\bb})=x_{\ba}\otimes x_{\bb}$, for all $\ba,\bb\in G$.
  \item A \emph{pivot} for $H_\bullet$ is a $G$-grouplike element
    $\{g_{\ba}\}_{\ba\in G}$ such that
    $S_{-\ba}S_\ba=g_\ba\cdot g_\ba^{-1}$ ,  for all $ \ba\in G$.  If 
    $H_\bullet$ has a pivot we say it is a \emph{pivotal} Hopf
    $G$-coalgebra.
  \item A \emph{right (resp. left) $G$-integral} for $H_\bullet$ is a family
    of linear forms $\lambda^R=\{\lambda^R_\ba\in H_{\ba}^*\}_{\ba\in G}$
    (resp. $\lambda^L=\{\lambda^L_\ba\in H_{\ba}^*\}_{\ba\in G}$)  such that
$$   (\lambda^R_{\ba} \otimes \Id_{H_{\bb}})\Delta_{\ba,     \bb}(x)=\lambda^R_{\ba+\bb}(x)1_{\bb},$$
$$      \text{ (resp. }
      (\Id_{H_{\ba}}\otimes \lambda^L_{\bb}) \Delta_{\ba,
      \bb}(x)=\lambda^L_{\ba+\bb}(x)1_{\ba})
$$
for all $\ba,\bb\in G$ and $x\in{H_{\ba+\bb}}$.
  \item $H_\bullet$ is \emph{unimodular} if $H_{\wb0}$ is a unimodular Hopf
    algebra i.e.
    there exists a non zero element $\coint\in H_{\wb0}$, called a \emph{two side cointegral}, such that
    $$x\coint=\coint x=\ve(x)\coint \text{ for all } x\in H_{\wb0}. $$
  \end{enumerate}
\end{Def}
From now we make the following assumption for $\U_{\wb0}$:
\begin{axiom}\label{ax:unicorn}
 The Hopf $G$-coalgebra $\U_\bullet=\{\U_{\ba},g_\ba\}_{\ba\in G}$ is unimodular.
\end{axiom}
\begin{theo}[Virelizier \cite{Vire02}]\label{Th:right-int}
  If $(H_\bullet,\coint)$ is a finite type unimodular Hopf $G$-coalgebra, then
  \begin{enumerate}
  \item the space of right $G$-integral is $1$-dimensional,
  \item the right $G$-integral $\{\lambda^R_\ba\}_{\ba\in G}$
    such that $\lambda^R_{\wb 0}(\coint)=1$ is unique,
  \item there exists a unique $G$-grouplike element $\gamma_\bullet$
    called distinguished such that
    $\{\lambda^R_\ba(\gamma_\ba\cdot)\}_{\ba\in G}$ is a left
    $G$-integral,
  \item $\lambda^R_\ba(\gamma_\ba\cdot)=\lambda^R_{-\ba}\circ S_{\ba}$
    and $(S_{-\ba}S_{\ba})^2=\gamma_\ba\cdot\gamma_\ba^{-1}$
  \item     $\lambda^R_\ba(xy)=\lambda^R_\ba(S_{-\ba}S_{\ba}(y)x)$ for all $x,y\in H_\ba$.
  \end{enumerate}
\end{theo}
\begin{proof}
  These results are adapted versions of \cite{Vire02}, Theorem 3.6,
  Theorem 4.2 and Lemma 4.6 when the group $(G,+)$ is commutative and
  $H_\bullet$ is unimodular so that the element called in
  \cite{Vire02} the distinguish element of $H_{\wb 0}^*$ simplifies to
  the counit which acts trivially on $H_\ba$.  Remark that the
  unimodularity is not needed for (1) and (3).
\end{proof}
The Hopf $G$-coalgebras $\U_\bullet=\{\U_{\ba}\}_{\ba\in G}$ and
$\UH_\bullet=\{\UH_{\ba}\}_{\ba\in G}$ are pivotal with the pivot
given by the image $g_\ba$ of the pivotal element $g$ of $\U$ into
$\U_{\ba}$ (resp. into $\UH_{\ba}$).  Moreover, $\U_\bullet$ is of
finite type, thus it has a right $G$-integral.

\subsection{Discrete Fourier transforms}\label{ss:DFT}
In this subsection we recall the discrete Fourier transform from
\cite{Ha18a} that will be used to send the universal invariant into a
Hopf $G$-coalgebra of finite type. The results allow us to construct
an invariant of $3$-manifolds in the next subsection.

\newcommand{\Ri}{{\Ce_{[m]}}}
Let $V$ be the space $(\h^*)^m$ for some $m$.   Denote $Z_1,\ldots,Z_n$ as the coordinate functions of $V=(\h^*)^m$.  Recall $\Ce(Z_1,\ldots,Z_n)$ is the set of entire functions on $V$ which is a subalgebra of $(\UH)^{\wh\otimes m}$.  Given a
map $f: \C^n\rightarrow \C$,
we define $t_i(f)$ on $\C^n$ by
\begin{equation*}
t_i(f)(Z_1, \ldots,  Z_n)=f(Z_1,  \ldots, Z_i + 1, \ldots, Z_n) \text{ for }1\leq i \leq n.
\end{equation*}
Let $\Lat_\va=\{(\alpha_1,\ldots, \alpha_n)+\mathbb{Z}^n \}$ be the
lattice of $\C^n$ corresponding to
$\va=(\ba_1,\ldots, \ba_n)\in (\C/\mathbb{Z})^n$.  

 A function
$f(Z_1,\ldots, Z_n)\in \Ce(Z_1,\ldots,Z_n)$ is called \emph{$\ell$-periodic} in
$Z_i$ on the lattice $\Lat_\va$ if it satisfies
$f_{\mid_{\Lat_\va}}=t_{i}^{\ell}\bp f_{\mid_{\Lat_\va}}$.
The functions $\{\xi^{mZ_i}\}_{m\in \mathbb{Z}}^{i=1,\ldots, n}$ are
$\ell$-periodic and $\xi^{\ell Z_i}-\xi^{\ell \alpha_i}$ is zero on
$\Lat_\va$.
Let $I_\va$ be the ideal generated by
$\xi^{\ell Z_i}-\xi^{\ell \alpha_i}$ for $1 \leq i \leq n$.  For
$m\le n$, let $\Ri$ be the subring of entire functions that are
Laurent polynomial in the first $m$ variables:
$$\Ri=\Ce(Z_{m+1},\ldots,Z_n)[\xi^{\pm Z_1},\ldots, \xi^{\pm Z_m}]\subset
\Ce(Z_1,\ldots,Z_n).$$
Then an element of ${\Ri}/I_\va$ defines a
$\ell$-periodic map in all first $m$ variables on $\Lat_\va$.
\begin{Prop} \label{key} Let
  $f=f(Z_1,\ldots, Z_n)\in \Ce(Z_1,\ldots,Z_n)$ be a $\ell$-periodic
  function on $\Lat_\va$ in the first $m$ variable. Then there is an
  element $\mathcal{F}_\va(f)\in {\Ri}$ called the \emph{discrete
    Fourier transform of $f$}
    which is
  unique modulo $I_\va$, coincides with $f$ on
  $\Lat_{\ba_1,\ldots,\ba_m}\times\C^{n-m}$ and is given by
\begin{equation}\label{F alpha}
\Fou_{\va,m}(f)=\sum_{k_1,\ldots, k_m=0}^{\ell-1}a_{k_1\cdots k_m}\xi^{k_1Z_1+\cdots+k_mZ_m}.
\end{equation}
The coefficients $a_{k_1\cdots k_m}\in\Ce(Z_{m+1},\ldots,Z_n) $ (Fourier
coefficients) are determined by
\begin{equation*}
  a_{k_1\cdots k_m}=\frac{1}{\ell^m}\sum_{\tiny
  \begin{array}{c}
    i_1,\ldots, i_m=0\\\beta_j=\alpha_j+i_j
  \end{array}}^{\ell-1}
\xi^{-k_1\beta_1-\cdots-k_m\beta_m}f(\beta_1,\ldots, \beta_m,Z_{m+1},\ldots,Z_{n}).
\end{equation*}
\end{Prop}
\begin{proof}
  As proven, in \cite[Proposition 4.6]{Ha18a}, for any fixed
  $(Z_{m+1},\ldots,Z_n)\in\C^{n-m}$, the element in Equation \eqref{F
    alpha} is the unique linear combination of the elements
  $\{\xi^{k_1Z_1+\cdots+k_mZ_m}:0\le k_i<\ell\}$ which coincide with $f$
  for $(Z_1,\ldots,Z_m)\in\Lat_{\ba_1,\ldots,\ba_m}$.
\end{proof}

The complex analytic Nullstellensatz Theorem from the work of
H. Cartan {\cite{Ca1953}}
implies that $I_\va$ is precisely the ideal of entire functions
that vanish on $\Lat_\va$.  Here we give an elementary proof that was
shown to us by David E. Speyer:
\begin{Prop}\label{P:Nullstellensatz}
  Let $f\in\Ce(Z_1,\ldots,Z_n)$ be zero on
  $\Lat_{\ba_1,\ldots,\ba_m}\times\C^{n-m}$ then $f\in I_\va$.
\end{Prop}
To prove this proposition, we need the following lemma:
\begin{Lem}
  Let $\{\vp_k\in \Ce(Z_2,\ldots,Z_n)\}_{k\in\Z}$ be a family of entire functions.  Then there exists
  $\vp\in\Ce(Z_1,\ldots,Z_n)$ such that 
  $$\quad\vp_k=\vp(k,Z_2,\ldots,Z_n)$$
  for all $k\in\Z$.
\end{Lem}
\begin{proof}
  Let $\vp(Z_1,\ldots,Z_n)$ be of the form
  \[\vp_0(y) + \sum_{k=1}^{\infty} \frac{\prod_{m=1}^{k-1} (Z_1^2-m^2)}
    {\prod_{m=1}^{k-1} (k^2-m^2)} \left( a_{-k} \frac{k-Z_1}{2k} + a_k
      \frac{k+Z_1}{2k} \right) \left( \frac{Z_1^2}{k^2}
    \right)^{C_k}\] for some entire functions
  $a_{\pm k}\in\Ce(Z_2,\ldots,Z_n)$ and some positive integers $C_k$
  which we compute inductively as we will now describe.  Note that,
  for $k>p$, the product $\prod_{m=1}^{k-1}$ vanishes at $Z_1=\pm p$,
  so terms beyond the $p$-th term don't affect the value of the sum at
  $Z_1=\pm p$. Also, at $Z_1=\pm k$, the $k$\ith term of the series
  reduces to $a_{\pm k}$.  So, if we have chosen $a_m$, $b_m$ and
  $C_m$ for $m<k$, there is a unique choice of $a_{\pm k}$ which will
  make the sum correct at $Z_1=\pm k$.  Make that choice for
  $a_{\pm k}$. Then choose $C_k$ large enough that the $k$\ith summand
  has module less than $2^{-k}$ on the compact
  $\{|Z_i|\le k-1,i=1\cdots n\}$ for all $i={1\cdots n}$. (Since
  $Z_1^2/k^2\le(1-1/k)^2$ on this compact, taking $C_k$ large enough
  will work). Thus, the sum will be uniformly convergent on compact
  subsets, and will thus give an entire function.
\end{proof}
\begin{proof}[Proof of Proposition \ref{P:Nullstellensatz}]
  First remark that the entire function
  $$\bp{\xi^{\ell Z}-\xi^{\ell \alpha}}\in\Ce(Z)$$ only has simple zeros
  at integers.  We prove the Proposition by induction on $n$: for any
  $k\in\Z$, $f(\alpha_1+k,Z_2,\ldots,Z_n)$ is zero on the lattice
  $\Lambda_{\alpha_2,\ldots,\alpha_n}$.  Then assume this function
  belongs to the ideal generated by
  $\bp{\xi^{\ell Z_i}-\xi^{\ell \alpha_i}}_{i=2\cdots n}$.  Then we
  can write
  $$f(\alpha_1+k,Z_2,\ldots,Z_n)=\sum_{i=2}^n\bp{\xi^{\ell Z_i}-\xi^{\ell \alpha_i}}
  \vp_{i,k}(Z_2,\ldots,Z_n)$$ then by the above Lemma
  applied to the family $\bp{\vp_{i,k}}_{k\in\Z}$, we obtain entire
  functions $\vp_i\in\Ce(Z_1,\ldots,Z_n)$ such that for
  $$\vp=\sum_{i=2}^n\bp{\xi^{\ell Z_i}-\xi^{\ell \alpha_i}}\vp_{i},$$
  $f-\vp$ vanishes on $(\alpha_1+\Z)\times\C^{n-1}$.  But then
  $\psi=\dfrac{f-\vp}{\xi^{\ell Z_1}-\xi^{\ell \alpha_1}}$ is an
  entire function and
  $f=\vp+\bp{\xi^{\ell Z_1}-\xi^{\ell \alpha_1}}\psi\in I_\va$.
\end{proof}
The following proposition says,
modulo the ideal $I_\va$, the
$\ell$-periodic
elements are equal to their Fourier transform and power quadratic
elements reduce to power linear.
Here we use the notation $z_\bullet=(Z_1,\ldots,Z_n)$ and say a \emph{linear element} is any complex linear combination of the $Z_i$ while a \emph{quadratic element} is a homogeneous degree $2$ polynomial in the variables $z_\bullet$ with integer coefficients.
\begin{Prop}\label{P:q=l}
  If $f\in\Ce(Z_1,\ldots,Z_n)$ is periodic in its first $m$ variables
  then $f=\mathcal{F}_{\va,m}(f)$ modulo
  $I_\va$.  Consequently, if $q(z_\bullet)$ is a quadratic element and
  $l$ is the linear element
  $l(z_\bullet)=\wt q(\tilde \alpha,z_\bullet)$ for any
  $\wt \alpha=(\alpha_1,\ldots,\alpha_n)\in\Lambda_\va$, then
  $\xi^{q(z_\bullet)}\in\xi^{l(z_\bullet)}\C[\xi^{Z_1},\ldots,\xi^{Z_n}]$
  modulo $I_\va$.
\end{Prop}

\begin{proof}

As $f-\mathcal{F}_{\va,m}(f)$ is zero on $\Lat_\va$ then Proposition \ref{P:Nullstellensatz} implies it belongs to $I_\va$. To prove the second statement of the proposition,
let $e_i$ be the vector $(0,\ldots,1,\ldots,0)$ then
$t_i^\ell(l(z_\bullet))-l(z_\bullet)=\wt q(\wt \alpha,\ell
e_i)=\ell\wt q(e_i,\wt \alpha)$ while
$t_i^\ell(q(z_\bullet))-q(z_\bullet)=\wt q(\ell e_i,z_\bullet)+ q(\ell
e_i)$ which on $\Lambda_\va$ takes values in
$\ell\wt q(e_i,\wt \alpha)+\ell\Z$.  So
$\xi^{q(z_\bullet)-l(z_\bullet)}$ is $\ell$-periodic in all its
variable on $\Lat_\va$ so, modulo $I_\va$, it is equal to
$\mathcal{F}_{\va,n}(\xi^{q(z_\bullet)-l(z_\bullet)})$ and
$\xi^{q(z_\bullet)}=\xi^{l(z_\bullet)}\mathcal{F}_{\va,n}(\xi^{q(z_\bullet)-l(z_\bullet)})$.
\end{proof}

\newcommand{\LU}{\mathcal L\U} Next we would like to apply the above
results to the context of the topological unrolled quantum groups.
For $\va=(\ba_1,\ldots,\ba_n)\in G^{n}$ we denote by $I_\va^H$ the ideal of $\bp{\UH}^{\wh\otimes n}$ generated by the $rn$
elements $\xi^{\ell H_{i,j}}-\xi^{\ell \ba_{i,j}}$
where $\ba_{i,j}=\ba_j(H_i)\in\C/\Z$ is the value of $H_{i,j}$ at $\va$.
Recall in Subsection \ref{SS:PowerElements} we define the power linear
$\mathcal L_n $ and quadratic elements $\mathcal Q_n$.  In particular,
recall $\mathcal L_n=\{\xi^H : {H\in\H^{(n)}} \}$ where the elements
$\{H_{i,j}\}_{i=1,\ldots,\rk,\;j=1,\ldots, n}$ are a basis of $\H^{(n)}$.
Let $\LU_\ba$ be the image of $\mathcal L_1\U$ in the quotient $\UH_\ba$.
\begin{Prop}\label{P:quad-lin-modI}
  Let $\va=(\ba_1,\ldots,\ba_m,\ldots,\ba_{m+n})\in G^{m+n}$ and let
  $h_\bullet$ denotes the set of $r(m+n)$ variables
 $\{H_{i,j}\}_{i=1,\ldots,\rk,\;j=1,\ldots, m+n}$.
 Let
  $$u\in \mathcal Q_{m+n}\mathcal L_{m+n}\U^{\otimes
    m+n}\subset\Ce\bp{h_\bullet}\otimes W^{\otimes m+n} \simeq(\UH)^{\wh\otimes m+n} 
    $$ be an
  $\ell$-periodic  element in the first $rm$ variables
       $\{H_{i,j}\}_{i=1,\ldots,\rk,\;j=1,\ldots, m}$
      on the
  lattice $\Lambda_\va$.
Then
  $$\pi(u)\in \U_{\ba_1}\otimes\cdots\otimes\U_{\ba_m}\otimes
  \LU_{\ba_{m+1}}\otimes\cdots\otimes\LU_{\ba_{m+n}}$$
  where $\pi:(\UH)^{\wh\otimes m+n}\to (\UH)^{\wh\otimes m+n}/ I_\va^H$  is the projection.  
\end{Prop}
\begin{proof}
  Let us write $u=\sum_ku_k(h_\bullet)w_k$ in a basis $(w_k)_{k}$ of
  $W^{\otimes m+n}$.  By the previous proposition, each $u_k$ is equal
  modulo $I_\va^H$ to a Laurent polynomial in $\C[\xi^{\pm H_{i,j}}]$
  times a power linear element in the last $rn$ variables
  $\xi^{\sum\beta_{i,j}H_{i,j}}$ where the sum range for
  ${i=1\cdots r,\ j=m+1\cdots m+n}$. Then for
  $\gamma_j(H_1,\ldots,H_r)=\sum_{i=1}^r\beta_{i,j}H_{i}$, we have:
  \\\hspace*{2ex}
  $\pi(u)\in \U_{\ba_1}\otimes\cdots\otimes\U_{\ba_m}\otimes
  \xi^{\gamma_{m+1}}\U_{\ba_m+1}\otimes\cdots\otimes\xi^{\gamma_{m+n}}\U_{\ba_m+n}$.
\end{proof}

\subsection{The symmetrized $G$-integral}\label{SS:symmetrizedGintegral}
Recall that we denote by $h=(H_1,\ldots,H_r)$ and
$h_j=(H_{1,j},\ldots,H_{r,j})$ so for example,
$\HH=\xi^{\Q(h_1,h_2)}$. Fix $\alpha\in\H$ and let $\ba\in G$ be its
class modulo $\Lambda$. Then we have
\begin{align*}
 \wt\Q((-\alpha,\alpha),(h_1,h_2))  &=  \Q(h_1-\alpha,h_2+\alpha)-\Q(h_1,h_2)+\Q(\alpha,\alpha) \\
    &  =\Q(\alpha,h_1)-\Q(\alpha,h_2).
\end{align*}
So
$\xi^{\wt\Q((-\alpha,\alpha),(h_1,h_2))}= L(h_1-h_2)$ where
$L(h)=\xi^{\Q(\alpha,h)}$ and Proposition~\ref{P:q=l} implies that
$L(h_2-h_1)\HH $ is in $\C[\xi^{h_1},\xi^{h_2}]$ modulo $I_{(-\alpha,\alpha)}$.  In particular, $L(h_2-h_1)\HH $ is
$\ell$-periodic in all its variable on $\Lambda_{(-\ba,\ba)}$.  Let
$\HH_\alpha\in\U_{-\ba}\otimes\U_{\ba}$ be its Fourier transform on
this lattice.  Hence $\HH=L_\alpha\HH_\alpha$ modulo
$I_{(-\alpha,\alpha)}$ where $L_\alpha=L(h_1-h_2)$.
\begin{Prop}
The finite dimensional  Hopf algebra  $\U_{\wb0}$ is ribbon  with
  universal $R$-matrix
  $$\mathcal R_{\wb 0}=\HH_{0}\check {\mathcal R}\in\U_{\wb 0}\otimes\U_{\wb 0}$$ and twist
  $\theta_{\wb 0}=g^{-1}\left( m\circ
  (S^2\otimes\Id)(\RR_{\wb0})\right)=m(\HH_{0})\sum_iy_igx_iB_{|x_i|}\in\U_{\wb 0}$.  
\end{Prop}
\begin{proof}
  For $\alpha=0$, one has $L_\alpha=1$ and the image of the
  $R$-matrix in $\UH_{\wb 0}\wh\otimes\UH_{\wb
    0}$ belongs in fact to $\U_{\wb 0}\otimes\U_{\wb
    0}$.  So ${\mathcal R}=\HH_{0}\check {\mathcal
    R}$ modulo the ideal $I^H_{\vec0}$ and ${\mathcal
    R}_{\wb 0}$ inherits all properties of the
    $R$-matrix
    ${\mathcal R}$. Finally the formula for the twist comes from
  $\theta=gu^{-1}$ and Equation \eqref{eq:ui}.
\end{proof}
Both the quotient morphism $\UH\to\UH_{\wb 0}$ and the inclusion morphism
$\U_{\wb 0}\to\UH_{\wb 0}$ are morphisms of ribbon Hopf algebra.

To build a invariant of 3-manifold, we will need the
following axiom that will fix the normalization of the cointegral
$\coint$ and integral $\lambda^R$:
\begin{axiom}\label{Ax:non-degenerate}
  Let $\delta=\lambda^R_{\wb0}(\theta_{\wb0})$ and
  $\wb\delta=\lambda^R_{\wb0}(\theta_{\wb0}^{-1})$.  We assume:
  $\delta\wb\delta=1.$
\end{axiom}
The true assumption in this axiom is the ``non-degeneracy'' condition
$\delta\wb\delta\neq0$.  Indeed, we have
\begin{Prop}
  \label{P:norm} If $\delta\wb\delta\neq0$ then 
  there exists a choice of $\coint$ so that Axiom \ref{Ax:non-degenerate} is satisfied. 
\end{Prop}
\begin{proof}
  If $\delta\wb\delta\neq0$, it has a square root $\rho\in\C$ such
  that replacing $\coint$ with $\rho\coint$ will lead to a new right
  integral equal to
  $\frac1\rho\lambda^R$ 
  for which Axiom
  \ref{Ax:non-degenerate} is satisfied.
\end{proof}
For $\ba\in G$
 and $\nu\in\Lambda$, let
$\pr_\nu^\ba:\U_\ba\to\U_\ba$ be the projection on the subspace of
weight $\nu$ of $\U_\ba$.
\begin{Lem}\label{L:int-weight0}
  The cointegral $\coint\in \U_{\wb0}$ and the right $G$-integral
  $\{\lambda^R_\ba\}_{\ba\in G}$ are of degree $0$, meaning  they satisfy: 
  $\coint=\pr^{\wb0}_{0}(\coint)$ and
  \begin{equation}
    \label{eq:int_deg0}
   \lambda^R_\ba=\lambda^R_\ba\circ\pr^\ba_{0} 
  \end{equation}
  for all $ \ba\in G$.
\end{Lem}
\begin{proof}
  For any $\nu\in\Lambda$, let $\coint_\nu=\pr^{\wb0}_{\nu}(\coint)$.
  Now for any homogeneous $x\in\U_{\wb0}$ of weight $|x|$,
  $$x\coint_\nu=\pr^{\wb0}_{|x|+\nu}(x\coint)=\ve(x)\pr^{\wb0}_{|x|+\nu}(\coint)
  =\ve(x)\coint_\nu,$$ where the last equality is because $\ve(x)=0$
  unless $|x|=0$. Hence $\coint_\nu$ is a left cointegral in
  $\U_{\wb0}$.  Since the space on left cointegral is one dimensional
  (see \cite[Theorem 10.2.2]{Ra12}), we get that $\coint_\nu=\coint$
  or $\coint_\nu=0$.  Finally, as $\coint=\sum_\nu\coint_\nu$, there
  exists a unique $\nu_0\in\Lambda$ such that
  $\coint=\coint_{\nu_0}$.

  Similarly, for any $\ba\in G$ and for any $\nu\in\Lambda$, let
  $\lambda^\nu_\ba=\lambda^R_\ba\circ\pr^\ba_\nu$.  Then let
  $x\in\U_{\ba+\bb}$ of weight $|x|$.
  Then
  we have $\lambda^\nu_\ba(x_{(1)})x_{(2)}=\lambda^R_\ba(\pr^{\ba}_\nu(x_{(1)}))x_{(2)}=\lambda^R_\ba(x_{(1)})\pr^{\bb}_{|x|-\nu}(x_{(2)})=\lambda^R_{\ba+\bb}(x)\pr^{\bb}_{|x|-\nu}(1_{\U_\bb})
  =\lambda^\nu_{\ba+\bb}(x)1_{\U_\bb}$.
   Thus $\lambda^\nu$ is a
  right $G$-integral. Once again, since
  the space of right $G$-integral is one dimensional, one has
  $\lambda^\nu=\lambda^R$ or $\lambda^\nu=0$ and since
  $\lambda^R_{\wb0}(\coint)=1$, we have $\lambda^R=\lambda^{\nu_0}$.

The fact that $\nu_0=0$ is now the consequence of Axiom
  \ref{Ax:non-degenerate} since $\theta_{\wb0}$ has weight $0$ and
  $\lambda^R(\theta_{\wb0})\neq0$.
\end{proof}
The next property of $\U_\bullet$ is called {\em unibalanced} in
\cite{BBG17,Ha18}:
\begin{Prop}\label{P:unibalanced}
  The distinguished element of $\U_\bullet$ is the square of its pivot:
  $$\gamma_\ba =g_{\ba}^{2} \text{ for all } \ba\in G. $$
\end{Prop}
To prove the proposition, we use the following Lemma (see also
\cite{Vire02}):
\begin{Lem}
  Let
  $\coint_{(1)}^{-\ba}\otimes\coint_{(2)}^{\ba}
  =\Delta_{-\ba,\ba}(\coint)\in\U_{-\ba}\otimes\U_{\ba}$.  Then   
  $$\lambda^R_{-\ba}(\coint_{(1)}^{-\ba})\coint_{(2)}^{\ba}=1\text{ and }
  \lambda^R_{-\ba}(\coint_{(2)}^{-\ba})\coint_{(1)}^{\ba}=\gamma_{\ba}.$$
  Furthermore for any $x\in\U_{-\ba}$, one has
  $$x\coint_{(1)}^{-\ba}\otimes\coint_{(2)}^{\ba}=
  \coint_{(1)}^{-\ba}\otimes S^{-1}_{-\ba}(x)\coint_{(2)}^{\ba}\text{
    and } \coint_{(1)}^{-\ba}x\otimes\coint_{(2)}^{\ba}=
  \coint_{(1)}^{-\ba}\otimes \coint_{(2)}^{\ba}S_{-\ba}(x).$$
\end{Lem}
\begin{proof}
  First
  $(\lambda^R\otimes\Id)\bp{\Delta_{-\ba,\ba}(\coint)}=\lambda^R(\coint)1=1$
  and since $\lambda^L=\lambda^R(\gamma\cdot)$ is a left $G$-integral,
  $(\Id\otimes\lambda^R)\bp{\Delta_{-\ba,\ba}(\coint)}=
  (\gamma\otimes\lambda^L)\bp{\Delta_{-\ba,\ba}(\gamma^{-1}\coint)}=\gamma.$
  $$\coint_{(1)}^{-\ba}x\otimes\coint_{(2)}^{\ba}
  =\coint_{(1)}^{-\ba}x_{(1)}\otimes\coint_{(2)}^{\ba}x_{(2)}S_{-\ba}(x_{(3)})=$$$$
  \Delta_{-\ba,\ba}(\coint x_{(1)})1\otimes S_{-\ba}(x_{(2)})=
  \coint_{(1)}^{-\ba}\otimes \coint_{(2)}^{\ba}S_{-\ba}(x).
  $$
  Similarly,
  $$x\coint_{(1)}^{-\ba}\otimes\coint_{(2)}^{\ba}
  =x_{(1)}\coint_{(1)}^{-\ba}\otimes S^{-1}_{-\ba}(x_{(3)})x_{(2)}\coint_{(2)}^{\ba}=$$$$
  \ \quad1\otimes S^{-1}_{-\ba}(x_{(2)})\Delta_{-\ba,\ba}(x_{(1)}\coint)
  =\coint_{(1)}^{-\ba}\otimes S^{-1}_{-\ba}(x)\coint_{(2)}^{\ba}.$$
\end{proof}
\begin{proof}[Proof of Proposition \ref{P:unibalanced}]
We have  $$\gamma_{\ba}=\lambda^R_{-\ba}(\coint_{(2)}^{-\ba})\coint_{(1)}^{\ba}
  =(\lambda^R_{-\ba}\otimes\Id)(\mathcal R\Delta_{-\ba,\ba}(\coint)\mathcal R^{-1})$$
 and
  $$\mathcal R\Delta_{-\ba,\ba}(\coint)\mathcal R^{-1}= L_\alpha\HH_\alpha\check{\mathcal
    R}\Delta_{-\ba,\ba}(\coint)(L_\alpha\HH_\alpha\check{\mathcal R})^{-1}.$$
  Now we write the linear element $L_\alpha$ given at the beginning of
  this subsection
  as $L_\alpha=L(h)\otimes L(-h)=L(h_1)L(h_2)^{-1}$ where
  $L(h)=\xi^{\Q(\alpha,h)}\in\mathcal L_1$.  Consider the application
  $f_+:\UH_{-\ba}\otimes\UH_{\ba}\to\UH_{\ba}$ given by
  $f_+(x\otimes y)=S(x)y$ and
  $f_-:\UH_{-\ba}\otimes\UH_{\ba}\to\UH_{\ba}$ given by
  $f_-(x\otimes y)=yS^{-1}(x)$.  Remark that
  $f_-(\mathcal R)=f_-((S\otimes S)(\mathcal R))=u=g\theta^{-1}$ and
  $f_+(\mathcal R^{-1})=f_+((S\otimes \Id)(\mathcal
  R))=S(u^{-1})=g\theta$.

  Then from Lemma above, one has:
  $$\mathcal R\Delta_{-\ba,\ba}(\coint)\mathcal R^{-1}=
  L_\alpha(1\otimes f_-\bp{\HH_\alpha\check{\mathcal R}})\Delta_{-\ba,\ba}(\coint)
  (1\otimes f_+(\check{\mathcal R}^{-1}\HH_\alpha^{-1}))L_\alpha^{-1}$$
  $$=L_\alpha(1\otimes f_-\bp{L^{-1}_\alpha{\mathcal R}})\Delta_{-\ba,\ba}(\coint)
  (1\otimes f_+({\mathcal R}^{-1}L_\alpha))L_\alpha^{-1}$$
  $$=L_\alpha L(h_2)(1\otimes f_-\bp{{\mathcal R}}) L(h_2)\Delta_{-\ba,\ba}(\coint)
  L(-h_2)( 1\otimes f_+({\mathcal R}^{-1}))L(-h_2)L_\alpha^{-1}$$
  $$=(L(h)\otimes g\theta^{-1}L(h)) \Delta_{-\ba,\ba}(\coint)
  (L(-h)\otimes L(-h)g\theta).$$ Finally, since
  $\lambda^R$ is homogeneous of weight $0$, only terms with
  $\coint_{(1)}$ commuting with $L(\pm h_1)$ will contribute so:
  $$\lambda^R_{-\ba}\otimes\Id(\mathcal R\Delta_{-\ba,\ba}(\coint)\mathcal R^{-1})=\lambda^R_{-\ba}\otimes\Id((1\otimes g\theta^{-1}L(h)) \Delta_{-\ba,\ba}(\coint)
  (1\otimes L(-h)g\theta))$$
  $$=f_-\bp{{\mathcal R}}L(h_2) L(-h_2)f_+({\mathcal R}^{-1})=uS(u^{-1})=g^2.$$
\end{proof}
The following terminology first appeared in \cite{BBG17} in the non-graded case, then in \cite{Ha18} for the example of special linear superalgebra $\mathfrak{sl}(2|1)$.
\begin{Def}
  The {\em symmetrised $G$-integral} of $H_\bullet$ is the family of
  linear forms $\si=\{\si_{\ba}\in \U_{\ba}^*\}_{\ba\in G}$ defined by
$$\si_{\ba}(x)= \lambda_{\ba}(g_{\ba}x) \text{ for } x\in \U_{\ba}.$$
\end{Def}

The symmetrized $G$-integral $\si$ has the properties of a {\em
  $G$-trace} in \cite{Vire02}. More precisely, we have the properties
below:
\begin{Prop}
  These equalities hold
\begin{align}
  \bp{\si_{\wb\alpha}\otimes g_{\wb\beta}}\Delta_{\wb\alpha,\wb\beta}(x)
  &=\si_{\wb\alpha+\wb\beta}(x)1_{\wb\beta} \text{ for } x\in \U_{\wb\alpha+\wb\beta},\label{eq:rightsi}\\
  \bp{g^{-1}_{\wb\alpha}\otimes \si_{\wb\beta}}\Delta_{\wb\alpha,\wb\beta}(x)
  &=\si_{\wb\alpha+\wb\beta}(x)1_{\wb\beta} \text{ for } x\in \U_{\wb\alpha+\wb\beta},\\
  \si_{\ba}(xy)&=\si_{\ba}(yx)\ \text{for}\ x, y\in \U_{\ba},\label{eq:mucyc}\\
  \si_{-\ba}\bp{S^{\pm1}_{\pm\ba}(x)}&=\si_{\ba}(x)\ \text{for}\ x\in \U_{\ba}.\label{eq:muS}\\
  \si_{\ba}(x)&=\si_{\ba}(\pr^\ba_0(x))\ \text{for}\ x\in \U_{\ba}.\label{eq:muW0}
\end{align}
\end{Prop}
\begin{proof}
  These identities are directly obtained from Theorem
  \ref{Th:right-int} and Lemma \ref{L:int-weight0} by the change
  of variable $x\mapsto g_{\ba+\bb}x$ and by using the properties of
  the pivot.
\end{proof}
The following proposition implies Examples \ref{ex1} and \ref{ex2}  satisfy 
Axiom~\ref{ax:unicorn}:
\begin{Prop}
  Suppose that $\Lambda'$ is a rank $r$ sub-lattice of $\Lambda^*$ as
  in Remark \ref{R:sublattice} and let
  $\U'_{\wb0}=\U_{\Lambda'}/(I_{\wb0}\cap\U_{\Lambda'})$.  If $\U'_{\wb0}$ is
  unimodular with a weight $0$ cointegral, then so is $\U_{\wb0}$.
\end{Prop}
\begin{proof}
  Since $\Lambda'$ has rank $r$, the quotient $\Lambda/\Lambda'$ is
  finite with $d$ elements. Let $(k_1,\ldots,k_d)\in\Lambda^d$ be
  representatives of these elements, then $\U'_{\wb0}\subset\U_{\wb0}$ and as
  $\C$-vector space, $\U_{\wb0}=\bigoplus_i k_i\U'_{\wb0}$.  Let
  $k_\coint=\sum_ik_i$. If $\coint'$ is the cointegral of $\U'_{\wb0}$ then
  one easily check that $\coint=k_\coint\coint'=\coint'k_\coint$ is a
  two side cointegral for $\U_{\wb0}$.  In particular, if $k\in\Lambda$,
  then $k$ will permute elements of $\Lambda/\Lambda'$ and there
  exists $K'_i\in\Lambda'$ such that $kk_\coint=\sum_ik_iK'_i$ thus
  $\coint k=k\coint=\sum_ik_i\ve(K_i')\coint'=\ve(k)\coint$.
\end{proof}
Since the $\Lambda'$ version of the quantum groups considered in
Examples \ref{ex1} and \ref{ex2} are unimodular (see \cite{DGP2}), we
have:
\begin{Cor}\label{C:Axiom3True}
  The quantum groups $\U_{\wb0}$ associated to Examples \ref{ex1}
  and \ref{ex2} are unimodular.  Thus, the quantum groups $\U$ of Examples \ref{ex1}
  and \ref{ex2} satisfy Axiom \ref{ax:unicorn}.  
\end{Cor}
\begin{Exem}\label{ex1c}
Here we build upon Examples \ref{ex1} (which is continued in Examples \ref{ex1b} and \ref{ex12d}). As shown in 
  \cite[section 5.1]{Ha18},  the symmetrized integral does not
  depend on $\ba$ in the PBW basis: There exists $\eta\in\C$ such that
  it is given for any $0\le i,j\le\ell'-1$ and any $0\le k\le\ell-1$ by
  $$\si_\ba(E^i F^j K^k ) = \eta\delta_{ i,\ell'-1}\delta_{ j,\ell'-1}\delta_{ k,0}.$$
  Furthermore, Equation \eqref{eq:mucyc} and the commutation relations
  for $K$ imply that the same formula holds for the product of $E^i$,
  $F^j$ and $K^k$ in any order.
  Using the Fourier transform, it can be rewritten for any $\ell$-periodic
  element $x=\sum_{i,j}E^iF^j\vp_{ij}(H)$
  as
  $$\si_\ba(x)=\frac\eta\ell\sum_{k=0}^{\ell-1}\vp_{\ell'-1,\ell'-1}(\alpha+k).$$
  We will apply this to compute
  $\lambda_{\wb0}(\theta_{\wb0}^{\pm1})$:
  In
  the formula of $\check \RR$, only the last term will contribute and
  this term is $c\,E^{\ell'-1}\otimes F^{\ell'-1}$ where
$$c=\frac{(\xi-\xi^{-1})^{\ell'-1}}{\qN{\ell'-1;\xi^{-2}}!}=\frac{-(-1)^\ell\xi}{\ell'}(\xi-\xi^{-1})^{2(\ell'-1)}.
$$
   Thus
  $$\lambda_{\wb0}(\theta_{\wb0})=\si_{\wb0}(g^{-2}S(u)^{-1})=
  \si_{\wb0}(cK^{-2}\xi^{2H^2}E^{\ell'-1}K^{2\ell'-2}F^{\ell'-1}K^{2-2\ell'})
  $$
  $$=c\xi^{-2}\frac\eta\ell\sum_{k=0}^{\ell-1}\xi^{2k^2-2k}.$$
  This last Gauss sum is $0$ if and only if $\ell\in8\N$ (see for
  example \cite{FcNgBp14}).  Similarly, using that
  $u=\sum_iS^2(b_i)S(a_i)=m(\HH^{-1})\sum_iS^2(y_i)S(x_i)$, one
  has
  $$\lambda_{\wb0}(\theta_{\wb0}^{-1})=\si_{\wb0}(g^{-2}u)
  =\si_{\wb0}\bp{cg^{-2}m(\HH^{-1})S^2(F^{\ell'-1})S(E^{\ell'-1})}$$
  $$=\si_{\wb0}\bp{cg^{-2}\xi^{-2H^2}\xi^{2-2\ell'}F^{\ell'-1}\xi^{-2}E^{\ell'-1}K^{2-2\ell'}}
  =c\frac\eta\ell\sum_{k=0}^{\ell-1}\xi^{2k^2-2k}.$$ Thus if
  $\ell\notin8\N$, one can choose
  $\eta=\frac{\xi\ell }c(\sum_{k=0}^{\ell-1}\xi^{2k^2-2k})^{-1}$ and
  Axiom \ref{Ax:non-degenerate} holds with $\lambda_{\wb0}(\theta_{\wb0}^{-1})=\xi=\lambda_{\wb0}(\theta_{\wb0})^{-1}$.
\end{Exem}
\begin{Exem}\label{ex2c}
Here we build upon Example \ref{ex2} (which is continued in Examples \ref{ex2b} and \ref{ex12d}) 
and sketch a proof that
  Axiom \ref{Ax:non-degenerate} hold for these examples.  To simplify notation in this example we set $\U'=\U_{\Lambda'}$.    First, the restriction of right
  integral of $\U_{\wb0}$ to $\U'_{\wb0}=\U'/(I_{\wb0}\cap\U')$ is a right integral for $\U'_{\wb0}$.  Second, 
  $R_{\wb0}$ and $\theta_{\wb0}$ belongs to the subalgebra
  $\U'_{\wb0}$. Hence $\delta$ can be computed in the
  small quantum group $\U'_{\wb0}$.  Finally, in \cite{LO17, Ly95} it is shown that $\U'_{\wb0}$ is ribbon and factorizable 
  thus \cite[Corollary 3.9]{DGP} implies it is non
  degenerate and $\delta\neq0$.
\end{Exem}

\subsection{Modified trace on $\Proj(\cat^H)$}
Here we discuss modified traces, or m-trace for short, defined in \cite{GPV13, GKP18}.   
If $\cat$ is a linear pivotal category, let $\Proj(\cat)$ be the ideal of projective objects of $\cat$. The \textit{right partial trace} of an endomorphism $f \in \End_{\cat}(V \otimes V')$ is the endomorphism $\ptr(f) \in \End_{\cat}(V)$ given by
\[
 \ptr(f) := (\Id_V \otimes \rev_{V'}) \circ (f \otimes \Id_{V'^*}) \circ (\Id_V \otimes \lcoev_{V'}).
\]
Following \cite{GPV13,GKP18},
a   \textit{right m-trace $\tv$ on $\Proj(\cat)$} is a family of linear maps
$\{ \tv_V : \End_{\cat}(V) \rightarrow \Bbbk \mid V \in \Proj(\cat) \}$ satisfying:
\begin{enumerate}
 \item \textit{Cyclicity}: $\tv_{V}(f' \circ f) = \tv_{V'}(f \circ f')$ for all objects $V,V' \in \Proj(\cat)$ and for all morphisms $f \in \Hom_{\cat}(V,V')$ and $f' \in \Hom_{\cat}(V',V)$;
 \item \textit{Partial trace}: $\tv_{V \otimes V'} (f) = \tv_V(\ptr(f))$ for all objects $V \in \Proj(\cat)$ and $V' \in \cat$ and for every morphism $f \in \End_{\cat}(V \otimes V')$.
\end{enumerate}
Similarly, using the left partial trace one can define a left m-trace.

Now let $\cat$ be the $\C$-linear category
$\bigoplus_{\ba\in G}\cat_{\ba}$ in which $\cat_{\ba}$ is the category
$\U_{\ba}$-mod of finite dimensional left $\U_{\ba}$-modules for
$\ba\in G$.  Since $\U$ is unimodular then there exist a right m-trace
$\tv$ on $\Proj(\cat)$.  Moreover, since $\U_\bullet$ is unibalanced
(i.e. Proposition \ref{P:unibalanced}), then \cite[Theorem 1.1]{Ha18}
implies the right m-trace $\tv$ is also a left m-trace and it is
determined by
\begin{equation}
  \label{eq:t}
  \tv_{\U_\ba}(R_x)=\si_\ba(x)
\end{equation}
where $R_x$ is the right multiplication by $x\in\U_\ba$ seen as a
morphism of left $\U_{\ba}$-module.

Recall that $\cat^H=\bigoplus_{\ba\in G}\cat_{\ba}^H$ is the category of weight modules over $\UH$, see Section \ref{Section unrolled qt group}. Since $\U\subset\UH$, we can
consider the forgetful functor obtain by restriction
$\ResF: \cat_{\ba}^H \rightarrow \cat_{\ba},\ V^H \mapsto
V=\ResF(V^H)$ and
$f\in \Hom_{\cat^H}(V^H,W^H)\mapsto \ResF(f)\in
\Hom_{\cat}(V,W)$. Here $V^H$ and $V$ have the same underlying vector
space, and $V$ is obtained from $V^H$ by forgetting the action of
$\H$.

Call $\Proj(\cat^H)$ the ideal of projective modules of $\cat^H$.
We now need to assume 
\begin{axiom}\label{A:proj}
  There exists a projective module $P_0$ of $\cat^H$ with $\ResF(P_0)$
  projective in $\cat$.
\end{axiom}
Axiom \ref{A:proj} implies that 
\begin{Prop}
  For any projective module $P$ of $\cat^H$, $\ResF(P)$ is projective
  module module of $\cat$.
\end{Prop}
\begin{proof}
  Since $P$ is projective it belongs to the ideal generated by $P_0$.
  But then $\ResF(P)$ belongs to the ideal generated by $\ResF(P_0)$
  which is the subcategory of projective objects of $\cat$.
\end{proof}
For $V^H\in \Proj(\cat^H)$, we define a linear map
$$\tv_{V^H}: \End_{\cat^H}(V^H)\rightarrow \C, f\mapsto \tv_{V^H}(f):=\tv_{\ResF(V^H)}\bp{\ResF(f)}. $$
Since the forgetful functor is
  pivotal and it commutes with the right and left partial traces and we have
\begin{Prop}
  The family $\{\tv_{V^H}\}_{V^H\in \Proj(\cat^H)}$ determines a
  right and left m-trace on $\Proj(\cat^H)$.
\end{Prop}
We finish by a proposition used to complete Example \ref{ex2b} and below:
\begin{Prop}\label{P:U inject in CH}
 Suppose there exists a  dense open subset $O$ of $G$ such that the categories $\cat_\ba$ and $\cat^H_\ba$ are semi-simple for any $\ba\in O$ (i.e.\ $\cat$ and $\cat^H$ are generically semi-simple).  Also, suppose
 that for $\ba\in O$, every simple module in $\cat_\ba$ is the image of a simple module in  $\cat^H_\ba$.
   Then
   \begin{enumerate}
   \item If $x\in\UH$ satifies $\rho_V(x)=0$ for all $ V\in\cat^H$ 
     then $x=0$.
   \item Axiom \ref{A:proj} holds.
   \end{enumerate}
\end{Prop}
\begin{proof}
  Consider a non zero $x(h)=\sum x_if_i(h)\in\wh\UH$ seen as a
  $W$-valued entire function. Since $x$ is not zero, there exists
  $\alpha\in\H^*$ such that $x(\alpha)\neq 0$ and $\cat_\alpha$ is
  semi-simple.  Then
  $\Fou_\alpha(x)\neq0\in\U_\ba\simeq\bigoplus_i\End_\C(V_i)$.
  Lifting one of the $V_i$ to $\cat^H$, one can found a simple module
  $V\in\cat^H$ such that $\rho_V(\Fou_\alpha(x))\neq0$.  Let $v\in V$
  be a homogeneous vector of weight $\alpha'$ such that
  $\Fou_\alpha(x)v\neq0$.  $\Fou_\alpha(x) $ and $x$ have the same
  values when $h$ belongs to
  $\{\alpha+\sum_ik_i\alpha_i,\,0\le k_i\le\ell-1\}$ thus there exists
  $\alpha''\in\H^*$ such that
  $\alpha''-\alpha'=\ell\lambda\in\ell\Lambda$ and
  $\Fou_\alpha(x)(\alpha'')=x(\alpha'')$.  Now let $\sigma=\C$ endowed
  with the action of $\UH$ given by
  $\forall u\in\U,\rho_\sigma(u)=\ve(u)\Id_\sigma$ and
  $\rho_\sigma(H_i)=\ell\lambda(H_i)\Id_\sigma$.  Then the vector
  $v'=v\otimes1\in V'=V\otimes\sigma$ satifies
  $\rho_{V'}(x)(v')\neq0$.  Indeed, $v'$ has weight $\alpha''$ and for
  $y=x(\alpha'')=\Fou_\alpha(x)(\alpha'')=\Fou_\alpha(x)(\alpha')\in
  W\subset\U$ one has
  $$x.v'=y.v'=(y_{(1)}.v)\otimes(\ve(y_{(2)})1)
  =(y.v)\otimes1=(\Fou_\alpha(x).v)\otimes1\neq0.$$

  Now consider $\ba\in G$ such that $\cat_\ba$ and $\cat^H_\ba$ are
  semi-simple. Then the finite dimensional algebra $\U_\ba$ is
  isomorphic to the product of algebra $\End_\C(V_i)$ for
  representatives $(V_i)_i$ of the isomorphism class of simple
  $\U_\ba$-modules.  Lift one of these modules $V_0$ to a simple
  module $V_0^H\in\cat^H_\ba$ then by semi-simplicity, $V_0^H$ is
  projective in $\cat^H_\ba$ thus also in $\cat^H$ and so is its image
  $V_0$ in $\cat$.
\end{proof}
\begin{Exem}\label{ex12d}
  Combining Examples \ref{ex1}, \ref{ex1b}, \ref{ex1c}, Corollary
  \ref{C:Axiom3True} and Proposition \ref{P:U inject in CH} we can
  conclude that Axioms \ref{pivotal axiom}--\ref{A:proj} are satisfied
  for the quantum group $\U$ associated to $\mathfrak{sl}_2$ at any
  root of unity $\xi=\exp(\frac{2i\pi}\ell)$ with $\ell\ge3$ and
  $\ell\notin8\Z$.
  Moreover, Examples \ref{ex2}, \ref{ex2b}, \ref{ex2c}, Corollary
  \ref{C:Axiom3True} and Proposition \ref{P:U inject in CH}
imply that Axioms \ref{pivotal axiom}--\ref{A:proj} are satisfied for the quantum group $\U$ associated to simple finite
  dimensional complex Lie algebra for any root of unity $\xi=\exp(\frac{2i\pi}\ell)$ with $\ell$ odd. 
     \end{Exem}
\section{Topological invariant of bichrome graphs}\label{Section of inv of bichrome graphs}
This section does not require the results of Section \ref{S:Galg_DFT}.  It contains the construction of topological invariant corresponding from  the algebra of Section \ref{topo quntum and uni invariant}. In particular, in this section we use the topological ribbon Hopf algebra
$\widehat{\UH}$ and the category $\cat^H$ of weight modules over $\UH$
to construct a topological universal invariant. The invariant is
defined on bichrome graphs colored by both the algebra $\widehat{\UH}$ and objects of $\cat^H$.
Every manifold we will consider in this paper will be oriented, every diffeomorphism of manifolds will be positive, and every link and tangle will be oriented and framed. 

 Let $L:{\UH}^{\wh\otimes n}\to\End_\C(\wh\UH^{\wh\otimes n})$
be the left multiplication: $L_{x}(y)=xy$.  
\subsection{Bichrome graphs}
For a  non-negative integer $n$, we describe the category $[n]\cat^H$ as follows.   The objects of $[n]\cat^H$ are vector spaces of the form
$$[n]V: = ({\UH})^{\wh\otimes n}\otimes V$$
 for some object $V$ of
$\cat^H$.  A morphism in $[n]\cat^H$ from $[n]V$ to $[n]V'$ is a  morphism of topological $\wh\UH$-modules $f:[n]V\to [n]V'$ which is of the form
\begin{equation}\label{coupon-morphism}
f=(L_{\xi^{q+l}}\otimes\Id)\sum_{i=1}^{m}L_{{u_i}}\otimes f_i
\end{equation}
for some ${u_i}\in \U^{\otimes n}$, a quadratic element $q$, a linear
element $l$
and for some linear maps $f_i \in \Hom_{\C}(V, V'),\ 1\leq i\leq m$.

Let $f:[n]V\to [n]V'$ be a non-zero morphism as in Equation \eqref{coupon-morphism}.  Then $f$ determines
$q$ and $\wb l\in(G')^n$ the class of $l$ modulo
$(\Lambda^*)^n$ uniquely.    The pair $(q,\wb l)$ is called the {\em exponent} of
$f$.

As we will explain, bichrome ribbon graph mimic Turaev's definition of ribbon graphs, given in \cite{Tura94}, but have additional information including red and blue edges and special coupons.  For more on ribbon graphs, ribbon categories and their associated Reshetikhin-Turaev functors see \cite{Tura94}.
Here 

An \textit{$n$-string link} is an $(n,n)$-tangle whose $i$-th incoming
boundary vertex is connected to the $i$-th outgoing boundary vertex by
an edge directed from bottom to top for every $1 \leq i \leq n$.

A \textit{bichrome graph} is a ribbon graph with edges divided into two groups, red and blue, satisfying the following condition:
for every coupon there exists a number $k \geq 0$ such that the first $k$ input legs and
the first $k$ output legs are red with positive orientation, meaning incoming and outgoing respectively, while all the other ones are blue.  The \textit{smoothing} of a coupon with $k$ red edges is the union of $k$ parallel red strings
union the coupon with its red edges removed, see Figure 
\ref{Smoothing of a coupon}.  Note this is a different smoothing than given in \cite{DGP} where the blue part was removed in this process.
\begin{figure}
$$
\epsh{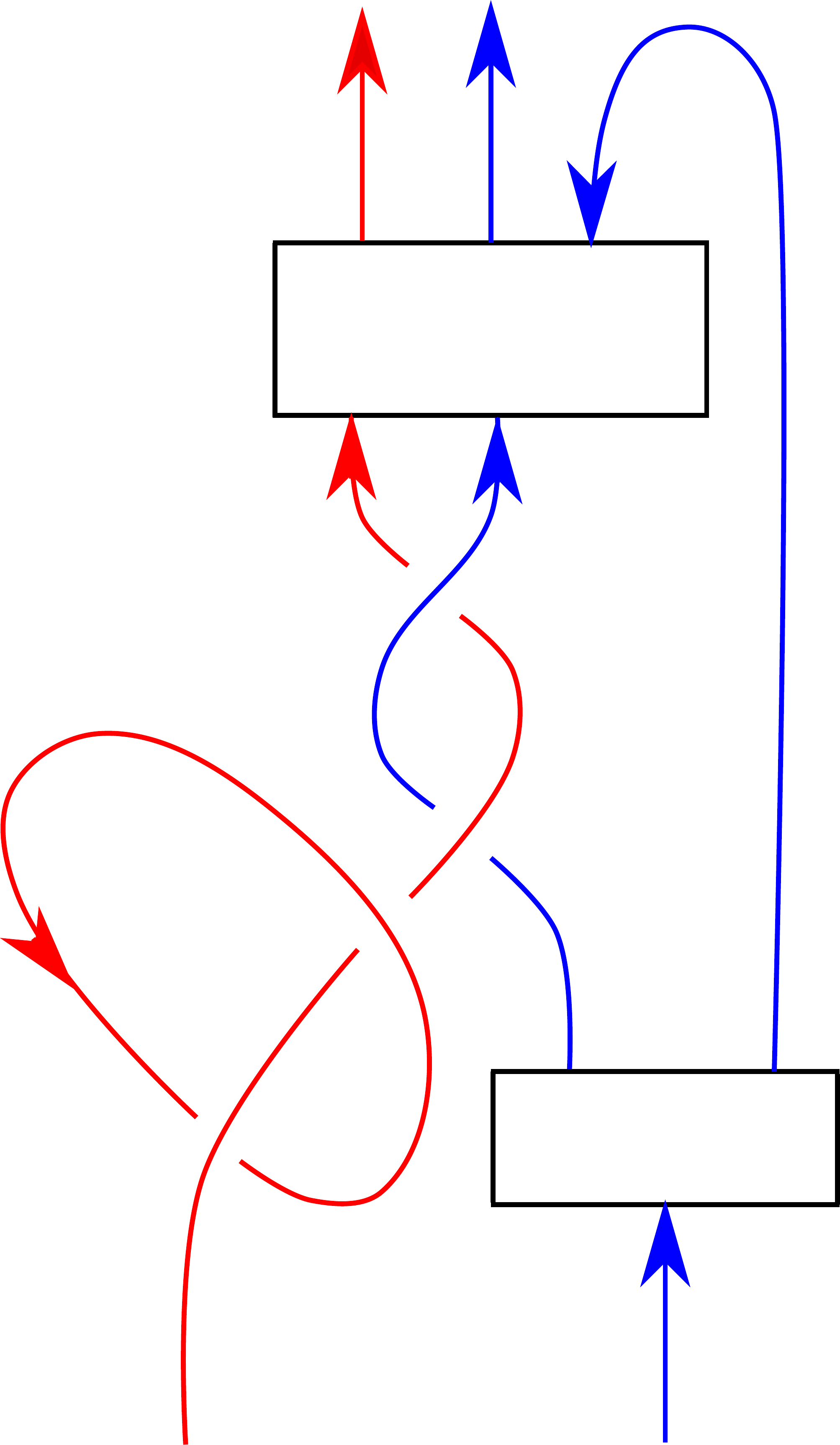}{34ex}\qquad \leadsto \qquad \epsh{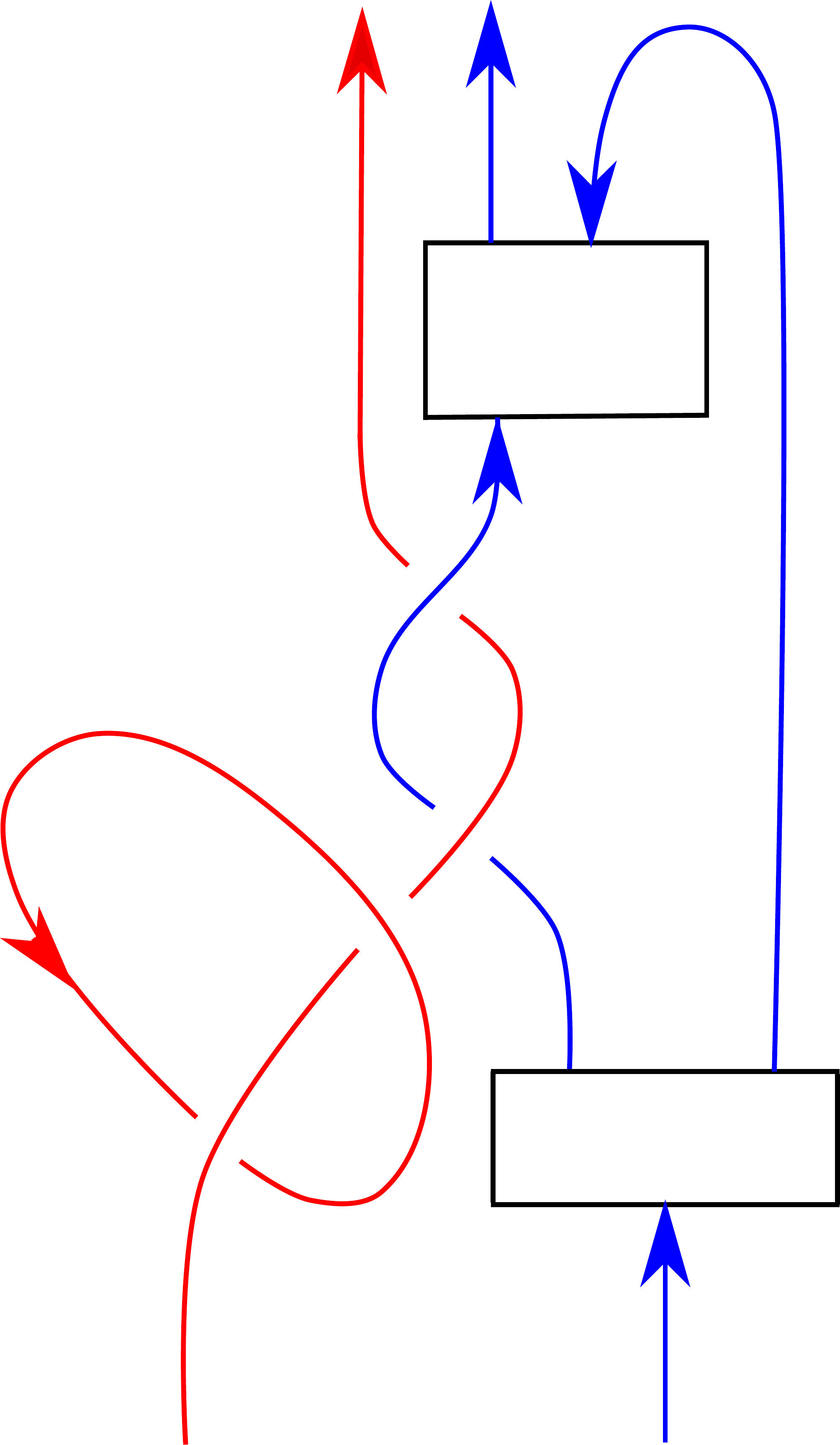}{34ex}
$$
\caption{Smoothing of a coupon}
\label{Smoothing of a coupon}
\end{figure}

A \emph{$n$-string link graph} is a bichrome ribbon graph $\Gamma$ in
$\R^2\times[0,1]$ satisfying the following conditions: (1) the first
$n$ incoming boundary vertices and the first $n$ outgoing boundary
vertices are red, while all the other ones are blue and, (2) the red
sub-tangle of the graph obtained by smoothing all coupons of $\Gamma$
is an $n$-string link.
We say such a graph is \emph{($\wh\UH$,$\cat^H$)-colored} if: (1) the red
edges are colored by $\wh\UH$, (2) the blue edges colored by objects of $\cat^H$ and (3) coupons are colored by morphisms in $[k]\cat^H$ (here $k$ can vary for the different coupons).

 To simplify notation, when it is clear we will say $n$-string link graph for a ($\wh\UH$,$\cat^H$)-colored $n$-string link graph. 
 We also consider the special kind of coloring on a $n$-string link graph:  We say a $n$-string link graph is a  {\em $\Vect_\C$-graph
with $\widehat{\UH}$-beads} if:  (1) the red
edges are colored by $\wh\UH$, (2) the blue edges colored by objects of $\cat^H$, (3) coupons do not have any red edges and are colored by morphisms in $\Vect_\C$ and (4) the graph is equipped with \emph{beads} which form a finite set $b$
of points on edges colored by an element of
$(\widehat{\UH})^{\wh\otimes b}$.
Here and in what follows, if $S$ is a finite set with $n$ element the tensor
product $V^{\otimes S}$ means $V^{\otimes n}$ where the factors are
indexed by elements of $S$.

\subsection{The universal invariant}\label{SS:UnivInv}
Let $\Uo=\text{H\!H}_{0}\bp{\wh\UH}:=\wh\UH/[\wh\UH,\wh\UH]$ seen as a
topological $\wh\UH$-module with trivial action, and let
$$\tru:\wh\UH\to\Uo$$ be the canonical projection.

Let  $\Gamma$ be a $n$-string link
graph.
We will now define  the
universal invariant associate to $\Gamma$ which is an element
$$J(\Gamma)\in (\Uo)^{\wh\otimes C}\,\wh\otimes\,\, (\UH)^{\wh{\otimes} n}\otimes\Hom_{\C}(V,V')$$
where $C$ is the  set  of closed red components of  $\Gamma$.  
Here $V$ (respectively $V'$) is the tensor product of the colors of
the blue bottom (respectively top) boundary points of $\Gamma$.
We will also show that $J(\Gamma)$ is $\wh \UH$-equivariant.

Choose a \emph{diagram} $D$ of $\Gamma$ which  is a regular projection
  of $\Gamma$.  We define an element $J(D)$ in the following three steps.

\vspace{5pt}
\noindent
\textbf{Step 1, Construction of
 $B(D)$:}  
Using the rules of Figure \ref{Fig1} we put beads  at the cup, the cap and the
  crossing (this is same as the usual universal invariant). In this figure the symbol  $S^{\downarrow}(x)$ for $x\in \widehat{\UH}$ is defined by
\begin{equation*}
 \epsh{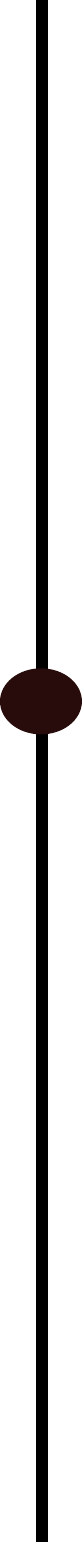}{8ex}\put(-25,2){\ms{S^{\downarrow}(x)}}\quad = \quad 
 \begin{cases}
\quad \qquad \epsh{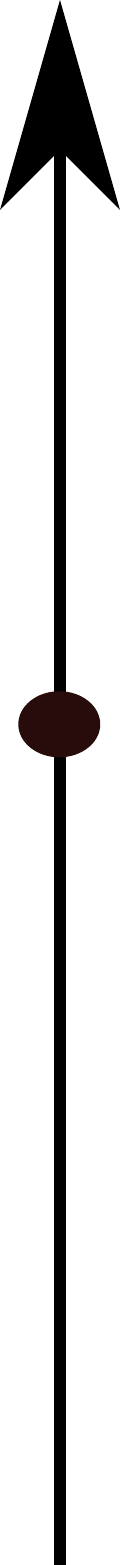}{6ex}\put(-10,2){\ms{x}}\ \text{if the arrow is upward},\\
 \qquad \epsh{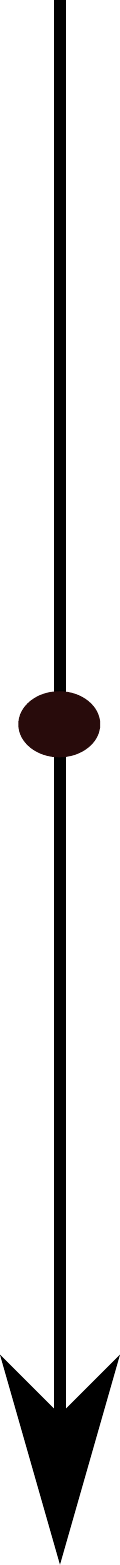}{6ex}\put(-22,2){\ms{S(x)}} \ \text{if the arrow is downward}.
 \end{cases}
\end{equation*}
Then we replace each coupon with $k$ red edges with its smoothing
equipped with $k$ beads as shown in Figure \ref{Fig2}. The result is a
linear combination $B(D)$ of $\Vect_\C$-graphs with $\widehat{\UH}$-beads.
\begin{figure}
$$\fbox{$\quad
\epsh{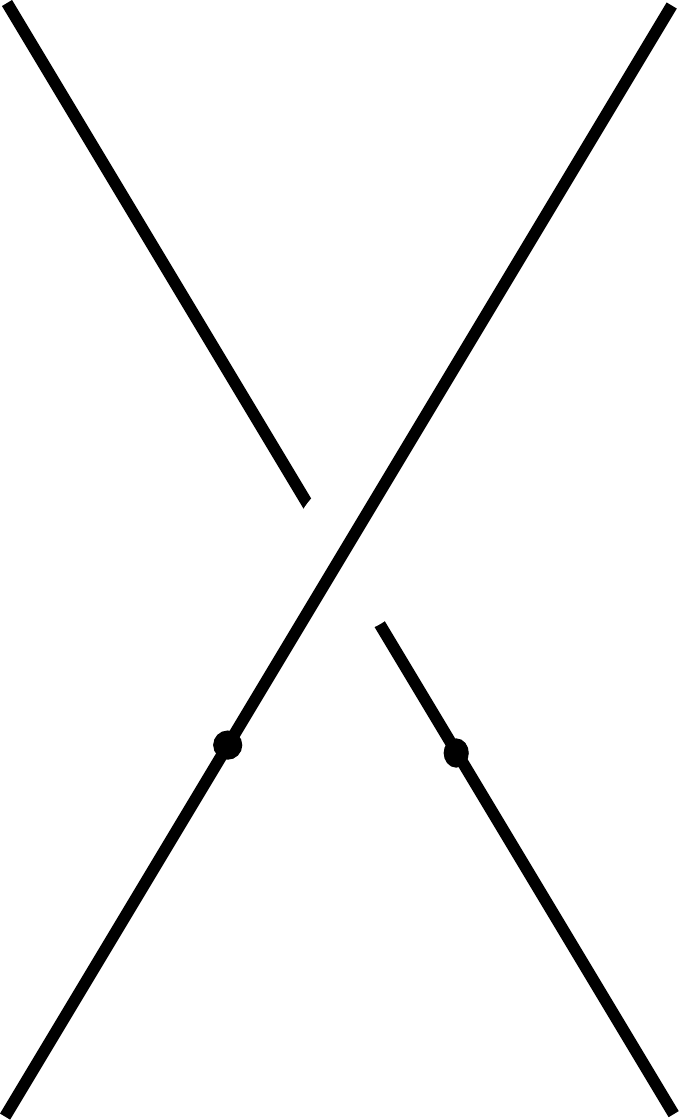}{8ex} \put(-39,-3){\ms{S^{\downarrow}(a)}} \put(-5,-2){\ms{S^{\downarrow}(b)}} \qquad \quad
\epsh{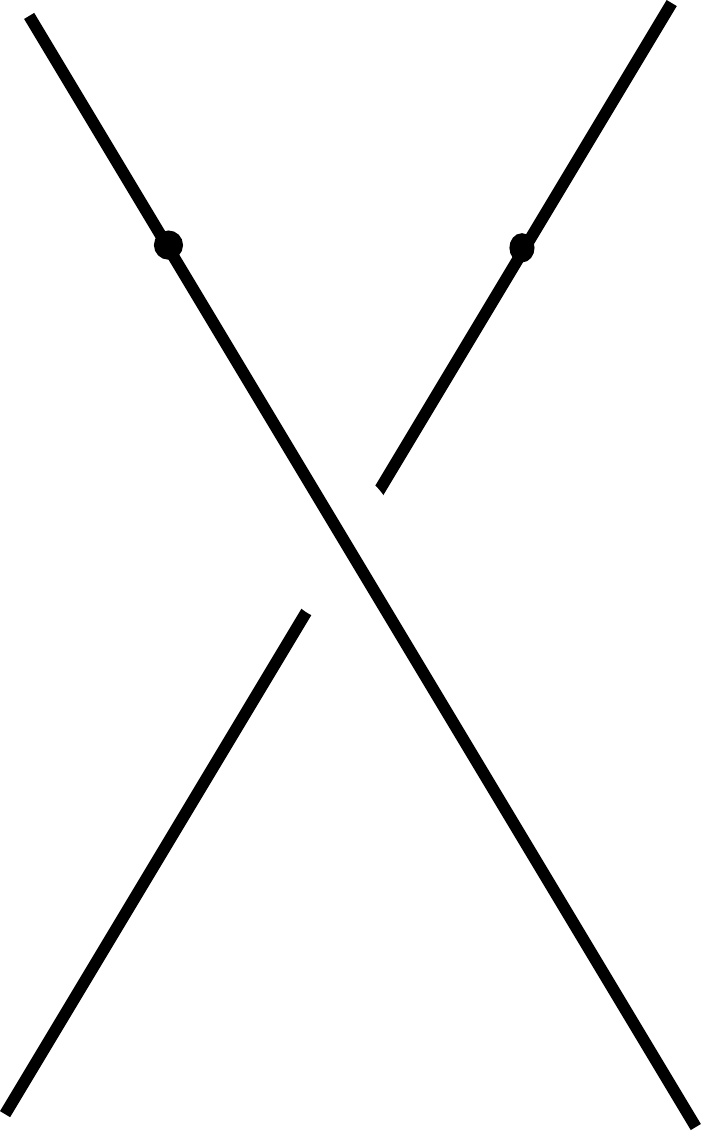}{8ex} \put(-42,12){\ms{S^{\downarrow}(a)}}
\put(-4,12){\ms{S^{\downarrow}(S^{-1}(b))}} \qquad \quad \epsh{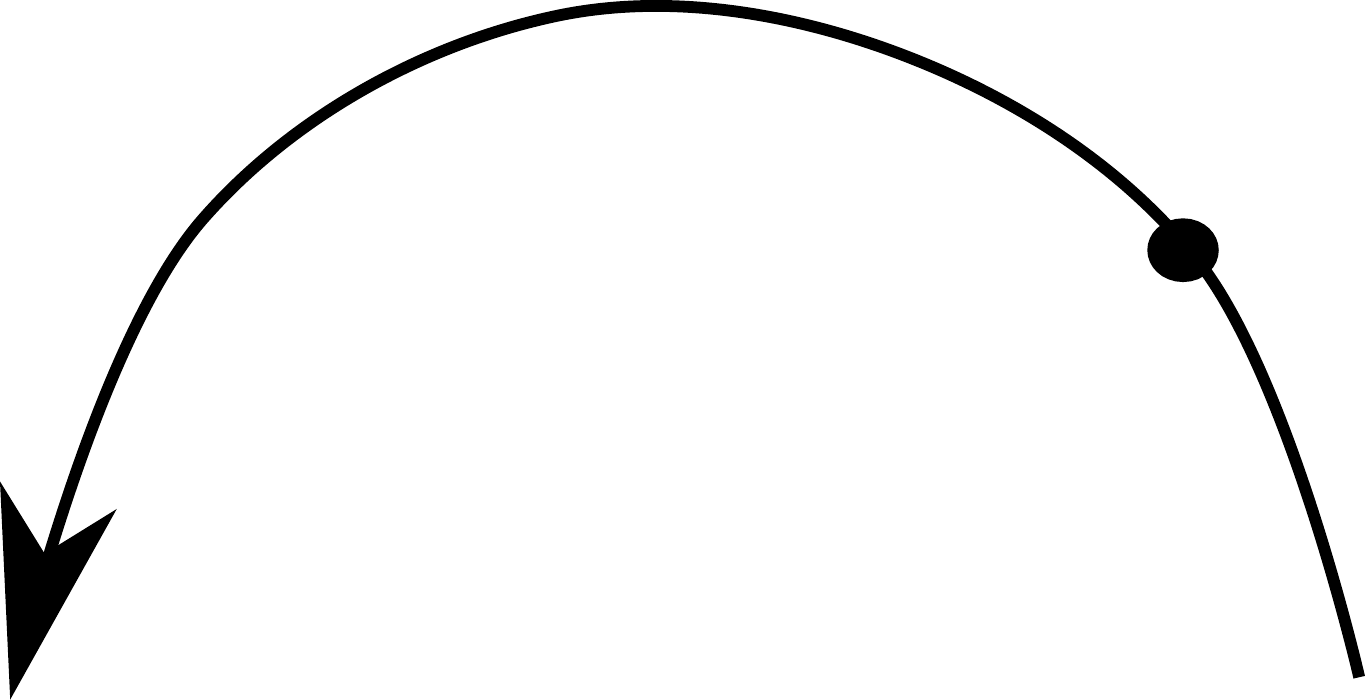}{4ex}
\put(-11,2){\ms{1}} \  \epsh{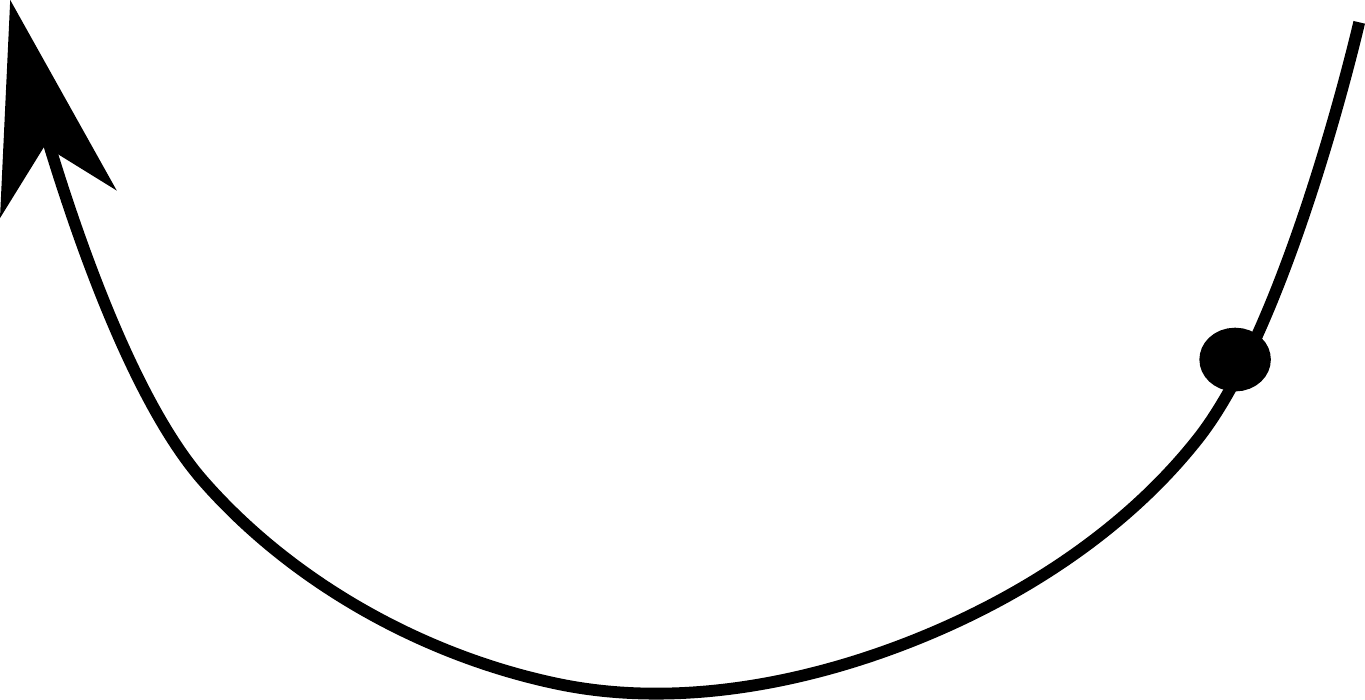}{4ex} \put(-10,1){\ms{1}}
\  \epsh{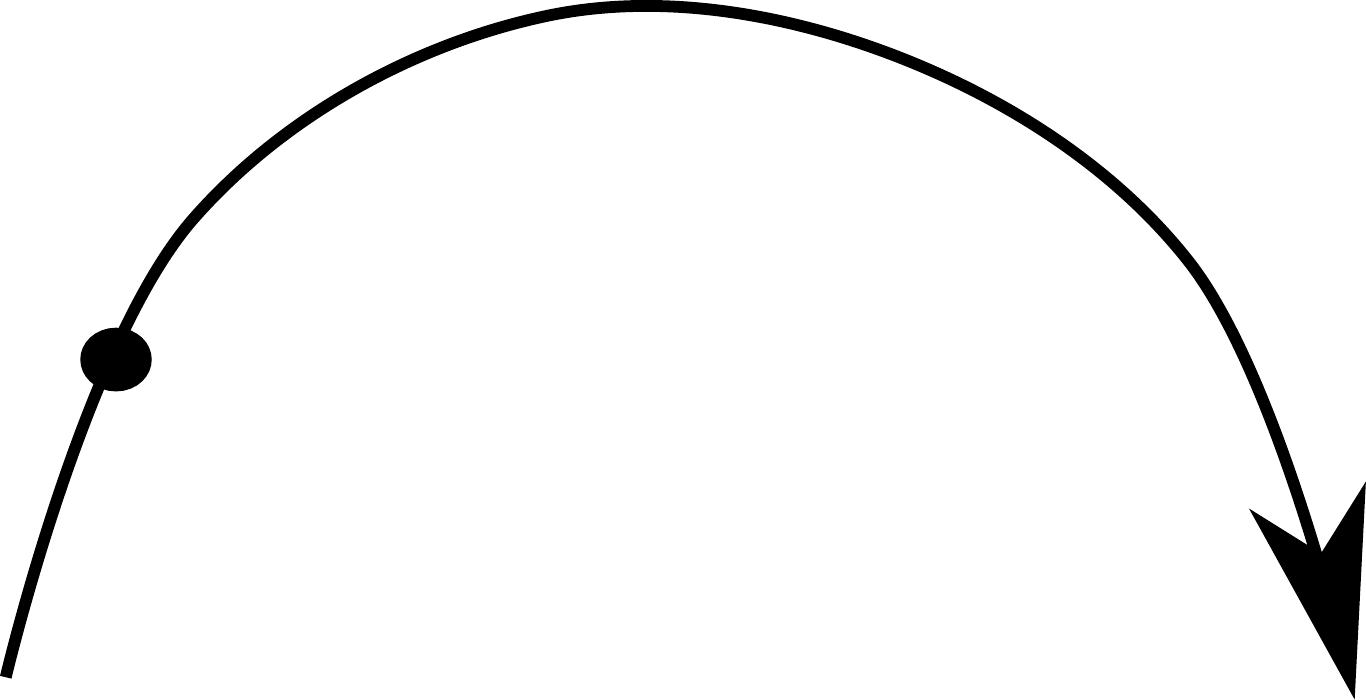}{4ex} \put(-35,0){\ms{g}} \ 
\epsh{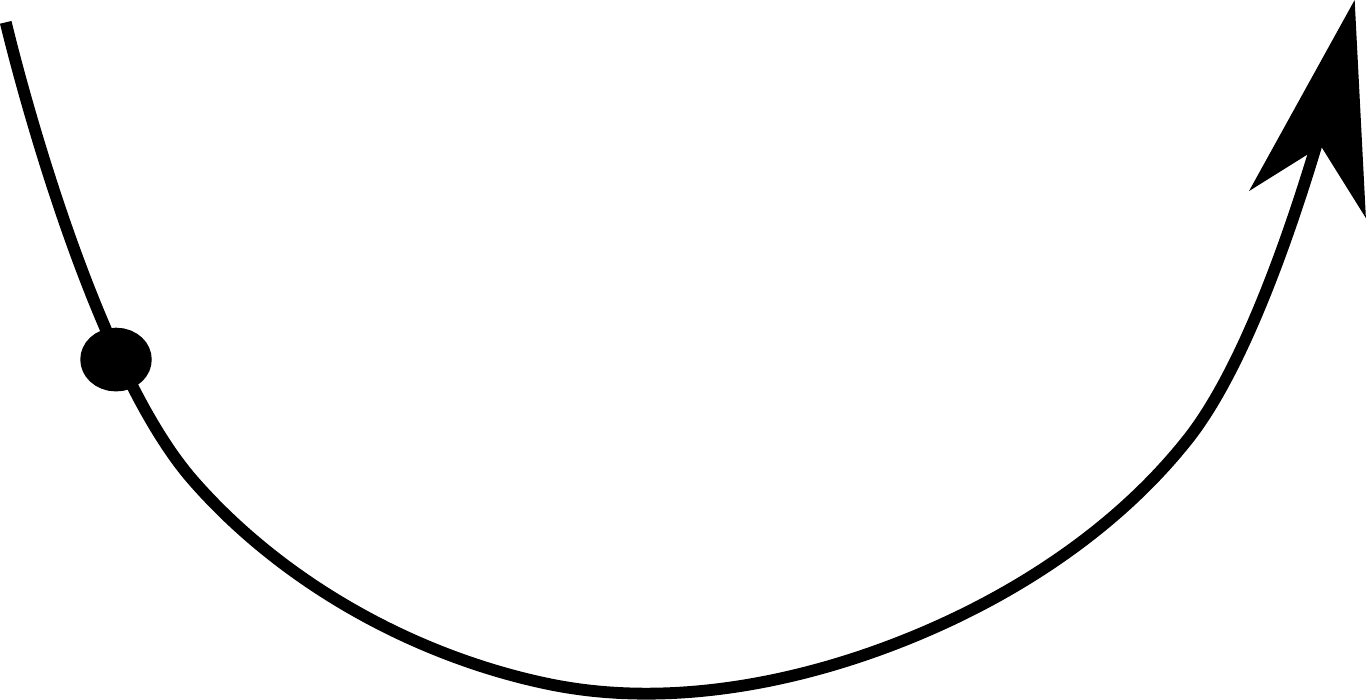}{4ex} \put(-33,2){\ms{{g^{-1}}}}\ $}
$$
\caption{Place beads on the strings where $\mathcal{R}=\sum a\otimes b$}
	\label{Fig1}
\end{figure}
\begin{figure}
  $$\fbox{$
  \epsh{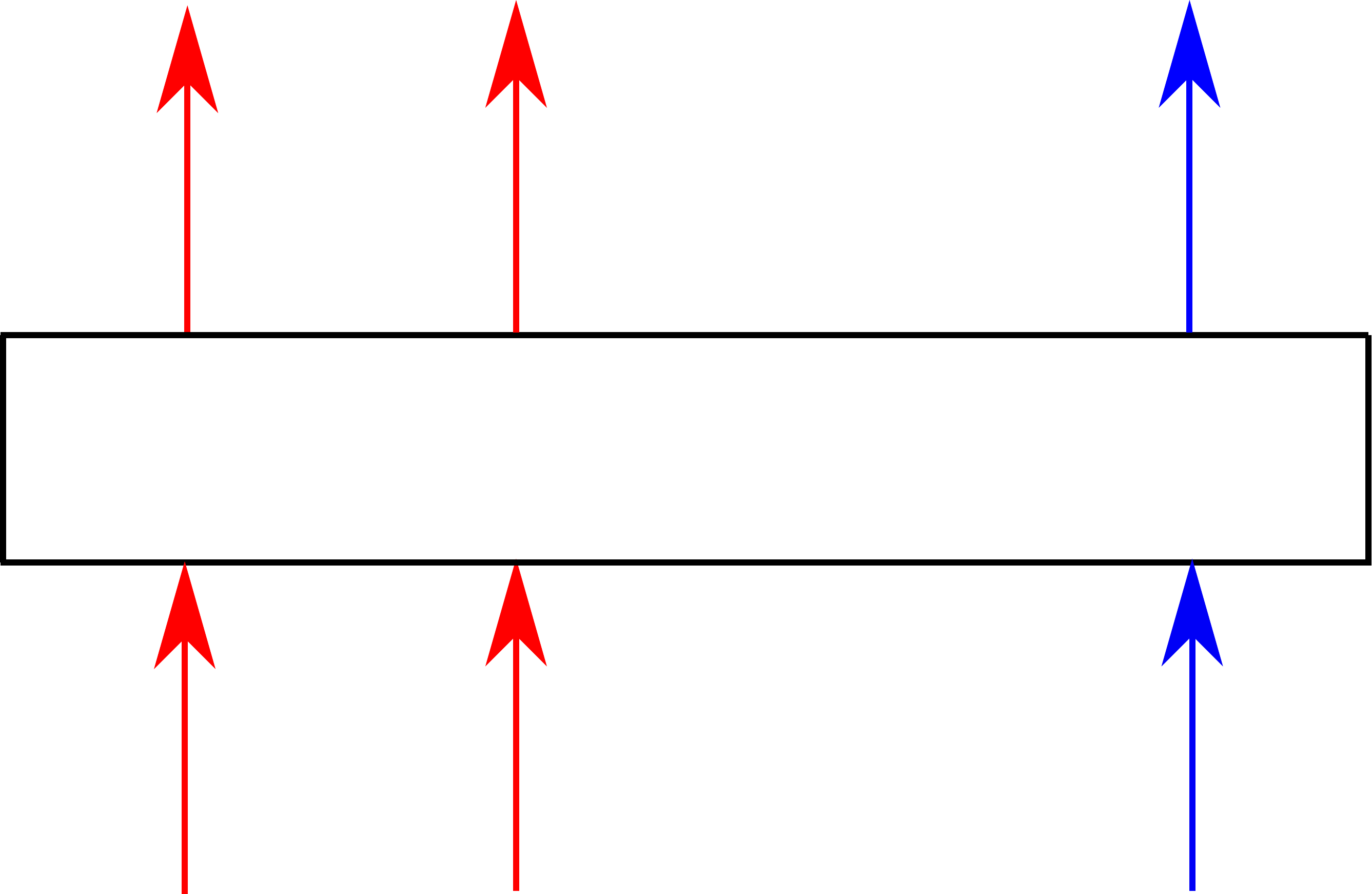}{13ex}\put(-97, 1){\ms{\sum L_{x_{i_1}}\otimes \cdots
      \otimes L_{x_{i_n}}\otimes f_i}} \put(-75,-25){\ms{\cdots}}
  \put(-8,25){\ms{V'}} \put(-8,-25){\ms{V}} \quad
  \simeq
  \quad \sum\quad \epsh{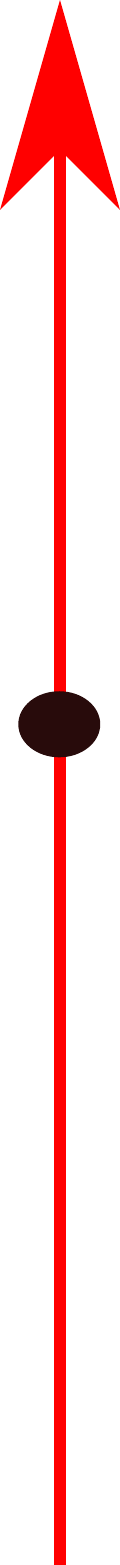}{12ex}
  \put(-15,3){\ms{x_{i_1}}} \put(7,-25){\ms{\cdots}} \qquad
  \epsh{fig3_3r}{12ex}\put(-16,3){\ms{x_{i_n}}} \quad
  \epsh{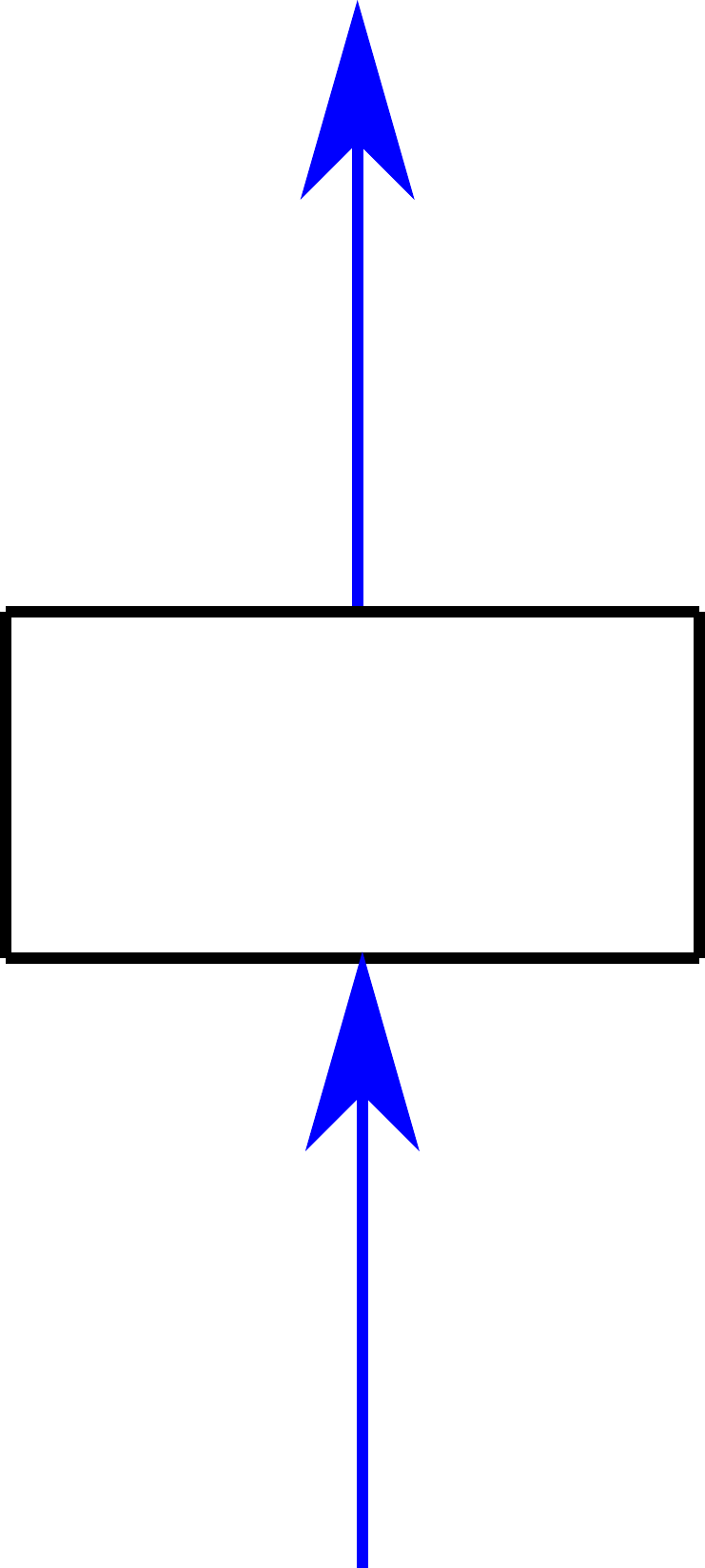}{12ex}\put(-16,1){\ms{f_i}} \put(-8,25){\ms{V'}}
  \put(-8,-25){\ms{V}}\ $}
$$
\caption{Smoothing a mixed coupon using beads.}
	\label{Fig2}
\end{figure}

\vspace{5pt}
\noindent
\textbf{Step 2, Construction of $J_{D}^b$ via contraction:}  
For each closed red component of $D$ fix a base point $b$.  
Using the rules of Figure \ref{Fig0}, we replace each bead on a blue strand colored by
  $V$ by using the action of $\wh \UH$ on $V$ or its
  transpose\footnote{Remark that the convention for beads on blue
    shown in Figure \ref{Fig0} is different from \cite{DGP} where we
    used ${^t\rho_V(S(x))}$.}.  We obtain 
    $$J_D^b\in (\UH)^{\wh\otimes C}\,\wh\otimes\,\, (\UH)^{\wh{\otimes} n}\otimes\Hom_{\C}(V,V')$$ by 
  multiplying the beads on each red strand in order opposite to the
  orientation of the component (if the red component is closed start at the fixed base point) and we separately evaluate the blue part of the graph
  using the Penrose graphical calculus for $\Vect_\C$-colored ribbon
  graphs.  
 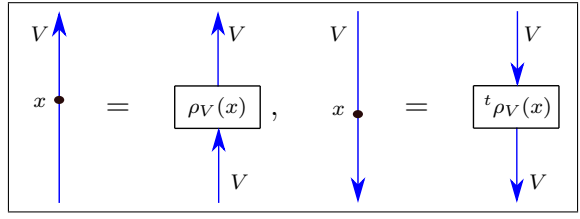
\begin{figure}
   $$\fbox{$\quad
\epsh{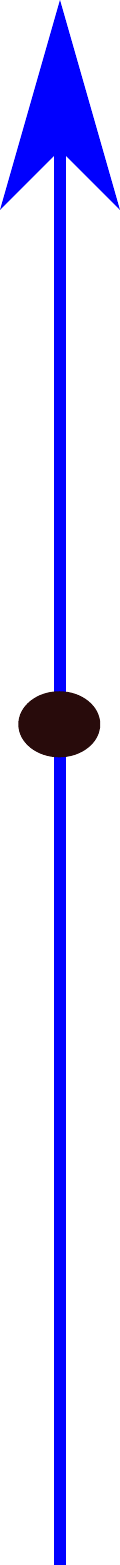}{14ex} \put(-12,3){\ms{x}} \put(-13,28){\ms{V}}\quad = \quad \epsh{fig3_2b}{14ex}\put(-27,1){\ms{\rho_V(x)}}\put(-12,28){\ms{V}} \put(-11,-28){\ms{V}} \ , \qquad
\epsh{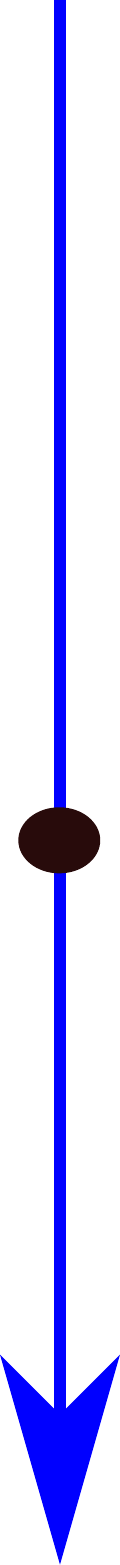}{14ex}\put(-12,0){\ms{x}} \put(-13,28){\ms{V}}\quad = \quad \epsh{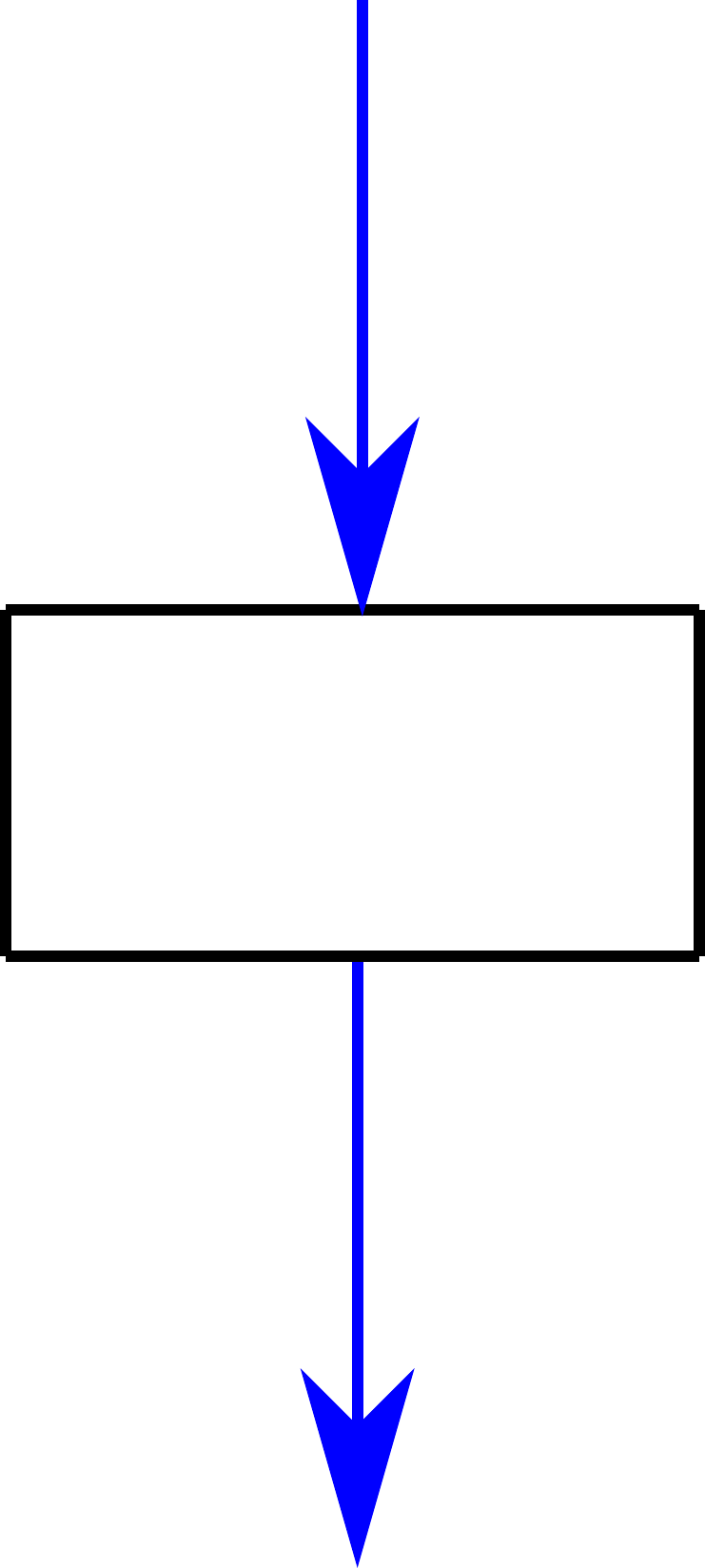}{14ex} \put(-28,1){\ms{^t\rho_V(x)}}\put(-13,28){\ms{V}} \put(-11,-28){\ms{V}}\ $}
$$
\caption{Beads on blue edges are $\Vect_{\C}$-colored coupons.}
	\label{Fig0}
\end{figure}

\vspace{5pt}
\noindent
\textbf{Step 3, Definition of $J(D)$ via trace:}  For each closed red component apply the trace $\tru:\wh\UH\to\Uo$  to obtain the element 
$$J(D)=\bp{\tru^{\otimes C}\otimes \Id^{\otimes n}\otimes \Id }\bp{J_{D}^b} \in (\Uo)^{\widehat{\otimes} C}\wh \otimes \,\, (\UH)^{\wh{\otimes} n}\otimes \Hom_\C(V, V').$$
The element $J(D)$ is independent of the choice of base points $b$ because the starting point of multiplying the beads on
  any closed red circle component does not matter as the result is sent in the
  quotient $\Uo$.  

This construction illustrated for the diagram $D$ given in
Figure \ref{Fig3}, where $J\bp{D}$ is an element of
$\Uo\widehat{\otimes}\wh \UH \otimes \Hom_\C(V_1, V_4)$
given by 
\begin{align*}J\bp{D}=\sum \tru\bp{g^{-1}S(b_s)1a_k}\otimes
x_ib_ta_jS^{-1}(b_k)a_s\otimes F_{{\C}}(T)
\end{align*}
where 
$F_{{\C}}(T)=\bp{\Id_{V_4}\otimes \lev_{V_3}}\circ \bp{f_i
  \circ\rho_{V_2}(a_t b_j)\otimes \Id_{V_3}}\circ f\in \Hom_{\C}(V_1,
V_4).$
Here $F_{{\C}}$ is the Reshetikhin-Turaev functor from the category of $\Vect_{\C}$-colored ribbons graphs to the category of finite dimensional vector spaces $\Vect_{\C}$.  

\begin{figure}
   $$
   \epsh{fig3_7b}{44ex}\put(-75, 63){\ms{\sum L_{x_i}\otimes
       f_i}}\put(-29,-64){\ms{f}}\put(-24,
   -110){\ms{V_1}}\put(-40,-44){\ms{V_2}} \put(-6,-44){\ms{V_3}}
   \put(-52,116){\ms{V_4}} \qquad\stackrel B\longrightarrow
   \qquad \epsh{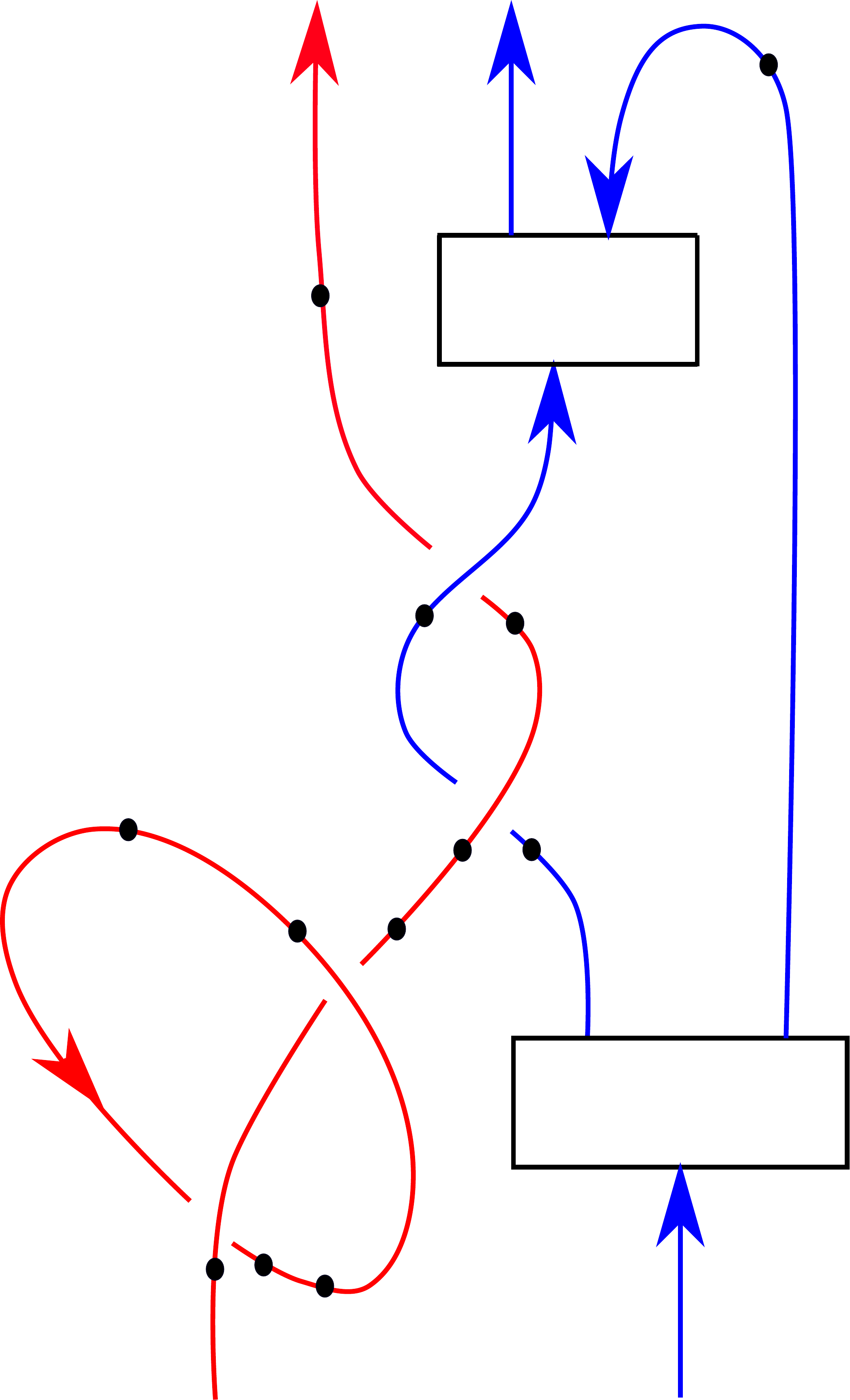}{44ex}
   \put(-50, 65){\ms{f_i}}\put(-29,-64){\ms{f}}\put(-24,
   -110){\ms{V_1}}\put(-40,-44){\ms{V_2}} \put(-6,-44){\ms{V_3}}
   \put(-52,116){\ms{V_4}} \put(-10,110){\ms{1}} \put(-98,
   68){\ms{x_i}} \put(-81, 14){\ms{a_t}} \put(-48, 14){\ms{b_t}}
   \put(-75, -20){\ms{a_j}} \put(-48, -20){\ms{b_j}} \put(-100,
   -40){\ms{a_k}} \put(-77, -44){\ms{S^{-1}(b_k)}} \put(-110,
   -16){\ms{1}} \put(-116, -90){\ms{a_s}} \put(-98, -84){\ms{S(b_s)}}
   \put(-88, -102){\ms{g^{-1}}}
$$
\caption{Putting beads on a graph $D$}
	\label{Fig3}
\end{figure}
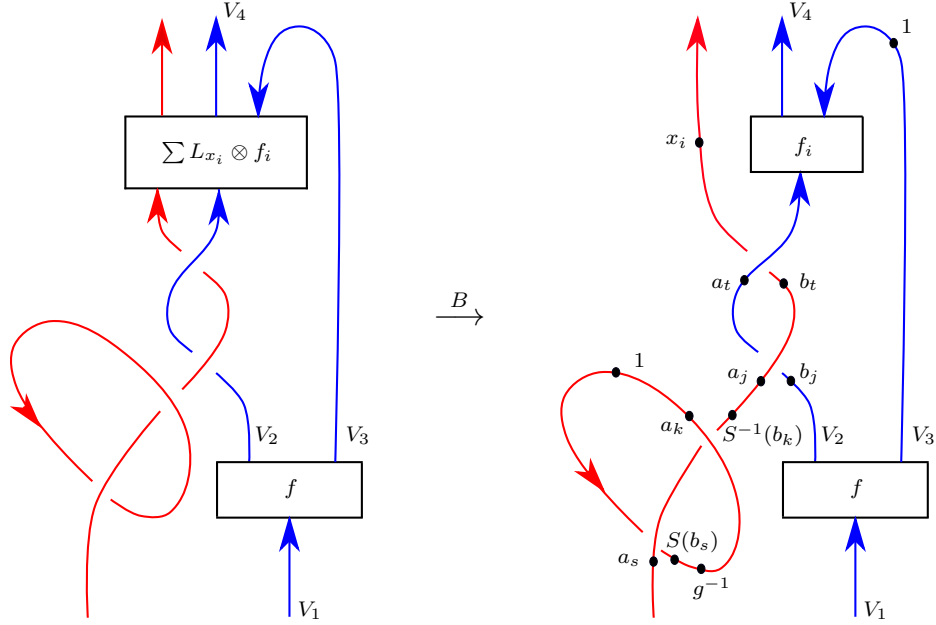
The following theorem is an adapted version of the universal invariant of \cite{Henning96,Ohtsuki02,kauffman2001oriented}:
\begin{theo} \label{T:univ-inv}
 The element $J(D)$
  computed above on does not depend on the choice of $D$ and is an isotopy  invariant $J(\Gamma)$ of $\Gamma$
  in $\R^2\times[0,1]$.  Furthermore $J(\Gamma)$ is
  $\wh \UH$-equivariant, i.e.,
\begin{equation}\label{E:EquivarJ}
  J(\Gamma)\bp{u_{(1)}\otimes\cdots\otimes u_{(k)}\otimes\rho_{V}(u_{(k+1)})}
  =\bp{u_{(1)}\otimes\cdots\otimes
    u_{(k)}\otimes\rho_{V'}(u_{(k+1)})}J(\Gamma)
  \end{equation}
  for all $u\in \wh \UH$ where $\Delta^{k+1}(u)=u_{(1)}\otimes\cdots\otimes u_{(k+1)}$ and $k$ equals $n$ plus the number of closed red components of $\Gamma$.  
\end{theo}
\begin{proof}
  We will use the symbol $\equiv$ to relate two
  $\Vect_\C$-graphs with $\widehat{\UH}$-beads whose images are equal under $F_{{\C}}$.
  
  The following local relation for $\Vect_\C$-graphs with $\widehat{\UH}$-beads 
  apply for any color of their edges:
  \begin{equation}
    \label{eq:mult-beads}
    \epsh{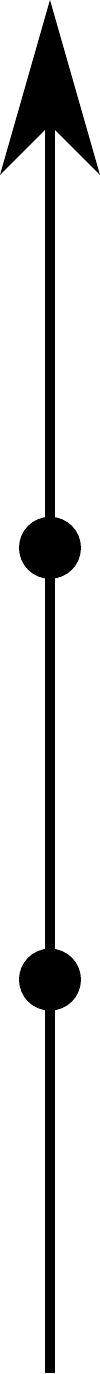}{10ex}\put(-12,8){\ms x}\put(-12,-9){\ms y}\quad
    \equiv\quad\hspace{1ex}\epsh{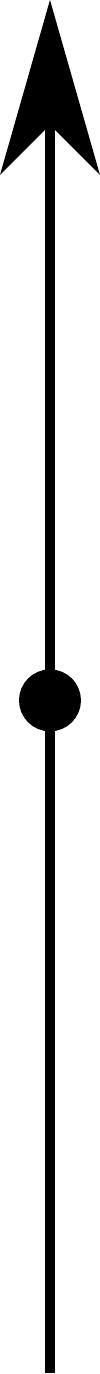}{10ex}\put(-15,2){\ms {xy}}\qquad\text{ and }\qquad
    \epsh{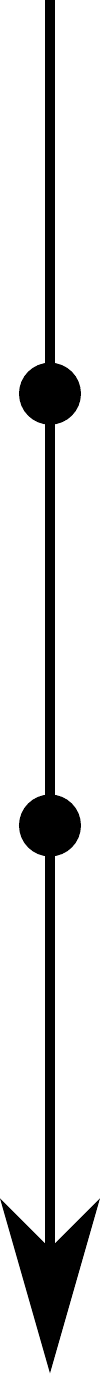}{10ex}\put(-12,12){\ms x}\put(-12,-5){\ms y}\quad
    \equiv\quad\hspace{1ex}\epsh{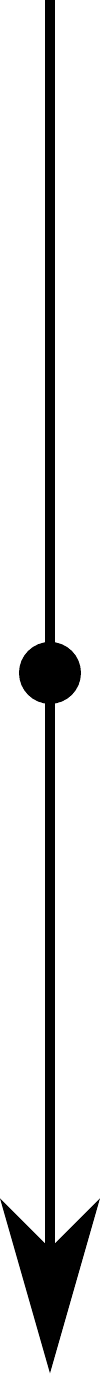}{10ex}\put(-15,2){\ms {yx}}\quad.
  \end{equation}
  Also beads can freely move around cap and cup with any orientation:
  \begin{equation}
    \label{eq:cap-beads}
    \epsh{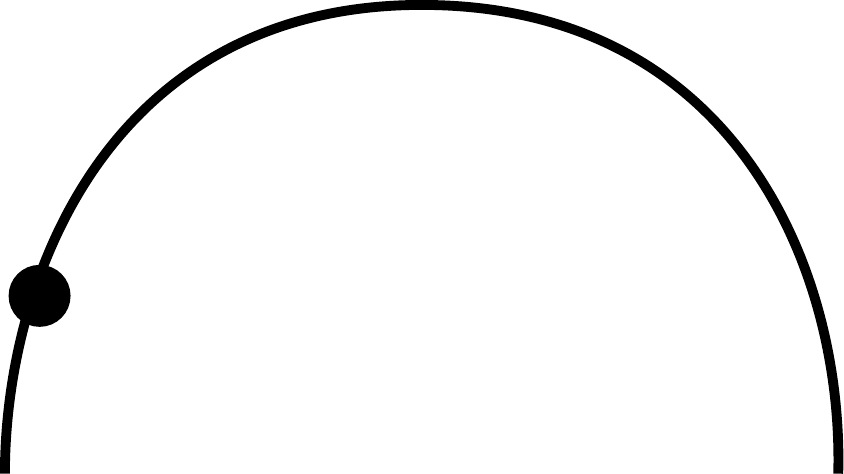}{4ex}\put(-30,-1){\ms x}\quad\equiv\quad
    \epsh{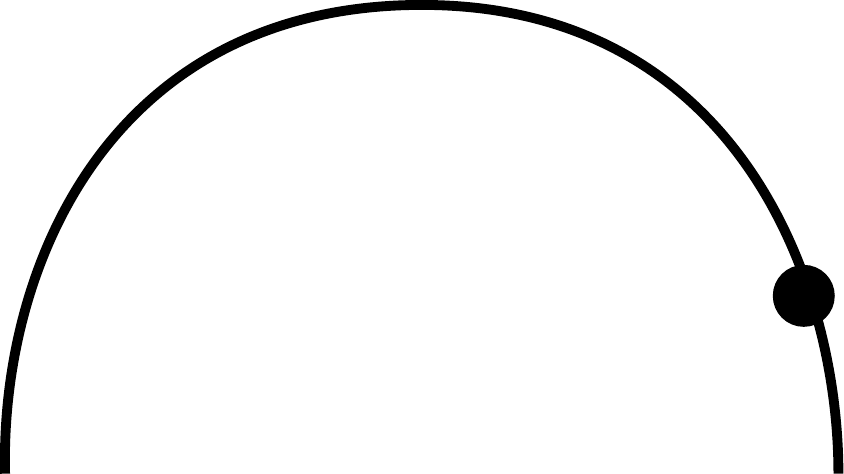}{4ex}\put(-10,-1){\ms {x}}\qquad\text{ and }\qquad
    \epsh{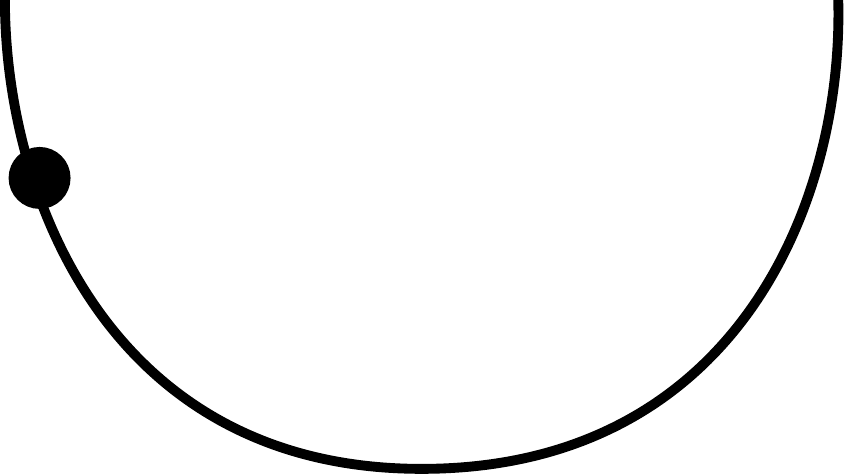}{4ex}\put(-30,4){\ms x}\quad\equiv\quad
    \epsh{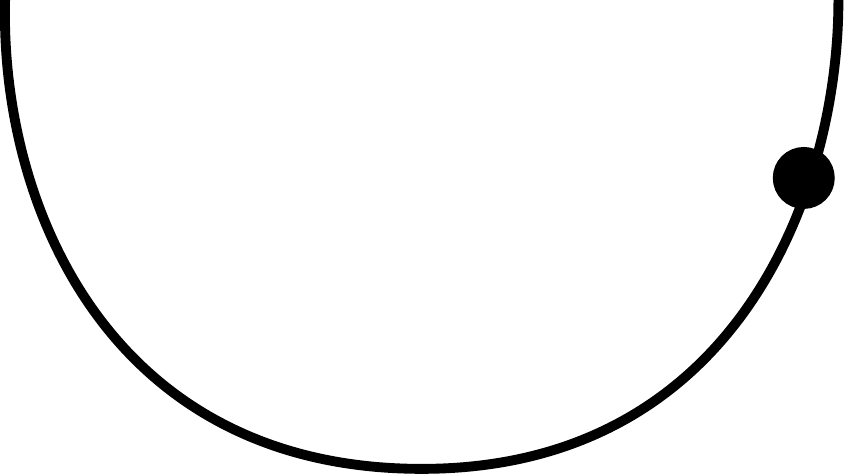}{4ex}\put(-10,4){\ms {x}}\quad.
  \end{equation}
  Using these relations the invariance of $J(\Gamma)$ under
  Redeimeister's moves follows the proof of \cite{Ohtsuki02} using the
  properties of the R-matrix.
  The invariance for the local move of a strand sliding over (or
  under) a coupon follows from the $\wh \UH$-equivariance of the
  morphisms coloring the coupon.
 We will prove the  invariance of the local move given in Figure \ref{move coupon 1}, the other moves are proved similarly.
  After putting the beads and smoothing the
  coupon the associated equality is
\begin{equation*}
\sum a_ja_i\otimes x b_i \otimes f\circ b_j = \sum a_t a_s \otimes b_s x\otimes b_t\circ f.
\end{equation*} 
This equality is equivalent to 
\begin{equation}\label{equality move coupon}
\bp{1\otimes \sum x \otimes f}\RR_{13}\RR_{12}=\RR_{13}\RR_{12}\bp{1\otimes \sum x \otimes f}.
\end{equation}
We consider the right hand side of Equality \eqref{equality move coupon}, one gets
\begin{align*}
  \RR_{13}\RR_{12}\bp{1\otimes \sum x \otimes f}
  &= \bp{\Id\otimes \Delta}(\RR)\bp{1\otimes \sum x \otimes f}\\
  &= \bp{\sum a_i\otimes \Delta(b_i)}\bp{1\otimes \sum x \otimes f}\\
  &= \sum a_i \otimes \Delta(b_i)\bp{\sum x \otimes f}\\
  &= \sum a_i \otimes \bp{\sum x \otimes f}\Delta(b_i)\\
  &= \bp{1\otimes \sum x \otimes f}\bp{\Id\otimes \Delta}(\RR)\\
  &= \bp{1\otimes \sum x \otimes f}\RR_{13}\RR_{12}.
\end{align*}
In the third equality we used $\sum L_x \otimes f$ being
$\wh \UH$-invariant. Hence Equality \eqref{equality move coupon}
holds.
\begin{figure}
$$
\epsh{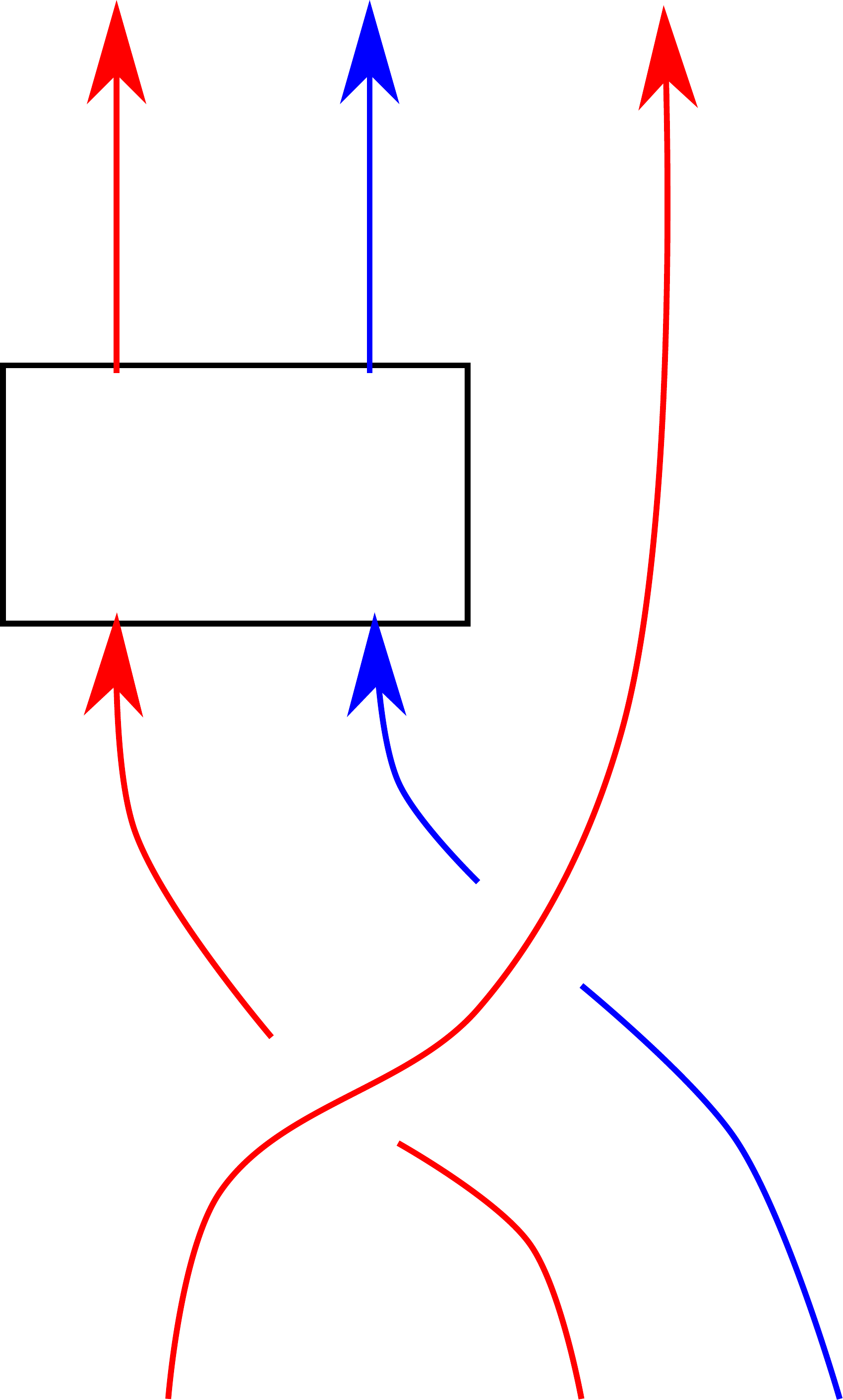}{24ex}\put(-72,19){\ms{\sum L_x \otimes f}}\qquad
\sim\qquad \epsh{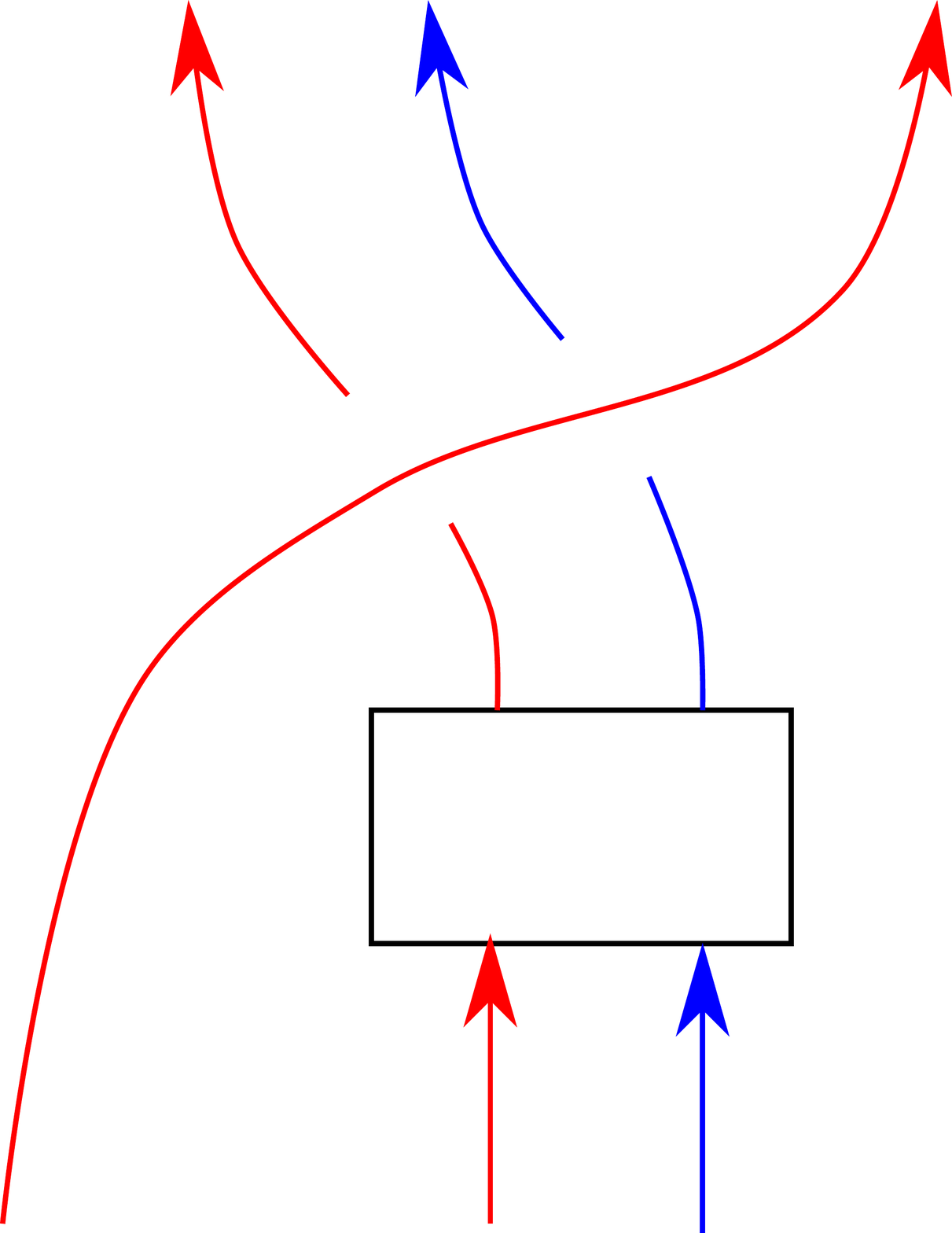}{24ex} \put(-56,-19){\ms{\sum L_x \otimes f}}
$$
\caption{Moving at local coupon}
\label{move coupon 1}
\end{figure}

As claimed above, the element $J(D)$ is independant of the choice of
base point: indeed, if one move the base point so that it pass a bead
colored by $a\in\UH$ the two corresponding values for $J(D)$ will be
the same except for a factor which will be $\tr_u(xa)$ in one case and
$\tr_u(ax)$ on the other case but these two elements are equal in
$\U^0$ so moving the base does not change $J(D)$.

We now prove that $J(\Gamma)$ is $\wh\UH$-equivariant, i.e.,
Equation \eqref{E:EquivarJ} holds.
One can compute the
left hand side of Equation \eqref{E:EquivarJ} by applying $F_{{\C}}$ to the diagram
$B(D)$ where some beads colored by the $u_{(i)}$ have been
threaded at the intersection point of $D$ with an horizontal line
near the bottom of $D$.  Similarly, the right hand side is the
image by $F_{{\C}}$ of the diagram $B(D)$ with some $u_{(i)}$ beads
along a horizontal line near the top of $D$.  For a horizontal
line at any level (not intersecting coupons) we can define similarly a
diagrams obtain from $B(D)$ with some $S^{\downarrow}(u_{(i)})$
beads at the intersection point of $\Gamma$ with the horizontal line.

We claim that all diagrams obtained by this process will have the same
image by $F_{{\C}}$.  Indeed, it is enough to see this is true when one moves
the horizontal line to pass an elementary diagram. Except for red caps
and cups, this follows from the $\wh\UH$-equivariance of the braiding
given by $\tau \RR$, of the morphisms coloring coupons and of the
(co)-evalutations in $\cat^H$.  For red cups and caps it is a direct
computation~; we develop here the case of red caps (see Figure
\ref{caps morphisms}): Assume the horizontal line under the cap
intersect $\Gamma$ at $N+2$ points and at $N$ points when the line is
above the cap.  If the cap is joining the $i$\ith and
the $i+1$\ith strands, we can write
$\Delta^{N}u=u_{(1)}\otimes\cdots\otimes u_{(i-1)}\otimes a\otimes
u_{(i+1)}\otimes\cdots\otimes u_{(N+1)}$. Then
$\Delta^{N-1}u=u_{(1)}\otimes\cdots\otimes u_{(i-1)}\otimes
\ve(a)\otimes u_{(i+1)}\otimes\cdots\otimes u_{(N+1)}$ and
$\Delta^{N+1}u=u_{(1)}\otimes\cdots\otimes u_{(i-1)}\otimes
a_{(1)}\otimes a_{(2)}\otimes u_{(i+1)}\otimes\cdots\otimes
u_{(N+1)}$.
Then the statements for the red caps (see Figure \ref{caps morphisms})
hold from the equalities $S(a_{(1)})a_{(2)}=\ve(a)$ and
$S(a_{(2)})ga_{(1)}=gS^{-1}(a_{(2)})a_{(1)}
=\ve(a)g$.
The similar
statements for the red cups hold from the first equality above and
$a_{(2)}g^{-1}S(a_{(1)})=\ve(a)g^{-1}$ for $a\in \wh \UH$.
\begin{figure}
$$
\fbox{$\epsh{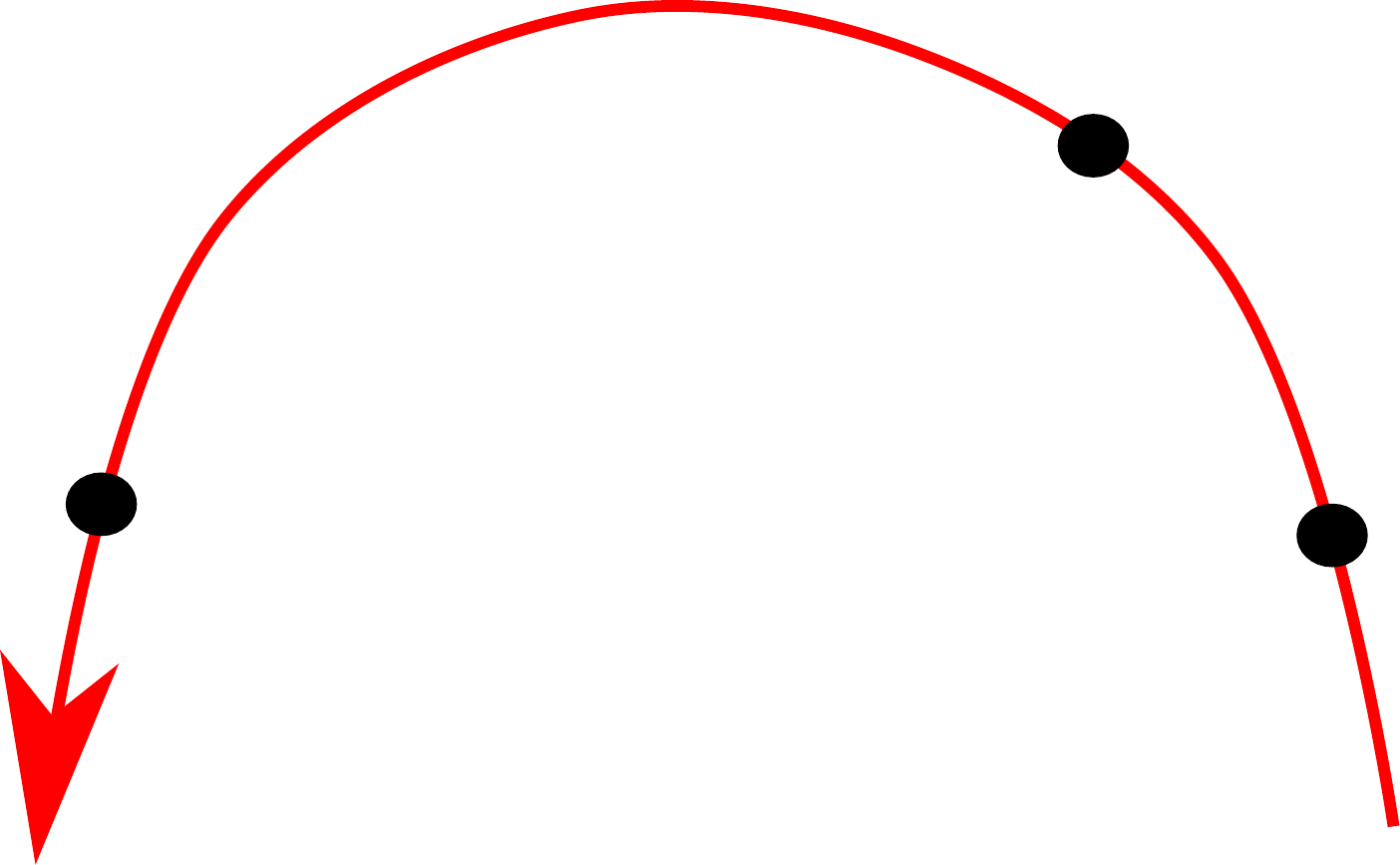}{8ex} \put(-57,-3){\ms{S(a_{(1)})}}
\put(-2,2){\ms{a_{(2)}}}\put(-12,20){\ms{1}}\quad\ \equiv \
\epsh{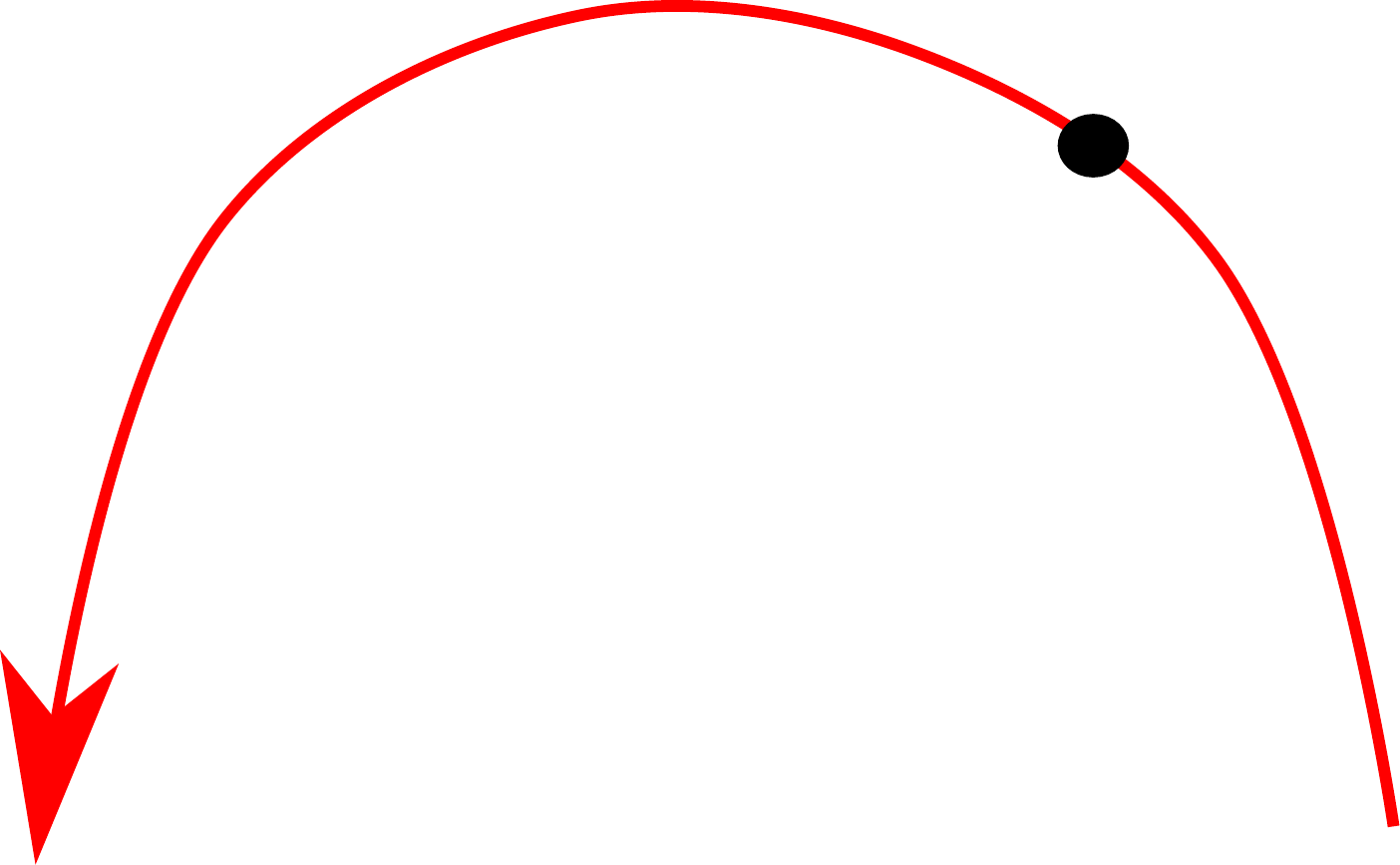}{8ex}, \put(-17,20){\ms{1}}
\put(-45,28){\ms{\ve(a)}} \quad \epsh{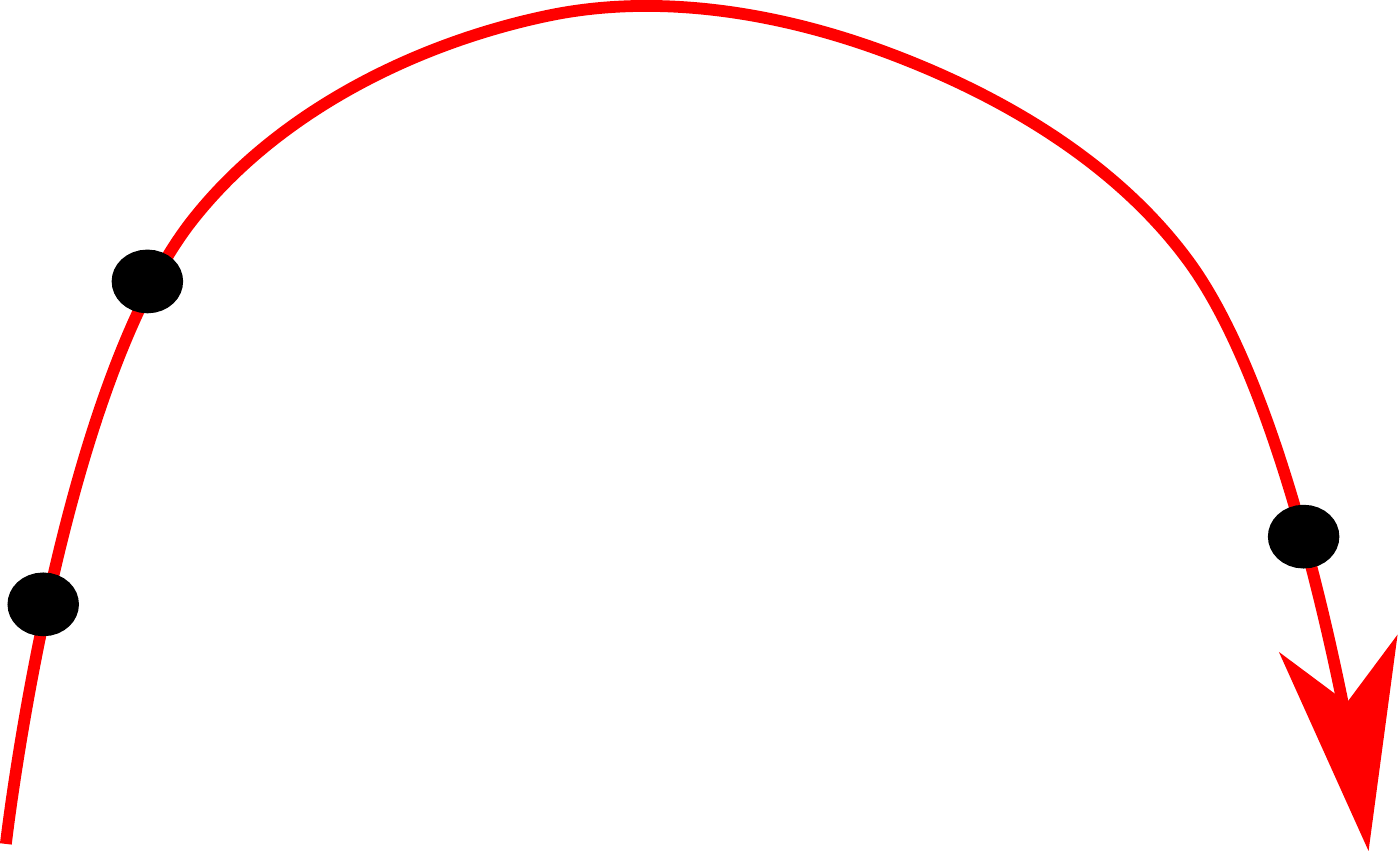}{8ex}
\put(-61,-8){\ms{a_{(1)}}} \put(-56,8){\ms{g}}
\put(-34,-8){\ms{S(a_{(2)})}}\ \equiv \
\epsh{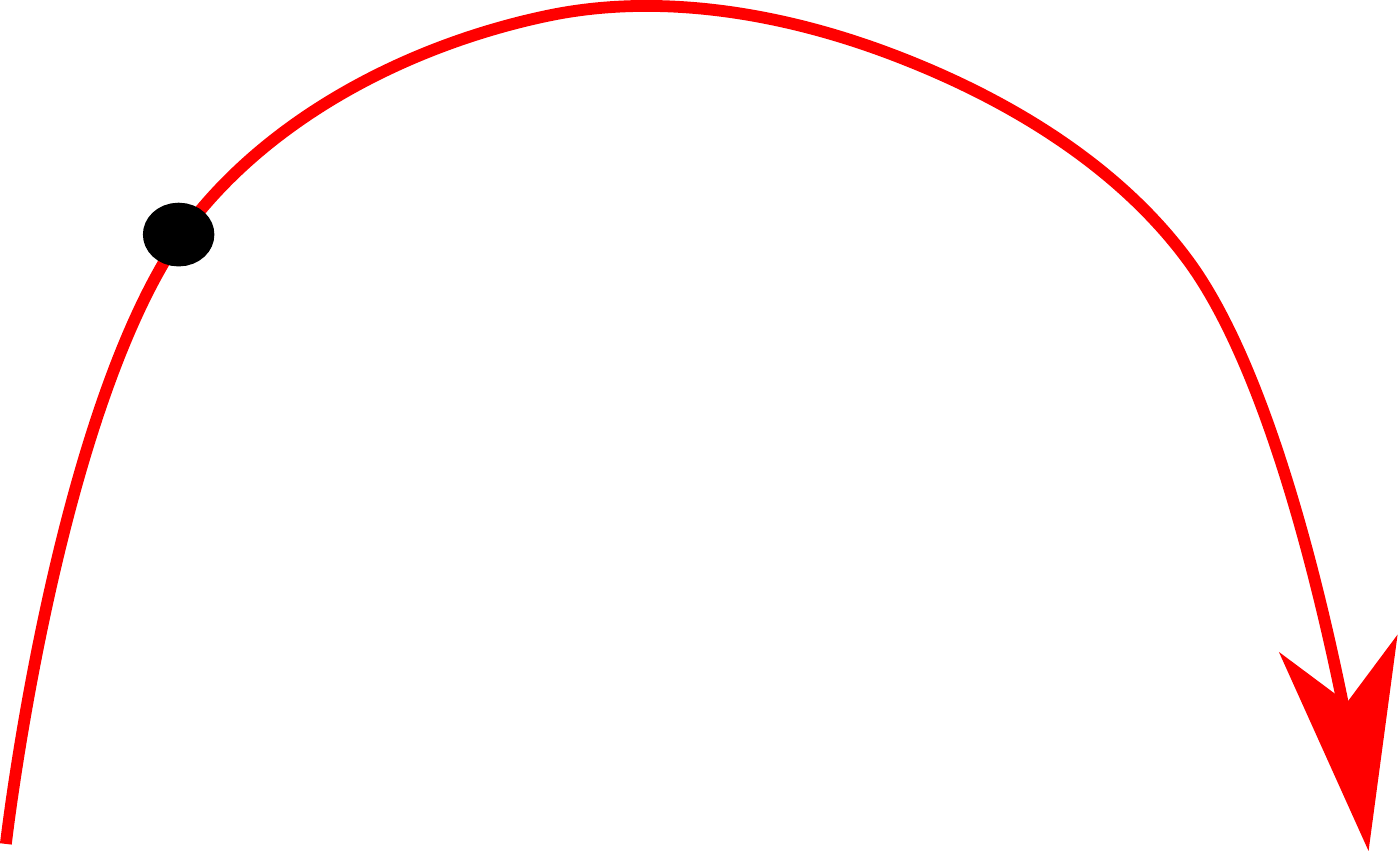}{8ex}\put(-56,8){\ms{g}}\put(-45,28){\ms{\ve(a)}}$}
$$
\caption{Equivariance for red caps}
	\label{caps morphisms}
\end{figure}
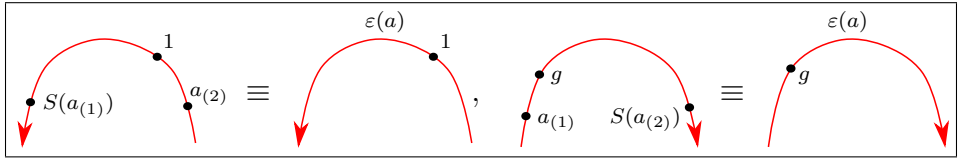
\end{proof}

\begin{Lem}\label{reducing commutator}
  Let $\wh\UH_0$ be the weight $0$ subspace of $\wh\UH $. Then
  $$\Uo\cong \wh\UH_0/([\U,\wh\UH]\cap\wh\UH_0).$$
\end{Lem}
\begin{proof}
  First $\wh\UH$ is a weight module and if $x$ is homogeneous of
  degree $\lambda\neq0$, then there exists $h\in\H$ such that
  $\lambda(h)=1$ and then $x=[h,x]$ is a commutator. Hence only weight
  $0$ elements contribute to $\Uo$. Next, remark that if
  $x=f(h)w\in\wh\UH$ and $y\in\wh\UH$ have opposite weights, then
  $f(h)wy=wyf(h)$ and
  $$[x,y]=[f(h)w,y]=f(h)wy-yf(h)w
  =[w,yf(h)]\in[\U,\wh\UH].$$
\end{proof}
\subsection{Value of the invariant}
Here we describe some properties of the invariant $J$ of the previous section.
Let $D$ be a planar diagram representing a
$n$-string graph $\Gamma$ in $\R^2\times[0,1]$ having $m$
closed red components.
In the rest of the paper we will assume that all $n$-string graph will be \emph{homogenous}:
meaning that each blue edge is colored homogenous object of $\cat^H$,
see \eqref{eq:G-graded}.  In particular, we use the following
notation: the $j$\ith blue edge is associated with an element
$\bb_j\in G$ such that this edge is colored by an object of
$\cat^H_{\bb_j}$.

Choose a base point $b$ for each closed red component of the diagram $D$.  Consider the element $J_{D}^b$ and  $\Vect_\C$-graph with $\wh{\UH}$-beads $B(D)$
given in Section \ref{SS:UnivInv}.  The diagram $B(D)$ consists of $m+n$ red components and a blue graph
$D_{bl}$ with beads and $n_{D_{bl}}$ edges. We numbered the components and edges
of $D$ starting with the $m$ red closed components then the $n$
ordered open red components and finally the blue edges.  As in Section \ref{SS:UnivInv} we 
express
$$J(\Gamma)=
\bp{\tru^{\otimes m}\otimes \Id^{\otimes n}\otimes \Id}\bp{J_{D}^b} \in \Uo^{\widehat{\otimes} m}\wh  \otimes \,\, (\UH)^{\wh{\otimes} n}\otimes \Hom_\C(V, V')$$
where
$J_{D}^b=\sum w_{b_1}\otimes \cdots \otimes w_{b_m}\otimes
y_1\otimes \cdots \otimes y_n\otimes F_{{\C}}(D_{bl})$
depends on the choice $b$.

The element $J^b_D$ can be seen as a vector-valued entire function with $r(m+n)$ variables: indeed, we can associate to the
$i$\ith red component, the set of variables
$$h_i=(H_{1,i},\ldots,H_{\rk,i}).$$
Recall  
that a morphism $f:[k]V\to [k]V'$ determines a pair $(q,\wb l)$ called the  exponent determined by Equation \eqref{coupon-morphism}.
For each coupon $c$ of $D$ with  $k$ red legs let $(q,l)$ be the exponent corresponding to the morphism
in $c$ with $(\xi^q,l)\in\mathcal Q_k\times(G')^k$.  For such a coupon, let $i_1,i_2,\ldots,i_k\in \{1,\ldots, m+n\}$ be the labels  (with possible repetition) of its adjacent red edges and let $q(c,D)=q(h_{i_1},\ldots,h_{i_k})$ then $\xi^{q(c,D)}\in\mathcal Q_{m+n}$ and $l(c,D)=l(h_{i_1},\ldots,h_{i_k})\in (G')^{m+n}$.
Finally, let
\begin{equation}
  \label{eq:power-coupons}
  \Q_{D}^c=\sum_{c\text{ coupon}}q(c,D)\text{ and }\wb L_{D}^c
  =\sum_{c\text{ coupon}}l(c,D).
\end{equation}

Next, we will use the elements $\Q_{D}^c$ and $\wb L_{D}^c$ to assign an exponent to  $J_{D}^{b}$. 
This exponent is a sum of three contributions: the first is a quadratic element
coming from the crossing between red strands, the second is a linear
element coming from bichrome crossings and the last is coming from the
coupons.

Let 
$\lk=\bp{\lk_{ij}}_{1\leq i, j\leq m+n+n_{D_{bl}}}$ be the 
linking matrix associated $D$.  Remark that the first $(m+n)\times(m+n)$ block of $\lk$ does
not depend on the planar projection but the second does: the coefficient $(i,j)$ is the sum of the sign of crossings of the planar diagram where the strand $j$ is above the strand $i$ (see \cite{FcNgBp14}). 
Recall the element 
$\Q=\Q(h_1,h_2)=\sum_i c_i \otimes c'_i\in\ \Lambda^*\otimes_\Z\Lambda^*\
\subset\ \H\otimes\H$ which corresponds to the symmetric bilinear form $\B$ on $\Lambda$ and
$\mathcal{H}=\xi^{\Q(h_1,h_2)}$.
Let 
\begin{align}\label{eq:Q_Gamma}
\Q_{D}&=\sum_{1\leq k,s\leq m+n}\lk_{ks}\Q(h_k,h_s),\quad\text{ so }\xi^{\Q_{D}}\in\mathcal Q_{m+n},\\
 \wb L_{D}&=\sum_{1\leq k\leq m+n< s\leq m+n+n_{D_{bl}}}(\lk_{ks}+\lk_{sk})\Q(\bb_s,h_k)\in(G')^{m+n}.\
\end{align}

The following theorem is analogous to  \cite[Theorem 3.4]{Ha18a}:
\begin{theo}\label{value of invariant}
  Let $L_{D}$ and $L_{D}^c$ be any representatives
  in $\H^{(m+n)}$ of $\wb L_{D}$ and $\wb L_{D}^c$ in
  $(G')^{m+n}=(\H/\Lambda^*)^{m+n}$, respectively and let
  $P_{D}=\Q_{D}+\Q_{D}^c+L_{D}+L_{D}^c$.  Then the element
  $\xi^{P_{D}}$ is independent of the choice of these representatives
  up to a unit in $\U^{\otimes m+n}$ (i.e. $\xi$ to an integral linear
  combination of the $H_{i,j}$) and
  \begin{equation}\label{E:JDPD}
  J_{D}^{b} \in \xi^{P_{D}}\U^{\otimes m+n}\otimes \Hom_\C(V, V').
  \end{equation}
\end{theo}
\begin{proof}
  The first statement of the theorem follows from Remark
  \ref{rk:powerG'}.  To prove Equation \eqref{E:JDPD} notice that
  the elements $H_{i,j}$ form a commutative algebra and by Lemma
  \ref{lemma changes the order of elements} power quadratic and power
  linear elements commutes with $\U$ so all contributions to the
  exponent can be added in any order.  We just have to check that the
  contribution of crossings is given by the elements $\Q_{D}$ and
  $\wb L_{D}$.

  At each crossing point of type $(k,s)$ in the diagram of $B(D)$
  with beads for $1\leq k, s\leq m+n$, i.e., a red-red type crossing
  point, the Cartan part of the $\mathcal{R}$-matrix gives us the
  factor
  $$\xi^{\ve\Q(h_k,h_s)}\in \mathcal Q_{m+n}$$
  where $\ve\in \{\pm 1\}$ is the sign of the
  crossing. Summing all these contributions, we get a factor
  $$\xi^{\lk_{ks}\Q(h_k,h_s)}\in \mathcal Q_{m+n}.$$
  Next we consider a crossing point $(k,s)$ in the the diagram of
  $B(D)$ for $1\leq k\leq m+n,\ m+n < s\leq m+n+n_{D_{bl}}$, i.e.,
  a red-blue type crossing point. Suppose that the color of the blue
  strand is $V\in \cat^H_\bb$ for $\bb\in
  G$.
  Then any weight $\lambda$ of $V$ is congruent to $\bb$ modulo
  $\Lambda$ and the Cartan part of the element $\mathcal{R}$-matrix
  acting on the $\lambda$ weight space of $V$ gives us a factor (see
  Remark \ref{rk:powerG'})
  $$ \xi^{\ve_{ks}\Q(\lambda,h_k)}\in
  \xi^{\ve_{ks}\Q(\bb,h_k)}\U^{\otimes
    m+n}\subset\mathcal{L}_{m+n}\U^{\otimes
    m+n}.$$ So the
  $k$\ith red component and the
  $s$\ith blue component give the factor
  $$ \xi^{(\lk_{ks}+\lk_{sk})\Q(\bb,h_k)}\in \mathcal{L}_{m+n}\U^{\otimes m+n}.$$
\end{proof}
\begin{Lem}\label{L:reverse}
  Let $D$ be as above and let
  $\stackrel{\leftarrow}D$ be the same diagram with the
  orientation of the
  $i$\ith closed red component reversed ($1\le i\le
  m$).  Here we assume that this component does not pass through a
  coupon. Then
  $J_{\stackrel{\leftarrow}D}^{b}$ is obtained from
  $J_{D}^{b}$ by applying the antipode or its inverse
  to the
  $i$\ith factor, depending if the strand is oriented
  to the top or to the bottom at the base point.
\end{Lem}
\begin{proof}
  We can split the
  $i$\ith component into a sequence of segments whose
  end points are the extremal points (critical points of the second
  coordinate)
  $p_1,\ldots,p_{2k}$ at cups and caps (cups and caps alternate when
  one follows the component).

  Let us assume first that the base point is on
  the segment
  $[p_{2k},p_0]$ which is oriented upward.
  Then multiplying together the beads on each segment, the
  contribution to the
  $i$\ith factor of $J_{D}^{b}$ has the form
  \begin{align*}
  X&=S^{\downarrow}(a_{2k})g^{-\ve_{2k}}S^{\downarrow}(a_{2k-1})g^{\ve_{2k-1}}S^{\downarrow}(a_{2k-2})\cdots
  g^{-\ve_{2}}S^{\downarrow}(a_1)g^{\ve_1}S^{\downarrow}(a_0)\\
  &=a_{2k}g^{-\ve_{2k}}S(a_{2k-1})g^{\ve_{2k-1}}a_{2k-2}\cdots
  g^{-\ve_{2}}S(a_1)g^{\ve_1}a_0
  \end{align*}
  where (see also Figure \eqref{Fig1})
  \begin{enumerate}
  \item
    $S^{\downarrow}(a_j)$ is the product of the beads on the
    $j$\ith segment $[p_j,p_{j+1}]$ where the
    $a_j$ are products of factors of $\RR$ or
    $\RR^{-1}$,
  \item the $j$\ith segment $[p_j,p_{j+1}]$ is oriented upward if
    $j$ is even and downward if $j$ is odd,
  \item $\ve_j$ is $0$ if the cap/cup is oriented to the left and
    $1$ if it is oriented to the right.
  \end{enumerate}
  Then the corresponding contribution to the
  $i$\ith factor of
  $J_{\stackrel{\leftarrow}D}^{b}$ has the form
  $$X'=S(a_0)g^{\ve'_1} a_1 g^{-\ve'_2}\cdots
  S(a_{2k-2})g^{\ve'_{2k-1}}a_{2k-1}g^{-\ve'_{2k}}S(a_{2k})$$
  where
  $\ve'_j=1-\ve_j$ because a left oriented cap/cup become a right
  oriented cap/cup  (once again see Figure \eqref{Fig1}).
  Then the lemma follows from the fact that for each $j$
  $$S\big(g^{-\ve_{2j}}S(a_{2j-1})g^{\ve_{2j-1}}\big)
  =g^{-\ve_{2j-1}}S^2(a_{2j-1})g^{\ve_{2j}}
  =g^{\ve'_{2j-1}}a_{2j-1}g^{-\ve'_{2j}},$$
  so we have $X'=S(X)$.
  
  Since changing the orientation twice recover the initial diagram,
  and $X=S^{-1}(X')$, we get that in case the base point is on a
  downward segment, the new invariant is obtained by applying
  $S^{-1}$.
\end{proof}
The following proposition established the universal character of the
red colored strands:
\begin{Prop}\label{P:univ-red}
  Let $\Gamma$ be a bichrome graph and
  $\Gamma_V$ be obtained by replacing the
  $i$\ith colored red strand by a blue strand colored
  by $V\in\cat^H$ then 
  \begin{enumerate}
  \item if the strand is closed,
    $J(\Gamma_V)$ is obtained from
    $J(\Gamma)$ by applying
    $\tr_\C^V\circ\rho_V$ on the factor of
    $\U^0$ coresponding to this strand.
  \item if the strand is the last open red strand, then
    $J(\Gamma_V)$ is obtained from $J_\Gamma$ by applying
    $\rho_V$ on the last factor of $\UH$.
  \end{enumerate}
\end{Prop}
\begin{proof}
  This is a consequence from the fact that beads can be collected
  without paying attention to the color of the strand by using the
  relations \eqref{eq:mult-beads} and \eqref{eq:cap-beads}.  Then
  Penrose graphical calculus in
  $\Vect_\C$ does not depends of the embedding of a strand because the
  braiding in
  $\Vect_\C$ is trivial so we can assume that the strand is either a
  simple circle or a vertical string.
\end{proof}
\section{Invariant of $3$-manifolds}\label{Section of inv of 3-manifolds}
The goal of this section is to define an invariant of the compatible
triples $(M,\Gamma,\omega)$
where, loosely speaking,   $M$ is a
$3$-manifold, $\Gamma$ is a closed bichrome graph inside
$M$ and $\omega\in
H^1(M\setminus\Gamma,G)$. The main ingredients of the construction are
the universal invariant defined in previous section,
  the discrete Fourier transforms, the graded integral, the modified trace and the
modified integral.

\subsection{Compatible triple}\label{ss:Compat}
A closed bichrome ribbon
graph is a bichrome ribbon graph with no boundary vertices.  
Let $\Gamma$ be a $n$-string link graph in $M=\R^2\times[0,1]$ or a closed bichrome graph embedded in any oriented
closed
3-manifold $M$.
Let
$\wt \Gamma$ be the smoothing of $\Gamma$ in $M$.
Let $\omega\in H^1(M\setminus\wt\Gamma,G)$.
We now defined the notion of a \emph{compatible} triple, first in a simple situation and then the general case.  First, we 
consider the situation with the following two requirements:  1) 
the bilinear form
$\B$ (involve in the braiding) is non degenerate and 2) every closed
red component of
$\Gamma$ is a framed knot with no coupons.
With these requirements
we say the
triple $(M,\Gamma,\omega)$ is \emph{compatible} if:
\begin{enumerate}
\item For each blue edge $e$ of $\Gamma$ colored by $V\in\cat^H_g$ with
  oriented meridian $m_e$, one has $g=\omega(m_e)$.
\item For each $i$\ith closed red component of
  $\Gamma$ with oriented parallel $\ell_i$, one has $\omega(\ell_i)=0\in G$.
\end{enumerate}
For the general case, consider the element $\vec\omega_m\in
G^{m+n}$ obtained by evaluation of
$\omega$ on the meridians of the $m+n$ red components of
$\wt\Gamma$.  By Lemma \ref{lemma transfers quadratic to linear}, the
quadratic element
$\Q^c_\Gamma$ of Equation \eqref{eq:power-coupons} produces an element
$\wt
\Q^c_\Gamma(\vec\omega_m,h_\bullet)\in(G')^{m+n}$.  Let
$\lc_i^c\in G'$ be the $i$\textsuperscript{th} component of $\wt
\Q^c_\Gamma(\vec\omega_m,h_\bullet)+ \wb
L_\Gamma^c$.  Similarly, for $\Q(h_1,h_2)$ the exponant of
$\HH$ and for any $\ba\in
G$, we consider the element $\Q(\ba,h)\in
G'$.  Explicitly, if $\Q={\sum_i c_i(h)\otimes
  c'_i(h)}$, then $\Q(\ba,h)=\sum_i c_i(\ba)
c'_i$ where the
$c_i(\ba)\in\C/\Z$ are complex numbers modulo integers.  Then we say the
triple $(M,\Gamma,\omega)$ is \emph{compatible} if:
\begin{enumerate}
\item For each blue edge $e$ of $\Gamma$ colored by $V\in\cat^H_\bb$ with
  oriented meridian $m_e$, one has $\bb=\omega(m_e)$.
\item For each $i$\ith red closed component of
  $\Gamma$ with oriented parallel
  $\ell_i$, one has $2\Q(\omega(\ell_i),h)+\lc^c_i=0\in G'$.
\end{enumerate}

If $(M,\Gamma,\omega)$ is a compatible triple where $M$ is a closed
3-manifold presented by surgery on a link $L\subset S^3$ then
$S^3\setminus(\wt\Gamma\cup L)\subset M\setminus\wt\Gamma$ and in what follows we will 
still denote $\omega$ as the restriction of the cohomology class $\omega$ in
$H^1(S^3\setminus(\wt\Gamma\cup L),G)$.
\begin{Prop} \label{compatible of L union Gamma}
Let $(M,\Gamma,\omega)$ be a compatible triple where $M$ is a be an oriented closed
3-manifold presented by surgery on a link $L\subset S^3$. Then
  $(S^3,\Gamma\cup L,\omega)$ is also a compatible triple
  where the components of $L$ are all red.
\end{Prop}
\begin{proof}
 Condition (1) of a compatible triple is still satisfied because the blue edges $\Gamma$ and $\Gamma\cup L$  are the same.
 We need to check condition (2) for the red components
  of $L$.  The surgery link $L$ has no coupon and
  its parallels bound discs in $M\setminus\wt\Gamma$.  Hence for
  each parallel $\ell_i$ of a component of $L$, $\omega(\ell_i)=0$.
  Hence the triple $(S^3,\Gamma\cup L,\omega)$ is compatible.
\end{proof}

Kirby's Theorem (given in \cite{Kir1978}) loosely says that  two surgery presentations of a manifold can be connected by a sequence of isotopies and Kirby moves.   Consider the Kirby move given in Figure \ref{fig:Kirby}.  Notice that the Kirby 0 and II moves only occur on closed red components without coupon.  
\begin{figure}
  \centering
  \begin{tabular}{|ccc|}
    \hline &&\\
    $\put(-5,8){\ms{\ba}}\epsh{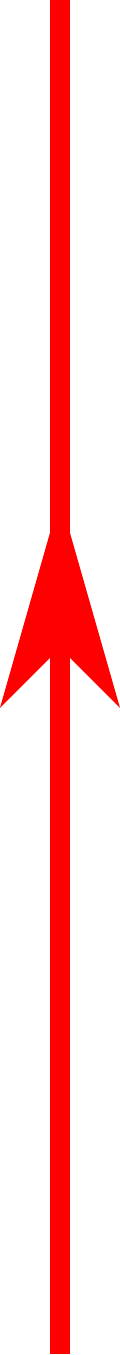}{8ex}\longleftrightarrow
    \epsh{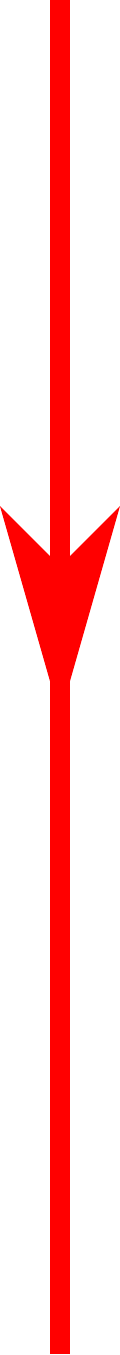}{8ex}\put(0,8){\ms{-\ba}}$&&
                                           $\epsh{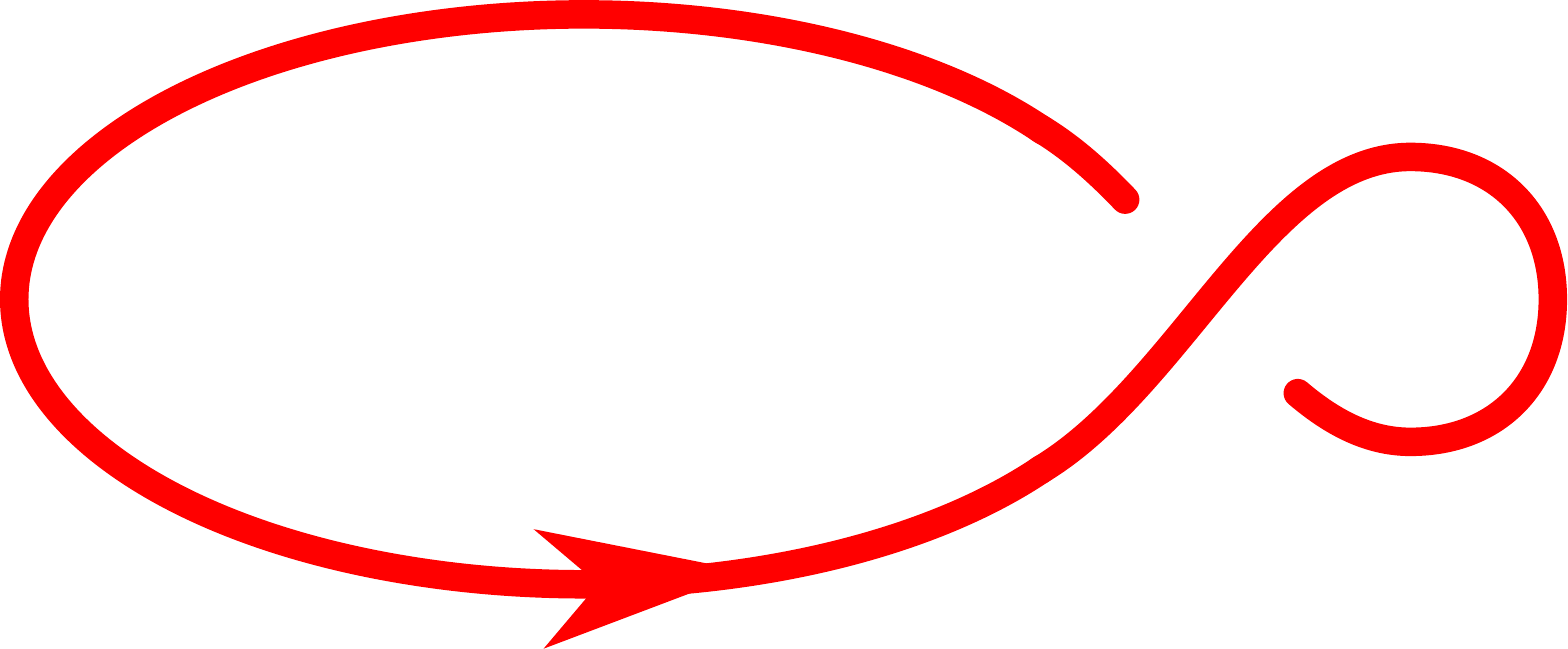}{4ex}\put(-31,-2){\ms{\overline0}}\longleftrightarrow\emptyset
                                           \longleftrightarrow\epsh{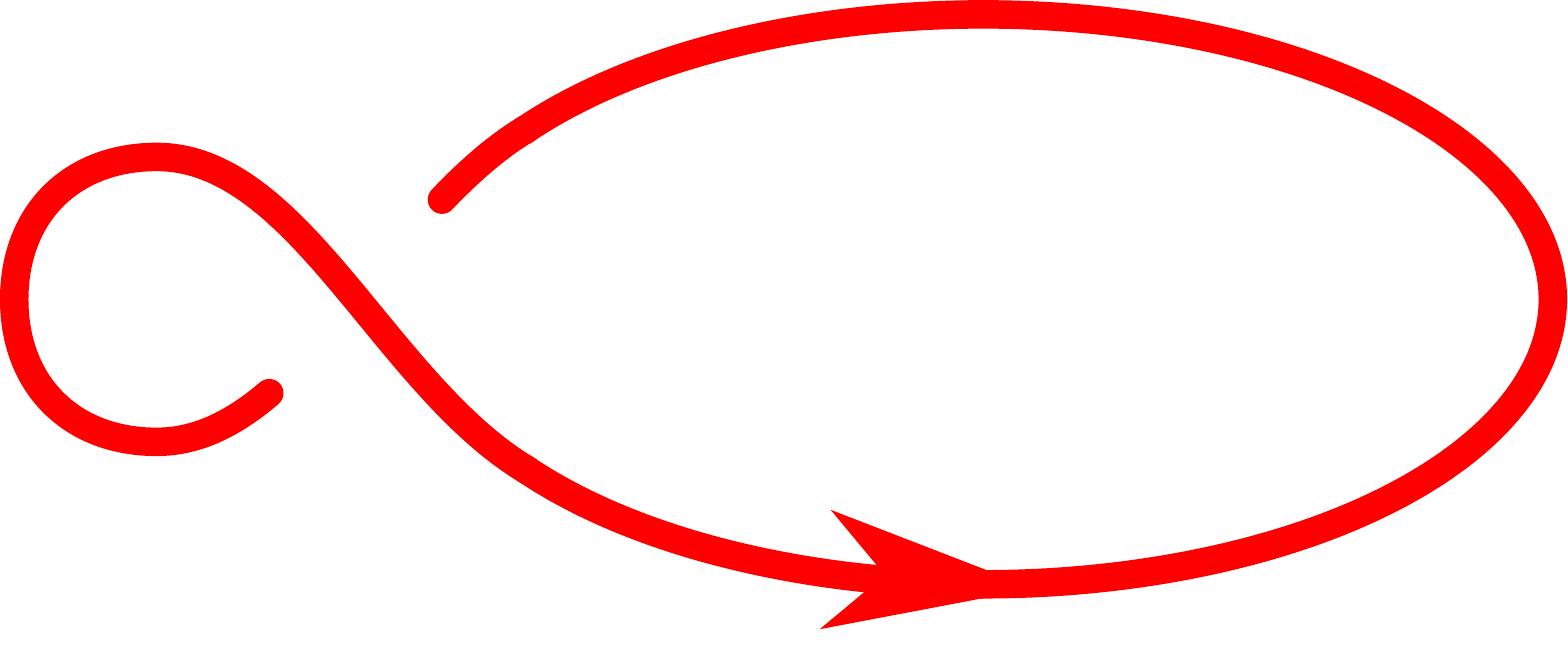}{4ex}\put(-24,-2){\ms{\overline0}}$\\&&\\
  Kirby 0 move&\hspace*{4ex}& Kirby I moves
  \\
    \hline
  \end{tabular}
  \\[2ex]
  \begin{tabular}{|c|}
    \hline \\
  {$\epsh{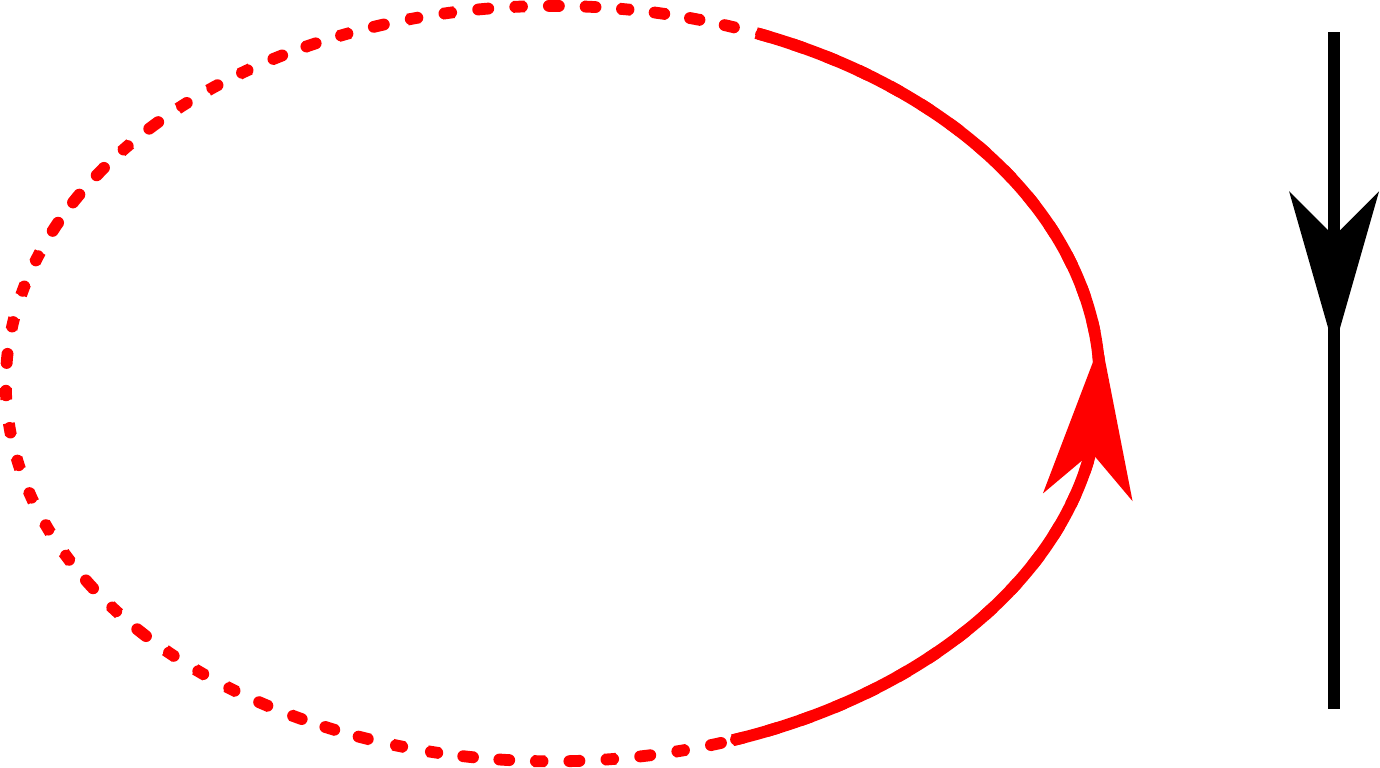}{9ex}\put(-11,10){\ms{\ba}}\put(-24,1){\ms{\bb}}
  \quad\longleftrightarrow\quad
  \epsh{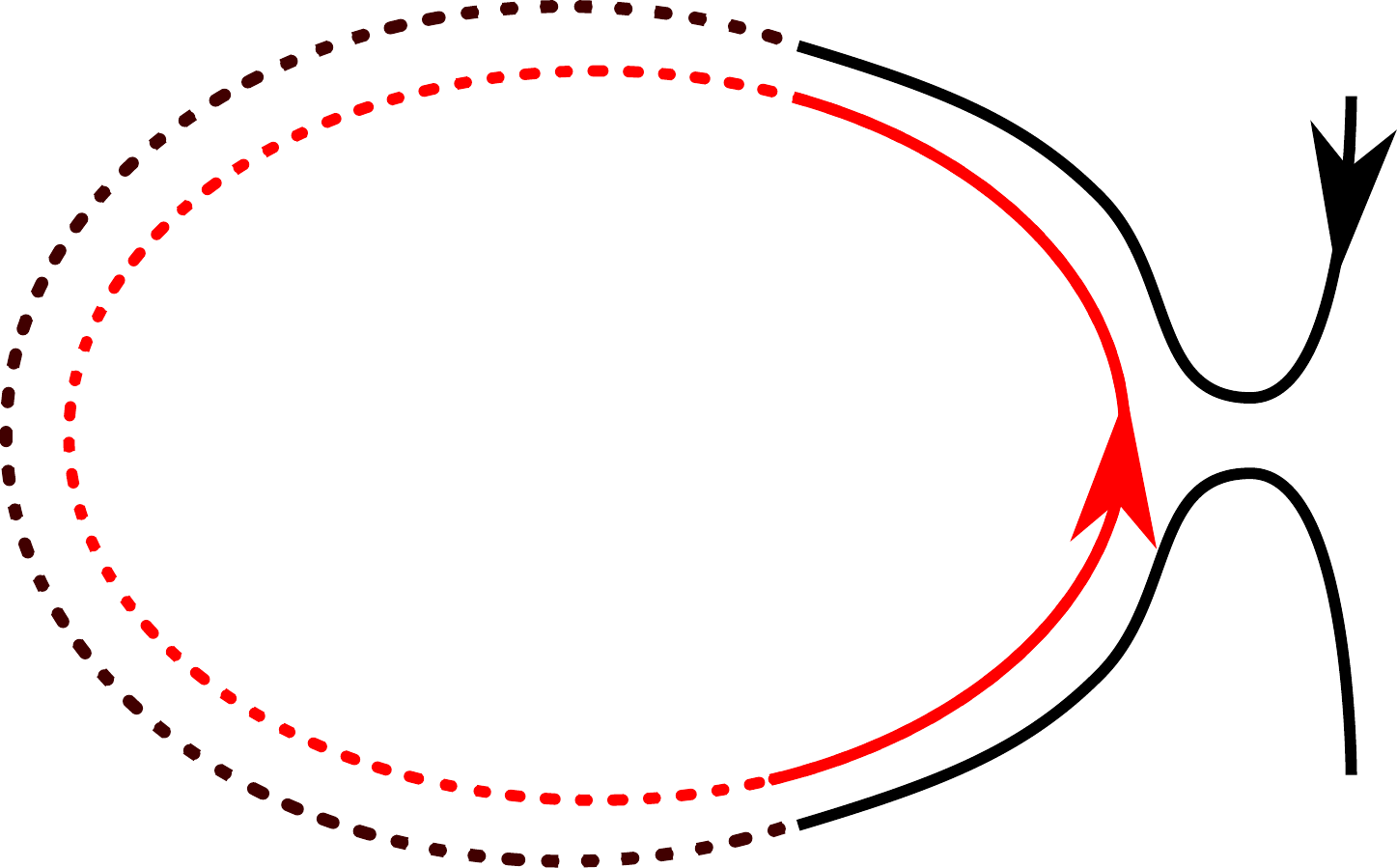}{10ex}\put(-11,15){\ms{\ba}}\put(-39,1){\ms{\bb-\ba}}$}
  \\  \\
  Kirby II move (black edge can be blue or red)
  \\
    \hline
  \end{tabular}
  \\[2ex]
  The label on edges is the value of $\omega$ on their oriented meridian.
  \caption{Colored Kirby moves}
  \label{fig:Kirby}
\end{figure}
\begin{theo}\label{T:Kirby}
Let $(M,\Gamma,\omega)$ and $(M',\Gamma',\omega')$) be  compatible triple where $M$ and $M'$ are oriented closed 3-manifolds.  Let $M$ and $M'$ be presented by surgery on the links $L$ and  $L'$, respectively.  Suppose $f:M\to M'$ is an be an orientation-preserving diffeomorphism such that $f(\Gamma) = \Gamma'$ as framed graph and $f^*(\omega')=\omega$.  Then
 there exists a finite sequence of  isotopies, Kirby 0, Kirby I, Kirby II moves (see Figure  \ref{fig:Kirby}) relating $(S^3,\Gamma\cup L,\omega)$ and  $(S^3,\Gamma'\cup L',\omega')$  such that the induced isomorphism is isotopic to $f$.
\end{theo}
\begin{proof}
In \cite{Rob1997}, the Kirby theorem is extended to the case of manifolds containing graphs.  It is easy to see that the condition $f^*(\omega') = \omega$ implies that this extension give a sequence of moves relating $(S^3,\Gamma\cup L,\omega)$ and  $(S^3,\Gamma'\cup L',\omega')$.   Note a similar result was considered in  Theorem 5.2 of \cite{FcNgBp14}.
\end{proof}

\subsection{Invariant of $3$-manifolds}
Let $\Gamma$ be a $n$-string link graph with $m$ red closed
components.  Let $D$ be a planar diagram representing $\Gamma$ with a
choice of base point $b$ for each closed red component. Recall we constructed in Step 2 of Subsection
\ref{SS:UnivInv} the element
$$J_{D}^b \in (\UH)^{\wh\otimes m+n}
\otimes\Hom_{\C}(V,V')
\simeq\Ce\bp{h_\bullet}\otimes W^{\otimes m+n}
\otimes\Hom_{\C}(V,V').$$
So $J_{D}^b$ is a vector-valued entire function with set of variables
$h_\bullet=\{H_{i,j}\}_{i=1,\ldots,\rk,\;j=1,\ldots, m+n}$ which are a
basis of $\H^{(m+n)}$.

\begin{Lem}\label{proof univ invariant is l-periodic}
 Suppose $\Gamma$ is equipped with a compatible cohomology class $\omega$.  Let $\va=(\ba_1,\ldots, \ba_{m+n})$ be the values of $\omega$ on the meridians of the red components.  
  Then
  $
  J_{D}^b$
   is $\ell$-periodic in the first $rm$ variables $\{H_{i,j}\}_{i=1,\ldots,\rk,\;j=1,\ldots, m}$ on $\Lat_\va$.
 \end{Lem}
\begin{proof}
Theorem \ref{value of invariant} implies 
$$J_{D}^b \in \xi^{P_{D}}\U^{\otimes m+n}\otimes \Hom_\C(V, V')\subset \mathcal
Q_{m+n}\mathcal L_{m+n}\U^{\otimes m+n}\otimes \Hom_\C(V, V')
$$
where $P_{D}=\Q_{D}+\Q_{D}^c+ L_{D}+ L_{D}^c\in S(\H^{(m+n)})$.
 Let us denote
$M=\R^2\times[0,1]$.  Now we show that
the compatibility of
$\omega\in H^1(M\setminus \Gamma;G)$ implies that the function
$\xi^{P_{D}}$ is $\ell$-periodic in the first $rm$ variables on
$\Lat_\va$.

We check the periodicity of $J^b_D(h_1,\ldots,h_{m+n})$ in the set of
variables $h_i=(H_{1,i},\ldots,H_{\rk,i})$ for $1\leq i\leq m$ by
computing the evaluation of $\xi^{P_{D}}$ at $\alpha+\ell \lambda_{i}$
with $\alpha\equiv\va$ mod $\Lambda^{m+n}$ and
$\lambda_i=(0,\ldots,\lambda,\ldots,0)\in(\H^*)^{m+n}$
($\lambda$ at the $i$\ith position) for some $\lambda\in\Lambda$.
 In the following we used the polarisation
of quadratic element defined in Equation \eqref{eq:polarisation}.
   One has
  \begin{align*}
    P_{D}&(\alpha+\ell \lambda_i)\\ 
              &=\Q_{D}(\alpha+\ell \lambda_i)
                +\Q_{D}^c(\alpha+\ell \lambda_i)\\
              & \qquad \qquad \qquad \qquad + L_{D}(\alpha
                +\ell \lambda_i)+ L_{D}^c(\alpha
                +\ell \lambda_i)\\
              &=P_{D}(\alpha)+ \Q_{D}(\ell\lambda_i)
                +\Q_{D}^c(\ell\lambda_i)+\ell \wt \Q_{D}( \alpha,\lambda_i)
                \\
              &\qquad +\ell\wt \Q_{D}^c( \alpha,\lambda_i)
                +\ell L_{D}(\lambda_i)
                +\ell L_{D}^c(\lambda_i) \\
              &=P_{D}(\alpha)+ \Q_{D}(\ell\lambda_i)
                +\Q_{D}^c(\ell\lambda_i) + X ,
  \end{align*}
  where $X$ mod $\ell\Z$ is given by 
  $$\bar X =  \wt \Q_{D}( \ell\va, \lambda_i)+\wt  \Q_{D}^c( \ell\va, \lambda_i)+ \ell\wb L_{D}(\lambda_i)+ \ell\wb L_{D}^c(\lambda_i).$$
  It is clear that
  $\Q_{D}(\ell\lambda_i)+\Q_{D}^c(\ell\lambda_i)=
  0 \text{ mod } \ell\Z$.  Now recall that $\Q=\Q(h_1,h_2) $ is also
  symmetric and linear in each set of variable thus
  \begin{align*}
    \wt \Q((h_1,h_2),(h'_1,h'_2))&=\Q(h_1+h'_1,h_2+h'_2)-\Q(h_1,h_2)-\Q(h'_1,h'_2)\\&=\Q(h_1,h'_2)+\Q(h_2,h'_1)
  \end{align*}
  then from \eqref{eq:Q_Gamma}, we get
  $$\wt \Q_{D}( \ell\va, \lambda_i)=\!\!\sum_{1\leq k,s\leq m+n}\!\!\lk_{ks}\bp{\Q(\ell\ba_k,\delta_i^s\lambda)+\Q(\ell\ba_s,\delta_i^k\lambda)}$$
  and
  $$
  \ell\wb L_{D}(\lambda_i)=\sum_{1\leq k\leq m+n<s\leq
    m+n+n_{D_{bl}}}\bp{\lk_{ks}+\lk_{sk}}\Q(\ell
  \bb_s,\delta_i^k\lambda).
  $$
  Let
  $X_D = \wt \Q_{D}( \ell\va,
  \lambda_i)+ \ell\wb L_{D}(\lambda_i)$.
  As $\bb_s=\omega (m_s)$,
   we have
  
  \begin{align*}
    X_D
               &=\sum_{1\leq k\leq m+n}\lk_{ki}\Q(\ell\ba_k,\lambda_i)+\lk_{ik}\Q(\ell\ba_k,\lambda_i)\\
               & \qquad\qquad\qquad+\sum_{m+n<s\leq m+n+n_{D_{bl}}}\bp{\lk_{is}+\lk_{si}}\Q(\ell\ba_s,\lambda_i )\\
               &=\Q\bp{\sum_{k=1}^{m+n+n_{D_{bl}}}\lk_{ik}\omega(m_k)+\sum_{k=1}^{m+n+n_{D_{bl}}}\lk_{ki}\omega(m_k), \ell\lambda_i}.
  \end{align*}
  In addition, for any $1\leq i \leq m$, in homology,
  $\sum_{k=1}^{m+n+n_{D_{bl}}}\lk_{ik}[m_k]=\sum_{k=1}^{m+n+n_{D_{bl}}}\lk_{ki}[m_k]$
  is the homology class of the parallel of the $i$\ith
  red closed component of $L\cup \Gamma$. So we have
  $$\sum_{k=1}^{m+n+n_{D_{bl}}}\lk_{ik}\omega(m_k)=\sum_{k=1}^{m+n+n_{D_{bl}}}\lk_{ki}\omega(m_k) =\omega(\ell_i).$$
  It implies that
  $X_D=2\ell\Q\bp{\omega(\ell_i),\lambda_i}$.  Thus with $\wt  \Q_{D}^c( \va, \lambda_i)+ \wb L_{D}^c(\lambda_i)=\sum_i\lc^c_i(h_i)$,
   we
  get
  \begin{align*}
    \bar X &=\ell\bp{2\Q\bp{\omega(\ell_i),\lambda_i}+ \sum_j\lc^c_j(\lambda_i)}\\
    &=\ell\bp{2\Q\bp{\omega(\ell_i),\lambda_i}+ \lc^c_i(\lambda_i)}\text{ mod }\ell\Z.
  \end{align*}
  By assumption $(M,\Gamma,\omega)$ is compatible triple,
  i.e.,
  $$2\Q\bp{\omega(\ell_i),h_i}+ \lc^c_i(h_i)=0\text{ modulo }\Lambda^*,$$ one gets
  $\bar X = 0 \mod \ell\Z$.  It implies that
  $$\xi^{P_{D}(\alpha+\ell\lambda_i)}=\xi^{P_{D}(\alpha)}.$$
  Hence one concludes $\xi^{P_{D}}$ is $\ell$-periodic on
  $\Lat_\va$ in the variables $H_{i,j}$ associated to the close red
  components.
\end{proof}
Lemma \ref{proof univ invariant is l-periodic} says that with the cohomology class $\omega$ then the entire function  $\xi^{P_{D}}u$ is $\ell$-periodic in the first $rm$ variables  on $\Lat_\va$, where $J_{D}^b= \xi^{P_{D}}\sum_i u_i \otimes f_i $. 
 Let  $$J_{D,\omega}^b=\sum_i\pi(\xi^{P_{D}}u_i)\otimes f_i $$  then Proposition \ref{P:quad-lin-modI} implies that 
$$J_{D,\omega}^b \in \U_{\ba_1}\otimes\cdots
  \otimes\U_{\ba_m}\otimes\LU_{\ba_{m+1}}\otimes\cdots\otimes\LU_{\ba_{m+n}}
  \otimes\Hom_\C(V, V')$$
where $\pi$ is the projection  $(\UH)^{\wh\otimes m+n}\to (\UH)^{\wh\otimes m+n}/ I_\va^H$.

\begin{Prop}\label{P:graded-Hennings}
  Let $\Gamma$ be a $n$-string link graph with $m$ red
  closed components in $\R^2\times[0,1]$ equipped with a compatible
  cohomology class $\omega$.  Let $\va=(\ba_1,\ldots, \ba_{m+n})$ be
  the values of $\omega$ on the meridians of the red components.
  Let
  \begin{align*}
    \label{eq:Fmu}
    F_\si(\Gamma,\omega)
    &=(\si_{\ba_1}\otimes\cdots\otimes\si_{\ba_m}\otimes\Id)\bp{J^b_{\Gamma,\omega}}\\
    &\in\LU_{\ba_{m+1}}\otimes\cdots\otimes\LU_{\ba_{m+n}}
      \otimes\Hom_\C(V, V').
  \end{align*}
  $$$$ Then $F_\si(\Gamma,\omega)$ does not depend
  on the base points, and is invariant by Kirby 0 and Kirby II moves.
\end{Prop}
\begin{proof}
  Let us call $L$ the sub-link of $\Gamma$ formed by the $m$ red
  closed components.  If one moves the base point of the
  $i$\ith component of $L$ then the universal
  invariant $J^b_\Gamma$ only change by terms which have a commutator
  $c_i$ at the $i$\ith factor.  If such a term has non
  zero weight, then so will be its representant $\wb c_i\in\U_{\ba_i}$
  and then $\si_{\ba_i}(\wb c_i)=0$ by Equation \eqref{eq:muW0}.  Now
  if the weight of $\wb c_i$ is zero, it can be written as the
  commutator of two elements in $\LU_{\ba_i}$.  From the proof of
  Lemma \ref{reducing commutator} we see that $\wb c_i$ belongs to
  $[\U_{\ba_i},\LU_{\ba_i}]\cap\U_{\ba_i}=[\U_{\ba_i},\U_{\ba_i}]$.
  Then, from Equation \eqref{eq:mucyc} we see that
  $\si_{\ba_i}(\wb c_i)=0$.  Hence $F_\si(\Gamma,\omega)$ does not
  depend of the base points.

  We now prove $F_\si(\Gamma,\omega )$ does not change under
  orientation reversal of components of $L$ (i.e.,\ the Kirby 0 move).  First, by Lemma
  \ref{L:reverse}, changing the orientation of a component of $L$
  changes $J^b_\Gamma$ by applying an antipode on the factor
  corresponding to the component, so the property of $G$-trace in
  Equation \eqref{eq:muS} imply $F_\si(\Gamma,\omega )$ does not
  depend on the orientation.

\begin{figure}
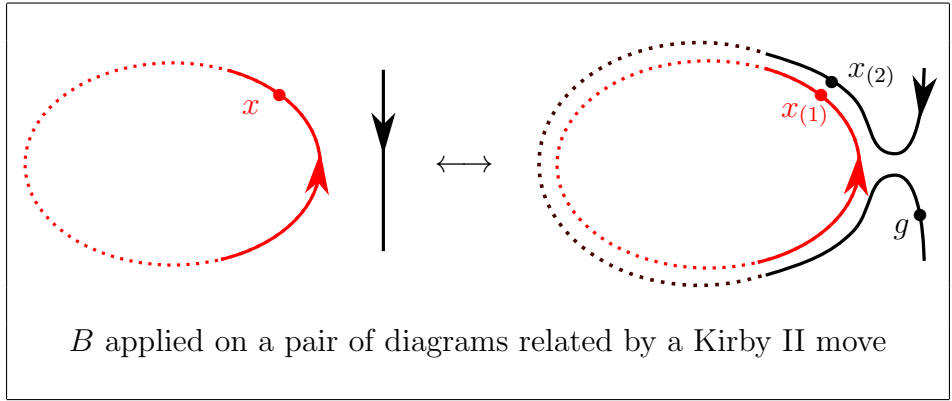

  \centering
  \begin{tabular}{|c|}
    \hline \\
    {$
    \epsh{KIIa}{15ex}
     \put(-46,26){{\color{red}$\bullet$}}
     \put(-57,22){{\color{red}$x$}}
  \quad\longleftrightarrow\quad
    \epsh{KIIb}{18ex}
    \put(-46,26){{\color{red}$\bullet$}}
     \put(-58,21){{\color{red}$x_{(1)}$}}
    \put(-42,31){{$\bullet$}}
    \put(-33,36){{$x_{(2)}$}}
    \put(-9,-19){{$\bullet$}}
    \put(-16,-23){{$g$}}
    $}
  \\  \\
  $B$ applied on a pair of diagrams related by a Kirby II move
  \\
  \\
    \hline
  \end{tabular}
  \\
  \caption{Invariance by Kirby II move}
  \label{Fig4}
\end{figure}
  Finally, to see the  Kirby II move hold, we consider a red component $K\subset L$
  with beads with product $x\in\wh\UH$ and a blue or red vertical edge
  $e$ with locally no bead or equivalently a bead $1$, as in the first
  component of Figure \ref{Fig4}.  Here the base point of $K$ is
  assumed to be near the vertical edge and the orientation of $K$ as
  in the figure. Let $(\omega',\va')$ be the image of $(\omega,\va)$
  by the second Kirby move.  After the move, a pivotal element $g$ has
  appeared on the slidding edge $e$ and since the two curves are
  parallel with the same orientation, all the elements along the
  component $K$ are replaced by $x_{(1)}$ on $K$ and $x_{(2)}$ on $e$
  where $\Delta x=x_{(1)}\otimes x_{(2)}$.  Since modulo $I_\va$,
  $x\in\U_{\bar\beta}$ we have that modulo $I_{\va'}$,
  $\Delta x\in\U_{\bar\beta-\bar\alpha}\otimes \U_{\bar\alpha}$.  Then
  \eqref{eq:rightsi} implies that the values of $F_\si$ before and
  after the Kirby II move are the same.
\end{proof}

Up to isotopy, one can identify closed bichrome graphs in
$\R^2\times[0,1]$ with closed bichrome  graphs in $S^3$.  The
same holds for pair $(\Gamma,\omega)$ where $\omega$ is a compatible
cohomology class in the complement of the smoothing $\wt \Gamma$ of
a closed bichrome graph
$\Gamma$.  Hence Proposition \ref{P:graded-Hennings} induce a complex
valued invariant of compatible triple $(S^3,\Gamma,\omega)$ that we
still denote by $F_{\si}(\Gamma,\omega)\in\C$.

\begin{theo}[Graded Hennings invariant]\label{T:GHI}
 Let $(M,\Gamma,\omega)$ be compatible triple where $M$ is a oriented closed 3-manifolds.  Let $L\subset S^3$ be a framed link which is a surgery presentation of $M$.  Denote $\omega$ as the restriction of the cohomology class $\omega$ in
$H^1(S^3\setminus(\wt\Gamma\cup L),G)$.
  Define
  \begin{equation}
    \He( M,\Gamma,\omega ) =\delta^{-s} F_{\si}(L\cup\Gamma,\omega)
  \end{equation}
  where
  $\delta= \si_{\overline{0}}(g_{\wb 0}^{-1}\theta_{\overline{0}})$
  and $s$ is the signature of the linking matrix of $L$.  Then
  $\He$ is a well defined invariant of 
  diffeomorphism class of $(M,\Gamma,\omega)$.
\end{theo}
\begin{proof}
  From Proposition \ref{compatible of L union Gamma}, we have 
  $(S^3,\Gamma\cup L,\omega)$ is a compatible triple and by Theorem
  \ref{T:Kirby} any two such presentation are related by Kirby moves.
  By Proposition \ref{P:graded-Hennings},
  $F_{\si}(L\cup\Gamma,\omega)$ is invariant by Kirby 0 and Kirby II
  moves.
  For the Kirby I move, if we add a $\pm1$-framed unknot to
  $L$ with meridian $m_0$, then $\omega(m_0)=\overline{0}$ because the
  class of $m_0$ is zero in $H_1(M\setminus \wt \Gamma,\Z)$.  Then the
  invariant $F'$ is multiplied by
  $\lambda_{\overline{0}}(\theta_{\overline{0}})^{\mp 1}=\delta^{\mp
    1}$. At the same time the signature of the linking matrix $s$
  changes by $\pm 1$, so $\He( M,\Gamma,\omega )$ does not
  change under the Kirby I move.
\end{proof}
In many known examples the invariant $ \He$ is zero for ``generic'' compatible triples.  In particular, in the case of Example \ref{ex1}, the invariant $ \He$ is zero when $\Gamma$ is colored with a projective module or when $\omega$ has non-integral values.  However, as we will now explain, m-traces can be used to renormalize $\He$ to define a non-zero invariant for these ``generic'' compatible triples.

\begin{Prop}\label{P:Fmu-equiv} Let $\Gamma$ be a $n$-string link graph
  and $(\R^2\times[0,1],\Gamma,\omega)$ be a compatible triple, then
  $F_\si(\Gamma,\omega)\in\LU_{\ba_{1}}\otimes\cdots\otimes\LU_{\ba_{n}}
  \otimes\Hom_\C(V, V')$ is equivariant: write
  $F_\si(\Gamma,\omega)=\sum_ix_i\otimes f_i$ then let
  $\UH_{\otimes\va}=\UH_{\ba_{1}}\wh\otimes\cdots\wh\otimes\UH_{\ba_{n}}$
  where the factors are equipped with the action of $\wh\UH$ by left
  multiplication.  Then
  $$\sum_iL_{x_i}\otimes f_i\in\Hom_\C\bp{\UH_{\otimes\va}\otimes V,
    \UH_{\otimes\va}\otimes V'}$$ is a morphism of $\wh\UH$-modules.
\end{Prop}
\begin{proof}
  The proof is by induction on the number $m$ of closed red
  components of $\Gamma$.  If $m=0$, by Equation \eqref{E:EquivarJ} of Theorem
  \ref{T:univ-inv}, the universal invariant $J(\check\Gamma)$ is
  $\wh \UH$-equivariant.  Since the projection modulo the ideals
  $I_\ba$ is equivariant, then $F_\si(\Gamma,\omega)$ is
  $\wh \UH$-equivariant.

  Now assume $m>0$ and consider a ($n$+1)-string link graph
  $\check\Gamma$ whose left braid closure of the first strand produces
  $\Gamma$.  Let $\check D$ be a diagram of $\check\Gamma$ and $D$ its
  left braid closure obtained by joining the $m$ top endpoints of
  $\check D$ to its $m$ bottom endpoints using an arc disjoint from
  the diagram on its left.  We still denotes by $\omega$ the
  restriction of the cohomology class to
  $\R^2\times[0,1]\setminus\check\Gamma\subset\R^2\times[0,1]\setminus\Gamma$. Then
  by induction $F_\si(\check \Gamma,\omega)$ is $\wh \UH$-equivariant.  Now
  computing the invariants using $\check D$ and $D$ with the new base point
  on the left arc, we have that if
  $F_\si(\check \Gamma,\omega)=\sum_iy_i\otimes z_i\otimes f_i$ with
  $y_i\in\U_{\ba}$ then
  \begin{equation}
    \label{eq:Fclosure}
    F_\si(\Gamma,\omega)=\sum_i\si_{\ba}(y_{i}g^{-1})z_i\otimes f_i.
  \end{equation}
  Now for $u\in\U$, let $\Delta^{(n+2)}u=u_a\otimes u_b\otimes u_c$
  with $u_a\in\U$, $u_b\in\U^{\otimes n}$ and $u_c\in \U$. We can
  write $\Delta^{n+1}u=\ve(u_a)u_b\otimes u_c$ and then
  \begin{align*}
    F_\si(\Gamma,\omega)(\Delta^{n+1}u)
    &=\sum_i\si_{\ba}(y_{i}\ve(u_a)g^{-1})z_iu_b\otimes f_i(u_c\cdot)\\
    &=\sum_i\si_{\ba}(y_{i}u_{a(2)}S^{-1}(u_{a(1)})g^{-1})z_iu_b\otimes f_i(u_c\cdot)\\
    &=\sum_i\si_{\ba}(u_{a(2)}y_{i}g^{-1}S(u_{a(1)}))u_bz_i\otimes u_cf_i\\
    &=\sum_i\si_{\ba}(S(u_{a(1)})u_{a(2)}y_{i}g^{-1})u_bz_i\otimes u_cf_i\\
    &=\sum_i\ve(u_a)\si_{\ba}(y_{i}g^{-1})u_bz_i\otimes u_cf_i
      =(\Delta^{n+1}u)F_\si(\Gamma,\omega),
  \end{align*}
  where the third equality uses the equivariance of
  $F_\si(\check \Gamma,\omega)$ and the cyclicity of the symmetrised
  integral \eqref{eq:mucyc}.  For the $\H$-linearity, we need to use
  Equation \eqref{eq:muW0}.  By this equality only the terms of
  \eqref{eq:Fclosure} where $y_i$ commute with any element $H$ of $\H$
  will contribute to the sum. But since $H$ is primitive,
  $\Delta^{(n+2)}H=H\otimes1+1\otimes\Delta^{(n+1)}H$ so projecting
  Equation \eqref{eq:Fclosure} on weight zero space for the first
  factor, we have
  $$\sum_{|y_i|=0}y_i\otimes \bp{(z_i\otimes f_i)\Delta^{(n+1)}H}=
  \sum_{|y_i|=0}y_i\otimes \bp{\Delta^{(n+1)}H(z_i\otimes f_i)}$$
  which implies the $H$ linearity of $F_\si(\Gamma,\omega)$.
\end{proof}
\begin{Def}\label{D:graph-admissible triple}
  The compatible triple $(M,\Gamma,\omega)$ is called a
  \emph{graph-admissible triple} if there exists a blue edge of $\Gamma$
  colored by $V\in \Proj(\cat^H)$.
\end{Def}

Let $(S^3, \Gamma, \omega)$ be a graph-admissible triple with $\Gamma$
is a closed bichrome graph embedded in $S^3$.  A cutting presentation
of $\Gamma$ is a bichrome graph $\Gamma_V\in \R^2\times[0,1]$ with
bottom and top end point $(V,+)$ whose braid closure is isotopic to
$(S^3, \Gamma)$.  It is equipped with the induced cohomology class
still denoted $\omega$.  Then Proposition \ref{P:Fmu-equiv} implies that $F_ \si(\Gamma_V,\omega)\in\Hom_\C\bp{ V,V}$ is a morphism of $\wh\UH$-modules (as $n=0$).  But $V$ is finite dimensional and so $F_ \si(\Gamma_V,\omega)$ is a morphism of $\UH$-modules, i.e.\    $F_ \si(\Gamma_V,\omega)\in\End_{\cat^H}(V)$.  Let
\begin{equation}
  \label{eq:F'}
  F'_\si(\Gamma, \omega)=\tv_V\bp{F_ \si(\Gamma_V,\omega)}.
\end{equation}
\begin{Prop}\label{P:F'invKirby} The modified invariant $F'_\si(\Gamma, \omega)$ of
  Equation \eqref{eq:F'} does not depend on the choice of a cutting
  presentation $\Gamma_V$.  Furthermore, it is invariant by Kirby 0
  and Kirby II move and for any closed compatible triple
  $(S^3,\Gamma',\omega')$, the modified invariant of the disjoint union
  is given by
  $$F'_\si(\Gamma\sqcup\Gamma',\omega'')=
  F'_\si(\Gamma, \omega)F_\si(\Gamma', \omega')$$
  where $\omega''$ is the cohomology class induced by $\omega$ and
  $\omega'$.
\end{Prop}
\begin{proof}
  The proof of the first statement is similar to that in
  \cite{NgBpVt09}: indeed, if $\Gamma_V$ and $\Gamma_ {V'}$ are both
  cutting presentation of $\Gamma$ then there exists $\Gamma_{V,V'}$
  such that $\Gamma_V$ and $\Gamma_ {V'}$ are the partial braid
  closure of $\Gamma_{V,V'}$, and $\Gamma_{V,V'}$ conjugated by the
  braiding, respectively.  Then we have
  \begin{align*}
    \tv_V\bp{F_ \si(\Gamma_V,\omega)}
    &=\tv_V\bp{\ptr_{V'}\bp{F_ \si(\Gamma_{V,V'},\omega)}}\\
    &=\tv_{V\otimes V'}\bp{F_ \si(\Gamma_{V,V'},\omega)}\\
    &=\tv_{V'\otimes V}\bp{c_{V',V}^{-1}F_ \si(\Gamma_{V,V'},\omega)c_{V',V}}\\
    &=\tv_{V'}\bp{\ptr_{V}\bp{c_{V',V}^{-1}F_ \si(\Gamma_{V,V'},\omega)c_{V',V}}}\\
    &=\tv_{V'}\bp{F_ \si(\Gamma_{V'},\omega)}
  \end{align*}
  where the second and fourth equality (resp.\ third equality) follow from the partial trace property (resp.\ cyclicity) of the m-trace.
  Now the invariance by Kirby 0 and 2 moves of
  $F'_\si(\Gamma, \omega)$ follows from that of
  $F_ \si(\Gamma_V,\omega)$.  Finally for a disjoint union, just
  remark that is $\Gamma_V$ is a cutting presentation of $\Gamma$,
  then $\Gamma_V\sqcup\Gamma'$ is a cutting presentation of
  $\Gamma\sqcup\Gamma'$ and
  $F_ \si(\Gamma_V\sqcup\Gamma',\omega'')=F_ \si(\Gamma_V,\omega)F_
  \si(\Gamma',\omega')$.
\end{proof}
Then the modified invariant naturally extends to graph-admissible
triples $(M, \Gamma, \omega)$ with $\Gamma$ a closed bichrome graph
embedded in $M$: 
\begin{theo}[Modified graded Hennings invariant]\label{T:MGHI}
  Let $M$ be an oriented closed $3$-manifold and $(M,\Gamma,\omega)$ be a
  graph-admissible compatible triple with surgery presentation
  $(S^3,\Gamma\cup L,\omega)$.  Define
  \begin{equation}\label{eq:He'}
    \He'( M,\Gamma,\omega ) =\delta^{-s} F'_{\si}(L\cup\Gamma,\omega)
  \end{equation}
  where 
  $\delta= \si_{\overline{0}}(g_{\wb 0}^{-1}\theta_{\overline{0}})$
  and $s$ is the signature of the linking matrix of $L$.  Then $\He'$
  is a well defined invariant of diffeomorphism class of
  $(M,\Gamma,\omega)$.  Furthermore, if $( M',\Gamma',\omega' )$ is
  any closed compatible triple, the modified invariant of the
  connected sum is
  $$\He'( M\# M',\Gamma\sqcup\Gamma',\omega'')=
  \He'(M,\Gamma, \omega)\He(M',\Gamma', \omega')$$ where $\omega''$ is
  the cohomology class induced by $\omega$ and $\omega'$.
\end{theo}
\begin{proof}
From  Theorem
  \ref{T:Kirby} it is enough to show that $\He'$ is invariant under the  Kirby moves.  
 Proposition \ref{P:F'invKirby} imply the Kirby 0 and Kirby II hold.  
 As in the proof of Theorem \ref{T:GHI},
if we add a $\pm1$-framed unknot to
  $L$ with meridian $m_0$, then $\omega(m_0)=\overline{0}$ and the  invariant $F'$ is multiplied by
  $\lambda_{\overline{0}}(\theta_{\overline{0}})^{\mp 1}=\delta^{\mp
    1}$.  So the Kirby I move holds.       
 The last  property of the theorem, follows from considering a surgery
  presentation of the connected sum, which is a disjoint union of surgery presentations in $S^3$.  
\end{proof}

\section{Modified symmetrized integral}\label{modified integral}
We introduce the notion of a modified integral which allows us to relax the
admissibility condition for the modified graded Hennings invariant and in particular gives an invariant for empty 3-manifolds.

Given $\ba_i\in G$, let $\va=(\ba_1,\ldots,\ba_n)$,  $\bb_j=\sum_{i=j}^n\ba_i$ for $j\leq n$ and $\bb=\bb_1$.  
Denote $\U_{\otimes\va}=\U_{\ba_1}\otimes \U_{\ba_2}\otimes ...\otimes \U_{\ba_n} $ and  let $\Delta_\va:\U_{\bb}\to \U_{\otimes\va}$ be the map 
$$\Delta_\va= (1\otimes ...\otimes 1 \otimes \Delta_{\ba_{n-1},\bb_n}) ... (1\otimes \Delta_{\ba_2,\bb_3})\Delta_{\ba_1, \bb_2}.$$
By definition of a Hopf $G$-coalgebra 
$\Delta_\va$ is an algebra morphism.
We denote by $Z_{\va}$ the centralizer of $\Delta_\va(\U_{\bb})$ which
is a subalgebra of $\U_{\otimes\va}$ formed by elements which commute
with any element of $\Delta_\va(\U_{\bb})$.  In particular, for $n=1$
and $\va=(\ba)$ then $Z_{\ba}$ is the center of $\U_{\ba}$.
\begin{Lem} Recall $L_g$ is the left multiplication of $g$.  We have
$$(\si_{\ba}L_{g_{\ba}^{-1}}\otimes\Id)(Z_{(\ba,\bb)})\subset
Z_{\bb}\text{ and }(\Id\otimes\si_{\bb}L_{g_{\bb}})(Z_{(\ba,\bb)})\subset
Z_{\ba}.$$
\end{Lem}
\begin{proof}
  Let $z=\sum z_1\otimes z_2\in Z_{(\ba,\bb)}$, $u\in \U_\bb$ and let
  $\Delta_{(-\ba,\ba,\bb)}(u)=u_{(1)}\otimes u_{(2)}\otimes u_{(3)}$. Then
  $\sum u_{(1)}\otimes z_1u_{(2)}\otimes z_2u_{(3)}=\sum u_{(1)}\otimes
  u_{(2)}z_1\otimes u_{(3)}z_2$ by definition of $Z_{(\ba,\bb)}$ so
  $$\sum \si_{\ba}(g_{\ba}^{-1}z_1)z_2u=\sum  \si_{\ba}(g_{\ba}^{-1}z_1u_{(2)}S_{-\ba}^{-1}(u_{(1)}))z_2u_{(3)}$$
  $$=\sum  \si_{\ba}(g_{\ba}^{-1}S_{-\ba}(u_{(1)})z_1u_{(2)})\otimes z_2u_{(3)}=$$
  $$\sum  \si_{\ba}({g_{\ba}^{-1}}S_{-\ba}(u_{(1)})u_{(2)}z_1)\otimes u_{(3)}z_2=\sum \si_{\ba}(g_{\ba}^{-1}z_1)uz_2.$$ The proof is similar for the second inclusion.
\end{proof}
\newcommand{\Gss}{{G\setminus X}}
\begin{Def} \label{D:MI}
Let $X$ be a subset of $G$.
We say that a family of maps
  $\{\si'_{\ba}:Z_{\ba}\to\C\}_{\ba\in
    \Gss}$ is a \emph{modified integral} on $\Gss$ if for any $\ba,\bb\in \Gss$
  the following two linear forms are equal on $Z_{(\ba,\bb)}$:
  $$\si_{\ba}L_{g_{\ba}^{-1}}\otimes\si'_{\bb}
  =\si'_{\ba}\otimes \si_{\bb}L_{g_{\bb}}.$$

\end{Def}

\begin{theo}\label{T:Wow}
  Let $X$ be the subset of $G$ such that $\Gss$ is the set of $\ba$ where 
  $\U_{\ba}$ is semi-simple.  Then there exists a family
  of central elements
  $\{z_{\ba}\in
  Z_{\ba}\}_{\ba\in\Gss}$ such that   $\si_{\ba}(x)=\tr^\C_{\U_{\ba}}(L_{z_{\ba}}L_x)$ for all
  $x\in \U_{\ba}$. 
  Furthermore, there exists a modified integral on $\Gss$ defined by
  \begin{equation}
    \label{eq:mint}
    \si_{\ba}'(z):=\si_{\ba}({z_{\ba}}z)=\tr^\C_{\U_{\ba}}(L_{{z_{\ba}^2}z})
  \end{equation}
for all $z\in Z_{\ba}$.
\end{theo}
\begin{proof}
  Since $\U_{\ba}$ is a finite-dimensional semi-simple
  algebra for $\ba\in \Gss$ then it is isomorphic to a product of matrix algebras:
  $\U_{\ba}\cong\bigoplus_i\operatorname{Mat}(V_i)$ where
  $V_i=\C^{n_i}$ are its irreducible representations.  Then the
  identity matrices of each summand form a basis $\{z_i\}_i$ of the
  center $Z_{\wb \alpha}$ and the characters
  $x\mapsto\tr^\C_{\U_{\ba}}(L_{z_i}L_x)=n_i\tr_{V_i}^\C(\rho_{V_i}(x))$
  form a basis of $\text{H\!H}_{0}(\U_{\ba})^*:=(\U_{\ba}/[\U_{\ba},\U_{\ba}])^*$.
  Finally since
  $\si_{\ba}\in \text{H\!H}_{0}(\U_{\ba})^*$ there exists
  $\{\delta_i\in\C\}_i$ such that
  $\si_{\ba}(x)=\sum_i\delta_i\tr^\C_{\U_{\ba}}(L_{z_i}L_x)$ for all $
  x\in\U_{\ba}$.
We define $z_{\ba}=\sum_i\delta_iz_i$.

  Remark that since $z_iz_j=\delta_i^jz_i$, applying the above formula
  for $x=z_i$ gives
  $\si_{\ba}(z_i)=\tr^\C_{\U_{\ba}}(L_{z_i})\delta_i=n_i^2\delta_i$
  because $L_{z_i}$ is the identity of the block
  $\operatorname{Mat}(V_i)$. But the symmetrized integral of $z_i$ is
  the modified trace of the induced endomorphism of $\U_{\ba}$ (by
  left or right multiplication) which is a projector on
  $V_i^{\oplus n_i}$ then $\si_{\ba}(z_i)=n_i d_i$ where $d_i$  is
  the modified dimension of $V_i$ (see \cite{Ha18}), thus $\delta_i=\frac{d_i}{n_i}$.  Finally, one gets
  \begin{equation}
    \label{eq:int=kirby}
     \si_{\ba}=\sum_id_i\tr_{V_i}^\C\circ\rho_{V_i}.
  \end{equation}
  Let us now prove that
  $\si_{\ba}'=\si_{\ba}({z_{\ba}}\
  \cdot)$ is a modified integral.  Let
  $x\in Z_{\ba,\bb}$ and consider the
  endomorphism $f$ of $\U_{\ba}\otimes \U_{\bb}$ given by left
  multiplication by
  $(z_{\ba}\otimes z_{\bb})x$.  Then the
  properties of the partial trace imply that
  \begin{equation}
    \label{eq:ptr}
    \tv_{\U_{\ba}}(\ptr_{\U_{\bb}}(f))=\tv_{\U_{\bb}}(\ptr_{\U_{\ba}}(f))
  \end{equation}
  now the partial trace in the category is the partial trace in Vect
  twisted by the action of the pivotal elements and the modified trace
  on $\U_{\ba}$ is given by
  $\tv_{\U_{\ba}}=\si_{\ba}=\tr^\C_{\U_{\ba}}(L_{z_{\ba}}\
  \cdot)$ thus Equation \eqref{eq:ptr} becomes
  $$\tr^\C_{\U_{\ba}\otimes \U_{\bb}}(L_{z_{\ba}^2\otimes g_{\bb}z_{\bb}}L_x)=\tr^\C_{\U_{\ba}\otimes \U_{\bb}}(L_{z_{\ba}g_{\ba}^{-1}\otimes z_{\bb}^2}L_x)$$
  with
  $\tr^\C_{\U_{\ba}\otimes \U_{\bb}}=\tr^\C_{\U_{\ba}}\otimes
  \tr^\C_{\U_{\bb}}$ so
  $\si'_\ba\otimes\si_\bb((1\otimes
  g_{\bb})x)=\si_\ba\otimes\si'_\bb((g_{\ba}^{-1}\otimes1)x)$ which
  proves that $\si'$ is a modified integral.
\end{proof}
Remark that the modified integral $\si_\ba'$ of Theorem \ref{T:Wow} has a
natural extension as a linear form on the whole algebra $\U_\ba$.  We
call $\si_\ba'$ the \emph{canonical} modified integral.

Recall the definition of a graph-admissible triple given in Definition~\ref{D:graph-admissible triple}.
\begin{Def}
  A compatible triple $(M,\Gamma,\omega)$ is called a \emph{$G$-admissible
  triple} if there exists a red edge of $\Gamma$ colored by
  ${\ba}\in G\setminus X$. A triple $(M,\Gamma,\omega)$ is called an
  \emph{admissible triple} if $(M,\Gamma,\omega)$ is a $G$-admissible triple or
  a graph-admissible triple.
\end{Def}
As for graph-admissible closed bichrome graph, we can introduce the
notion of cutting presentation $\Gamma_{\ba}$ of $G$-admissible triple
$(S^3,\Gamma,\omega)$ by cutting a red edge rather than a blue edge:
$\Gamma_{\ba}$ is a {$1$-string link graph} whose braid closure is
$\Gamma$ and where $\omega$ take on the meridian of the open strand
the value $\ba\in G\setminus X$.  Remark that the universal invariant
of such a graph is $\ell$-periodic in its $r$-variables and equivariant
so $F_\si(\Gamma_{\ba})\in Z_\ba$.
\begin{Prop}
  There exists unique extensions of $F'_\si$ to $G$-admissible
  bichrome closed graphs (resp. $\He'$ to $G$-admissible triples)
 ) 
 with
  $$F'_\si(\Gamma,\omega)=\si_{\ba}'(F_\si(\Gamma_{\ba},\omega))$$
  where $\Gamma_{\ba}$ is a cutting presentation of $\Gamma$.  The
  extension of $\He'$ is given by 
   \begin{equation*}
    \He'( M,\Gamma,\omega ) =\delta^{-s} F'_{\si}(L\cup\Gamma,\omega)
  \end{equation*}
  where  $(S^3,\Gamma\cup L,\omega)$ is a surgery presentation of an admissible triple  $(M,\Gamma,\omega)$, the constant $\delta$ is $ \si_{\overline{0}}(g_{\wb 0}^{-1}\theta_{\overline{0}})$
  and $s$ is the signature of the linking matrix of $L$.
\end{Prop}
\begin{proof}
  Suppose that $\Gamma$ admits two $G$-admissible cutting
  presentations $\Gamma_{\ba}$ and $\Gamma_{\bb}$.  Then cutting the
  two edges we get a {$2$-string link graph} whose image by $F_\si$ is
  in $Z_{(\alpha,\beta)}$ because of Proposition \ref{P:Fmu-equiv}. Then
  Definition \ref{D:MI} implies that the invariant computed by
  $\Gamma_{\ba}$ and $\Gamma_{\bb}$ are the same.

  Suppose now that $\Gamma$ admits a $G$-admissible cutting
  presentation $\Gamma_{\ba}$ obtained by cutting a red edge and a
  graph-admissible cutting presentation $\Gamma_V$ for some projective
  $V\in\cat^H_\bb$ obtained by cutting a blue edge. Then cutting the
  two edges lead to a graph $T$ whose image by $F_\si$ is in
  $\U_\ba\otimes\End_\C(V)$ (see Figure 
  \ref{Cutting at blue edge and at red edge}).  Let us write
  $F_\si(T,\omega)=L_x\otimes f$ (we omit the sum).
\begin{figure}
$$
\Gamma_V\quad =\quad \epsh{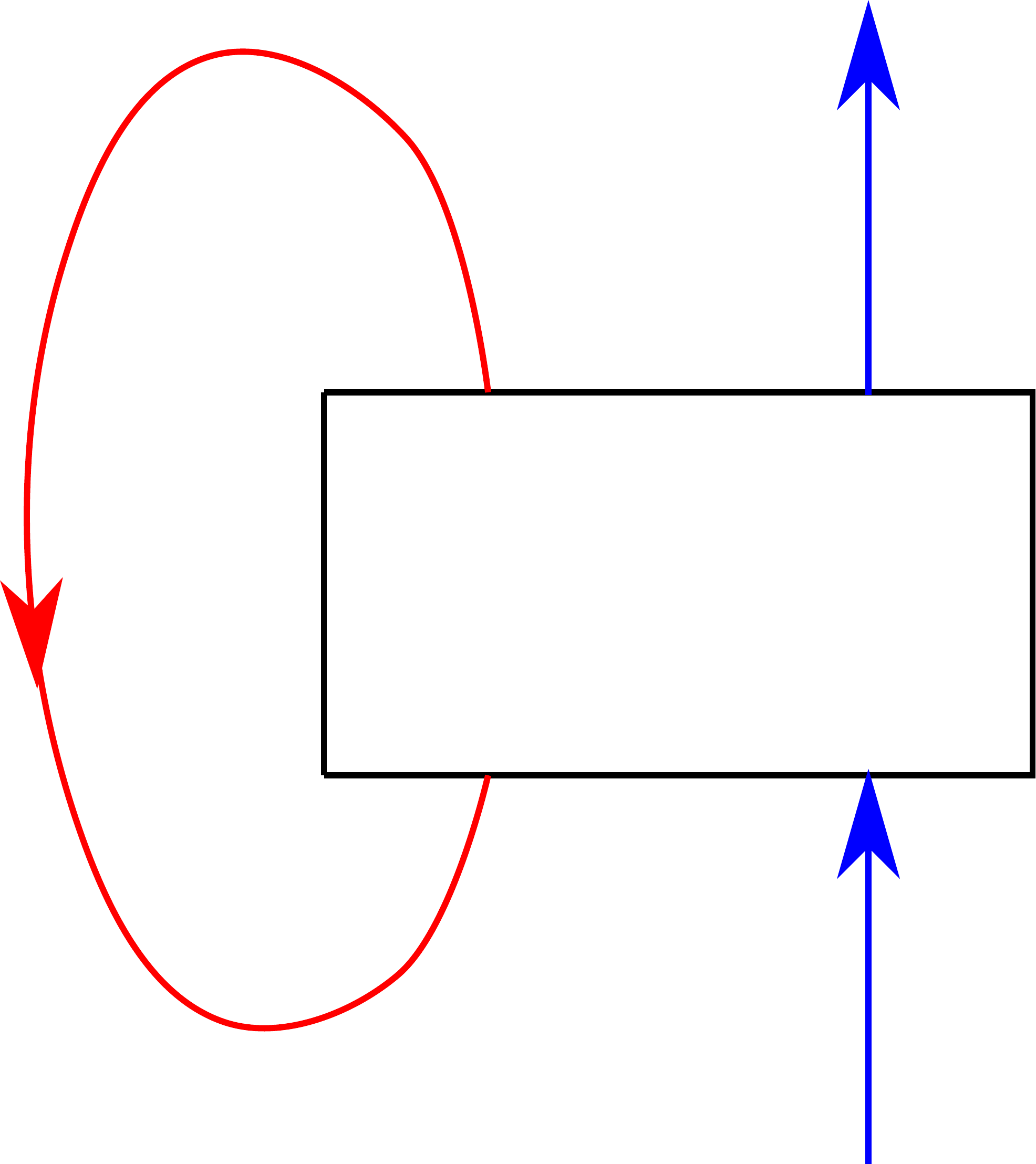}{12ex}\put(-32,1){\ms{\ \ \ \ T}}\quad ,\qquad \Gamma_\ba\quad = \quad\epsh{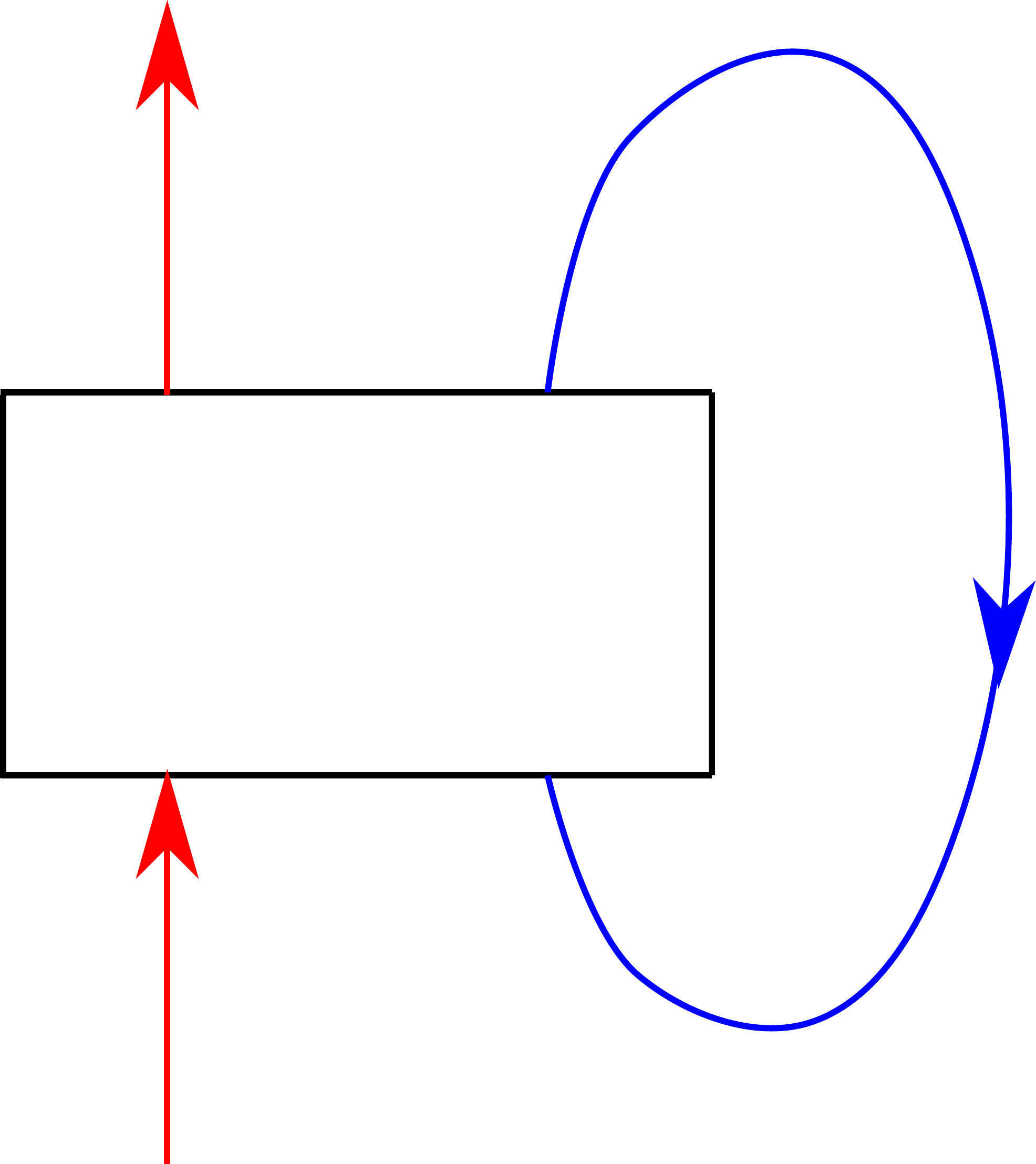}{12ex}\put(-50,1){\ms{\ \ \ \ T}}
$$
\caption{Cutting presentation $\Gamma_V$ at blue edge and $\Gamma_{\ba}$ at red edge}
\label{Cutting at blue edge and at red edge}
\end{figure}
Then
\begin{align*}
\si'_{\wb\alpha}(F_\si(\Gamma_\ba,\omega))&= \si'_{\wb\alpha}\bp{\epsh{right_morphism_T}{10ex}\put(-42,1){\ms{L_x\otimes f}}}=\si'_{\wb\alpha}(x)\tr_{\U_{\wb\beta}}^{\C}\bp{\rho_{V}(g_{\wb\beta})f}\\
&=\si_{\wb\alpha}(z_{\wb\alpha}x)\tr_R^{\cat^H}(f)=\tv_{\U_{\wb\alpha}}\bp{L_{z_{\wb\alpha}}L_x}\tr_{\U_{\wb\beta}}^{\C}\bp{\rho_{V}(g_{\wb\beta})f}\\
&=F'_\si\bp{\epsh{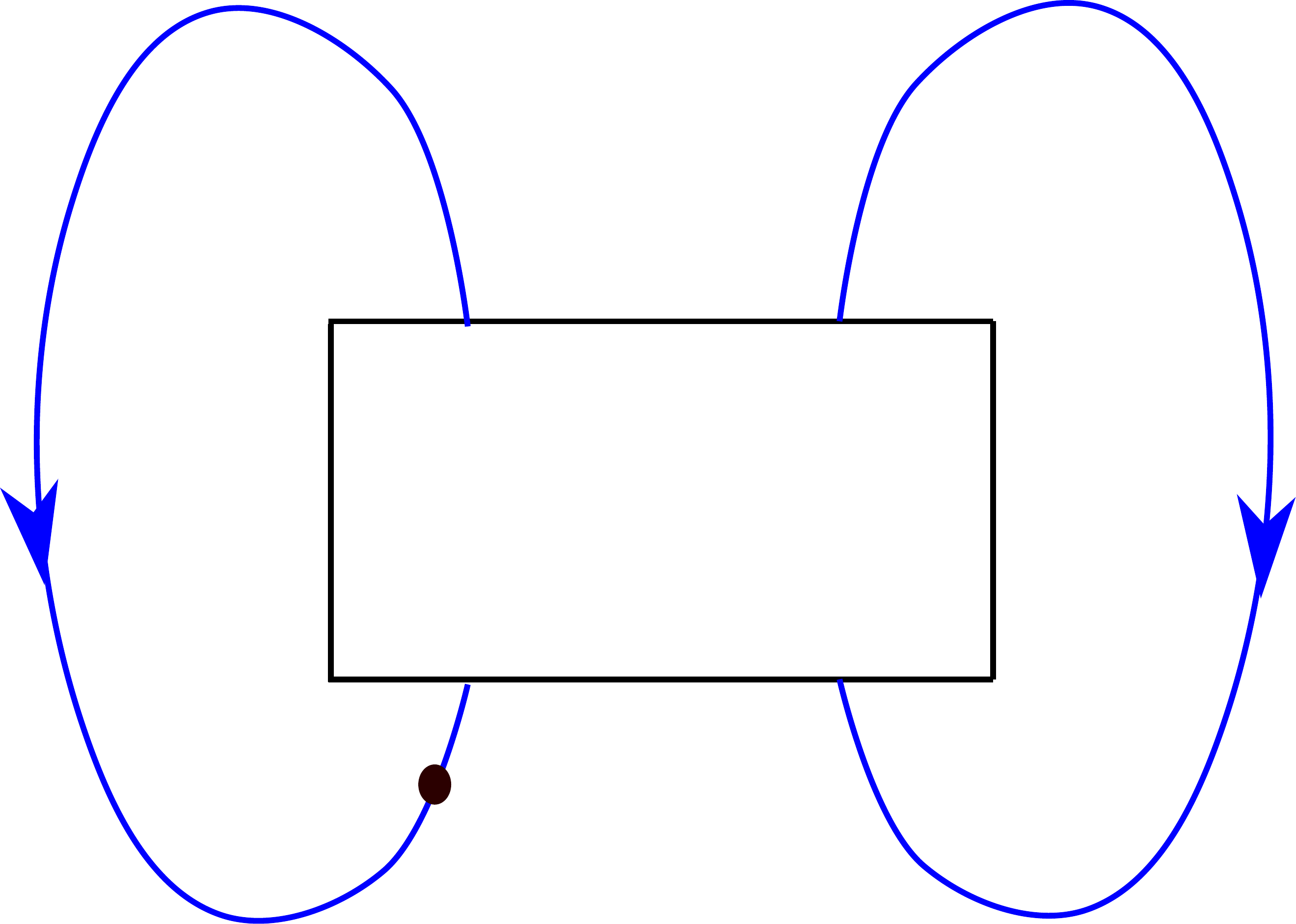}{10ex}\put(-47,-1){\ms{L_x\otimes f}}\put(-45,-18){\ms{z_{\wb\alpha}}}}=\tv_{V}\bp{\epsh{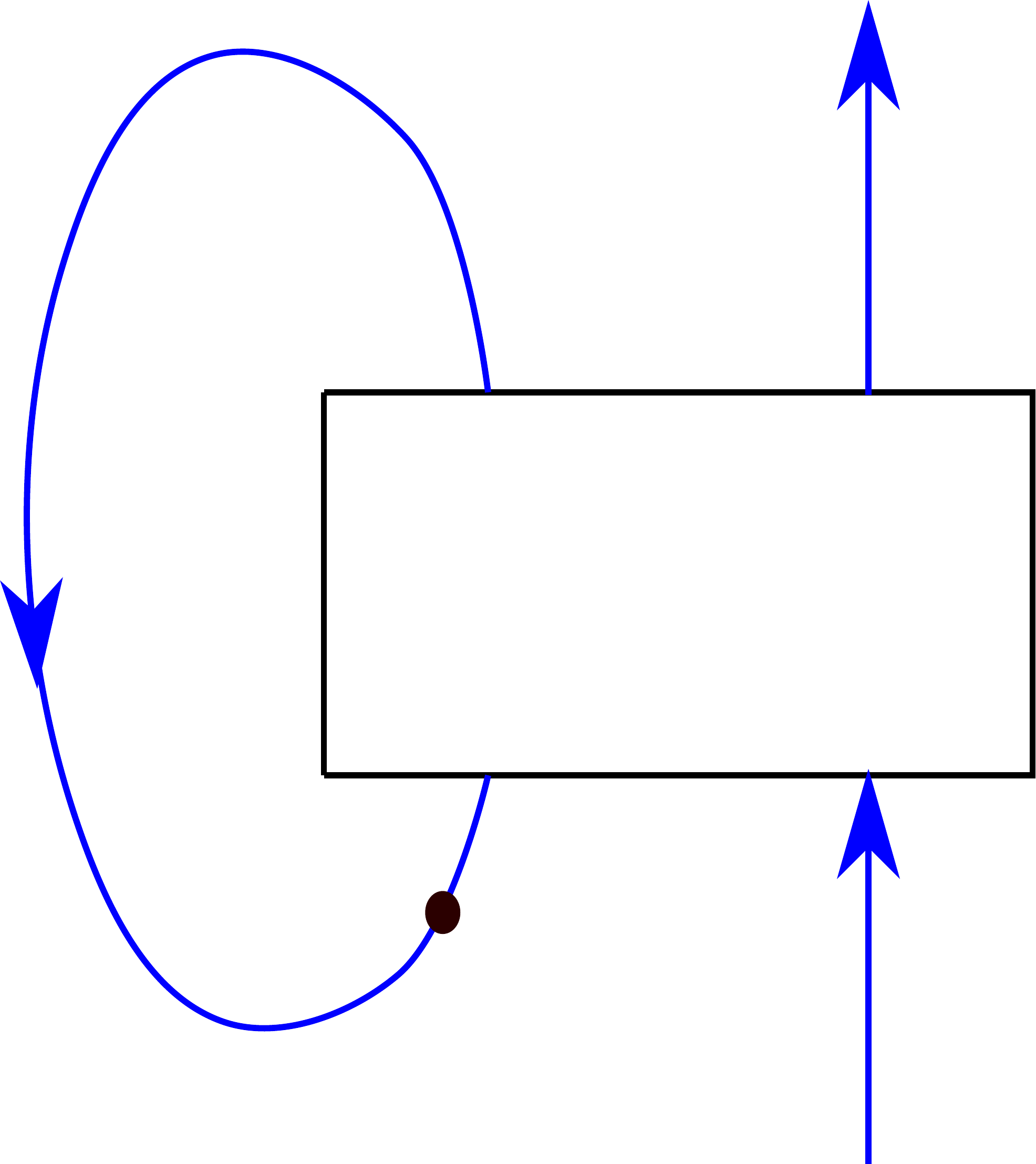}{12ex}\put(-33,1){\ms{L_x\otimes f}}\put(-28,-18){\ms{z_{\wb\alpha}}}}\\
&= \tr_{\U_{\wb\alpha}}^{\C}\bp{L_{z_{\wb\alpha}}L_{g_{\wb\alpha}^{-1}}L_x}\tv_V(f)=\si_{\wb\alpha}(xg_{\wb\alpha}^{-1})\tv_V(f)\\
                        &=\tv_V\bp{\epsh{left_morphism_T}{10ex}\put(-29,1){\ms{L_x\otimes f}}}=\tv_V\bp{F_ \si(\Gamma_V,\omega)}
\end{align*}
where the third equality follows because $x\tr_R^{\cat^H}(f)\in Z_\ba$
is central and so
$\tr_R^{\cat^H}(f)L_{z_\ba x}=\tr_R^{\cat^H}(f)R_{z_\ba x}$ has its
modified trace given by Equation~\eqref{eq:t}.  
\end{proof}

\section{Relations with other non semi-simple invariants}
\label{s:rel}
We compare the invariant of this paper with three previously defined
invariants: (1) the one from the second author with
$\mathfrak{sl}(2|1)$, (2) with the invariant defined by the first and
last author with F. Costantino using the weight representations of the
unrolled quantum groups and (3) the modified Hennings invariants.
\subsection{The $\mathfrak{sl}(2|1)$ case}
Let $\ell\ge3$ be an odd integer.
In \cite{Ha18a} the second author considers the semi-restricted super algebra $\U_\xi\mathfrak{sl}(2|1) $ and unrolled super algebra $\UH_\xi\mathfrak{sl}(2|1)$ associated with the super Lie algebra $\mathfrak{sl}(2|1)$.
It is shown that $\wh {\UH_\xi}\mathfrak{sl}(2|1)$ is a topological
ribbon Hopf super algebra and its bosonization
$\wh {\UH_\xi}\mathfrak{sl}(2|1)^\sigma=\wh
{\UH_\xi}\mathfrak{sl}(2|1)\rtimes \Z/2\Z$ is a topological ribbon
Hopf algebra.  Consider the bosonization
$\U_\xi\mathfrak{sl}(2|1)^\sigma$ of the semi-restricted super
quantum group.
This algebra can be realized as the  $\U$ of this paper where $\Lambda$ is the root lattice of  $\mathfrak{sl}(2|1)$ and  $W$ is the $\Lambda$-graded vector space generated by
the ``E" and ``F" parts 
of the Poincaré–Birkhoff–Witt basis.
 It follows from \cite{Ha18a} that $\U$ satisfies the
five Axioms \ref{pivotal axiom}--\ref{A:proj}.  Moreover, one can identify
$\wh {\UH_\xi}\mathfrak{sl}(2|1)^\sigma$ with the unrolled version
$\wh\UH$ of $\U$ as defined in Section \ref{topo quntum and uni
  invariant}.

The construction of \cite{Ha18a} was the motivation for this paper. 
 It produces
an invariant $\mathcal{J}(M,\omega)$ which is equal to the graded
Hennings invariant of this paper:
$$\mathcal{J}(M,\omega)=\He(M,\emptyset,\omega).$$
In addition, a special property for $\mathfrak{sl}(2|1)$ is that
$\delta=1$ (see \cite[Lemma 4.13]{Ha18a}).

\subsection{Comparison with invariants from nilpotent weight modules
  of unrolled quantum group}\label{ss:mgH=CGP}
In \cite{FcNgBp14} an invariant of a compatible triple $(M,\Gamma,\omega)$
where $\Gamma$ is a 
blue $\cat^H$-ribbon graph embedded in
$M$ was defined from the data of a relative premodular\footnote{In \cite{FcNgBp14} the term 
  relative modular is used instead of relative premodular, later in \cite{MDR17} it was shown that an additional requirement is need to define a TQFT and so the use of term was changed.}
  $G$-category $\cat^H$
with translation group $Z$.  Many examples of unrolled quantum groups
lead to such categories and in particular quantum groups of Lie
algebras at suitable root of unity.
Let us discuss how relative premodular categories fit into this paper.  

Let $\Lambda$ be a lattice and $W$ a $\Lambda$-graded vector space
which give rise to an algebra $\U$ satisfying Axioms \ref{pivotal
  axiom} and \ref{quasi axiom}.  One of the requirements of a relative
premodular category is the existence of a translation group.  The
lattice naturally gives a translation group, in the
category of weight modules,
which is the family of one dimensional modules
$(\C_\lambda)_{\lambda\in\Lambda}$ where $w\in W$ acts by
multiplication by $\ve(w)$ and $H_i$ by $\ell\lambda(H_i)$.
Here our group $G$ is $\H^*/\Lambda$.
Another requirement of relative premodular
$G$-category is the (generic) semisimplicity of the $\cat^H_g$ for
$g\in G\setminus X$ with $X$ being ``small.''  For the rest of the
subsection, we assume that $\U$ satisfies Axioms \ref{pivotal
  axiom}--\ref{A:proj} and that $\cat^H$ is a relative premodular
$G$-category.  We have proved this is the case for all semi-restricted
quantum groups associated to simple Lie algebras.

Let $\ba\in G\setminus X$ so that $\U_\ba$ is semisimple.
Let $\Theta_\ba$ be a set of simple
$\UH_\ba$-module such that $\{\ResF(V_i):V_i\in \Theta_\ba\}$ is a set
of representant of the isomorphism classes of simple $\U_\ba$-modules.
Then Equation \eqref{eq:int=kirby} implies that the character given by
the Kirby color
$x\mapsto\sum_{V_i\in \Theta_\ba}d_i\tr_\C(\rho_{V_i}(x))$ is equal on
$\U_\ba$ to the symmetrized integral $\si_\ba$.  Hence we get
\begin{Prop}\label{P:H'=CGP}
  Let $\Gamma$ be a monochrome blue $\cat^H$-ribbon graph and
  $(M,\Gamma,\omega)$ a compatible triple.  The invariant 
  $\mathsf N_{\cat^H}(M,\Gamma,\omega)$, defined in  \cite{FcNgBp14}, coming from the relative
  premodular category $\cat^H$  is equal to the
  modified Hennings invariant $\He'(M,\Gamma,\omega)$.
\end{Prop}
\begin{proof}
  As in \cite{FcNgBp14}, choose
  a computable presentation of
  $(M,\Gamma,\omega)$ then one gets a diagram where
  the surgery components are colored by Kirby colors of degree
  $\ba\notin X$.  From Proposition \ref{P:univ-red}, the
  invariant of this diagram is equal to the invariant $F'_\si$ of the
  diagram obtained by replacing the Kirby-colored component with red
  component.  But this is precisely the process for computing $\He'$.
\end{proof}
\subsection{Comparison with the modified Hennings invariants}
Let $\U$ be an algebra as in Section \ref{Section unrolled qt group} satisfying Axioms \ref{pivotal axiom}--\ref{A:proj} and $\UH$ its unrolled version.
Recall that the quotient $\UH_{\wb0}$ of $\UH$ contains a finite
dimensional ribbon Hopf algebra $\U_{\wb0}$ with the same R-matrix
$\RR_{\wb0}\in\mathcal Q_2{\U}^{\otimes2}/I_{(\wb0,\wb0)}\simeq
\U_{\wb0}\otimes \U_{\wb0}$. There is an obvious associated ribbon
functor $\cat^H_{\wb0}\to\U_{\wb0}$-Mod. The construction of
\cite{DGP} produces a modified Hennings invariant since Axiom
\ref{Ax:non-degenerate} ensure that $\U_{\wb0}$ is a non-degenerate
finite dimensional Hopf algebra.
\begin{theo}
  Let $(M,\Gamma,\omega)$ be compatible triple with $\omega=0$.  Then
   the invariant
  $\He'(M,\Gamma,\omega)$
 of this paper, associated to $\U$,
  is equal to the (non graded) modified
  Hennings invariant of \cite{DGP} associated to the finite
  dimensional ribbon Hopf algebra $\U_{\wb0}$.
\end{theo}
\begin{proof}
  If $\omega=0$, then the definition of $\He'$ only uses the integral
  $\si_{\wb0}=\lambda_{\wb0}(g\cdot)$ and the computation of the
  universal invariant can be done directly in $\U_{\wb0}$ as it is
  done in \cite{DGP}.
\end{proof}
Remark that combining this theorem with Proposition \ref{P:H'=CGP}
gives a new proof that the invariants of \cite{FcNgBp14} and \cite{DGP} coincide for
zero cohomology classes, as was first shown in \cite{DGP2}.
\bibliographystyle{plain}
\bibliography{Reference}

\end{document}